\documentclass[a4paper,oneside,reqno]{amsart}

\usepackage[english]{babel}
\usepackage[utf8]{inputenc}			% Euler for math | Palatino for rm | Helvetica for ss | Courier for tt
\usepackage[scaled]{helvet}			% ss
\usepackage{courier}				% tt
\usepackage{eulervm}				% a better implementation of the euler package (not in gwTeX)
\usepackage[bb=boondox,bbscaled=1.05,scr=dutchcal]{mathalfa}	%
\usepackage{dsfont}
\normalfont

\usepackage{tabularx,longtable,booktabs}			% tables
\usepackage{caption,subcaption}								% captions
\usepackage{fullpage}				    							% no margins
\usepackage[english]{varioref}								% references
\usepackage[dvipsnames]{xcolor}								% colors
\usepackage{bookmark}													% bookmarks
\usepackage{enumitem}													% list environments:

\setlist[itemize,1]{leftmargin=1.8em}
\setlist[enumerate,1]{leftmargin=1.8em}

\usepackage{textcomp}															% symbols
\usepackage{hyperref}													% hyper links
\definecolor{webbrown}{rgb}{0.65, 0.16, 0.16}	% red color for links
\hypersetup{
	colorlinks=true,%
	linktocpage=true,%
	pdfstartpage=1,%
	pdfstartview=FitV,breaklinks=true,%
	pdfpagemode=UseNone,%
	pageanchor=true,%
	pdfpagemode=UseOutlines,plainpages=false,%
	bookmarksnumbered,%
	bookmarksopen=true,%
	bookmarksopenlevel=1,hypertexnames=true,%
	pdfhighlight=/O,urlcolor=webbrown,%
	linkcolor=RoyalBlue,%
	citecolor=ForestGreen
}

\usepackage{graphicx}		% images
\usepackage{float}			% floating objects
\usepackage{tikz}				% TikZ images
\usepackage{tikz-cd}		% commutative diagrams
\usetikzlibrary{
	arrows,%
	decorations.pathreplacing,%
	decorations.markings,%
	shapes.geometric,%
	positioning,%
	calc,%
	fadings,%
	angles
	}
\tikzfading[name=inner fade, inner color=transparent!0, outer color=transparent!100]

\usepackage[
	style=alphabetic,%
	maxalphanames=4,%
	giveninits,%
	sorting=nyt,%
	maxbibnames=99,%
	backend=biber
]{biblatex}
\renewbibmacro{in:}{}
\usepackage{csquotes}

\usepackage{amsmath,amssymb,amsthm,mathtools}	% basic math commands
\usepackage{bm,braket,faktor,stmaryrd}			% fancy math commands
\numberwithin{equation}{section}				% number equations with sections

\newcommand{\de}{\partial}

\newcommand{\bbraket}[1]{\llbracket #1 \rrbracket}

\DeclarePairedDelimiter\floor{\lfloor}{\rfloor}
\newcommand\norder[1]{{\vcentcolon}\!\mathrel{#1}\!{\vcentcolon}}

\newcommand{\Z}{\mathbb{Z}}
\newcommand{\Q}{\mathbb{Q}}

\newcommand{\C}{\mathbb{C}}
\renewcommand{\P}{\mathbb{P}}

\newcommand{\iu}{\mathrm{i}}
\newcommand{\Id}{\mathrm{Id}}
\AtBeginDocument{\renewcommand{\hbar}{\hslash}}

\newcommand{\Ai}{\mathrm{Ai}}
\newcommand{\Bi}{\mathrm{Bi}}
\newcommand{\Gi}{\mathrm{Gi}}
\newcommand{\Hi}{\mathrm{Hi}}

\newcommand{\mc}[1]{\mathcal{#1}}
\newcommand{\ms}[1]{\mathsf{#1}}

\newcommand{\Mbar}{\overline{\mathcal{M}}}

\DeclareMathOperator*{\Res}{Res}

\DeclareMathOperator{\Tr}{\mathrm{Tr}}
\DeclareMathOperator{\End}{End}
\newcommand{\Aut}{\mathord{\mathrm{Aut}}}

\usepackage{thmtools}
\usepackage[noabbrev]{cleveref}

\theoremstyle{plain}
\newtheorem{theorem}{Theorem}[section]
\newtheorem{proposition}[theorem]{Proposition}
\newtheorem{conjecture}[theorem]{Conjecture}
\newtheorem{lemma}[theorem]{Lemma}
\newtheorem{corollary}[theorem]{Corollary}

\theoremstyle{definition}
\newtheorem{definition}[theorem]{Definition}
\newtheorem{remark}[theorem]{Remark}
\newtheorem{example}[theorem]{Example}

\theoremstyle{plain}
\newtheorem{introthm}{Theorem}
\newtheorem{introcor}[introthm]{Corollary}
\newtheorem{introconj}[introthm]{Conjecture}

\let\oldtocsubsection=\tocsubsection
\renewcommand{\tocsubsection}[2]{\hspace{1em}\oldtocsubsection{#1}{#2}}

\crefname{lemma}{lemma}{lemmata}
\Crefname{lemma}{Lemma}{Lemmata}
\crefname{subsection}{subsection}{subsections}
\Crefname{subsection}{Subsection}{Subsections}
\crefname{conjecture}{conjecture}{conjectures}
\Crefname{conjecture}{Conjecture}{Conjectures}

\title{Relations on $\overline{\mathcal{M}}_{g,n}$ and the negative $r$-spin Witten conjecture}
\author[N. K. Chidambaram]{Nitin Kumar Chidambaram}
\address[N. K. Chidambaram]{
	Max Planck Institut f\"ur Mathematik, Vivatsgasse 7, 53111 Bonn, Germany %
}

\curraddr{School of Mathematics, University of Edinburgh, James Clerk Maxwell Building, Peter Guthrie Tait Rd, Edinburgh EH9 3FD, U.K. %
}
\email{nitin.chidambaram@ed.ac.uk}

\author[E. Garcia-Failde]{Elba Garcia-Failde}
\address[E. Garcia-Failde]{
	Sorbonne Universit\'e, UMR 7586 CNRS, Institut de Math\'ematiques de Jussieu--Paris Rive Gauche, 75252 Paris, France %
}
\email{elba.garcia-failde@imj-prg.fr}

\author[A. Giacchetto]{Alessandro Giacchetto}
\address[A. Giacchetto]{
	Universit\'e Paris-Saclay, UMR 3681 CNRS, CEA, Institut de Physique Th\'eorique, 91191 Gif-sur-Yvette, France %
}
\curraddr{
		Departement Mathematik, ETH Zürich, Rämisstrasse 101, Zürich 8044, Switzerland
}
\email{alessandro.giacchetto@math.ethz.ch}
\date{}
\subjclass[2020]{Primary 14H10, 14H70; Secondary 37K20, 81R12}

\begin{document}

\begin{abstract}
	We construct and study various properties of a negative spin version of the Witten $ r $-spin class. By taking the top Chern class of a certain vector bundle on the moduli space of spin curves that parametrises $ r $-th roots of the anticanonical bundle, we construct a non-semisimple cohomological field theory (CohFT) that we call the Theta class $ \Theta^r $. This CohFT does not have a flat unit and its associated Dubrovin--Frobenius manifold is nowhere semisimple. Despite this, we construct a semisimple deformation of the Theta class, and  using the Teleman reconstruction theorem, we obtain tautological relations on $ \Mbar_{g,n} $. Furthermore, we prove that the descendant potential of the Theta class is the unique solution to a set of $ \mathcal{W} $-algebra constraints, which implies a recursive formula for all the descendant integrals. Using this result for $ r = 2 $, we prove Norbury's conjecture which states that the descendant potential of $ \Theta^2 $ coincides with the Brézin--Gross--Witten tau function of the KdV hierarchy. Furthermore, we conjecture that the descendant potential of $ \Theta^r $ is the $ r $-BGW tau function of the $ r $-KdV hierarchy and prove the conjecture for $ r = 3 $. 
\end{abstract}

\maketitle
\thispagestyle{empty}

\tableofcontents

%–––––––––––––––––––––––––––––––––––––––––––%
\section{Introduction and results}
%–––––––––––––––––––––––––––––––––––––––––––%

One of the earliest generalisations of the famous Witten--Kontsevich theorem \cite{Wit90,Kon92} is the so-called Witten's $ r $-spin conjecture \cite{Wit93}, which states that the generating function of integrals of the Witten $ r $-spin class coupled with $ \psi $-classes on $ \Mbar_{g,n} $ is a tau function for the $ r $-KdV integrable hierarchy. At the time, the Witten $ r $-spin class was not defined. Witten sketched a construction in genus zero using the moduli space of $ r $-spin curves, which is the moduli space parametrising $ r $-th roots of the canonical bundle. The first mathematical construction appeared many years later due to Polischuk and Vaintrob \cite{PV01} and turned out to be remarkably intricate. Proving the Witten $ r $-spin conjecture required the joint effort of mathematicians in different fields \cite{AvM92,Giv01,FSZ10}.  

The Witten $ r $-spin class also encodes a lot of information concerning the cohomology ring of the moduli space of curves $ \Mbar_{g,n} $. By analysing the Dubrovin--Frobenius manifold associated to the Witten $ r $-spin class, Pandharipande, Pixton and Zvonkine \cite{PPZ15} proved vanishing relations among tautological classes in the cohomology ring of $ \Mbar_{g,n} $, known as Pixton's relations \cite{Pix13}, which explain all presently known relations in the tautological ring. 

\addtocontents{toc}{\protect\setcounter{tocdepth}{1}}
%–––––––––––––––––––––––––––––––––––––––––––%
\subsection*{The Theta class}
%–––––––––––––––––––––––––––––––––––––––––––%
In this paper, we are interested in a version of the Witten $ r $-spin class  for ``negative" spin. More precisely, the geometric space of interest to us is the moduli space of $ r $-th roots of the \emph{anticanonical} bundle. For an integer $ r \geq 2 $, and integers $ a_i \geq 0 $ (called primary fields), we consider the moduli space of spin curves $ \Mbar_{g;a}^{r,-1} $ which parametrizes the data of a stable curve $ (C, x_1,\ldots, x_n) $ and a line bundle $ L $ on $ C $ such that 
\[
	L^{\otimes r} \cong \omega_{\log}^{-1} \left(-\sum_{i=1}^n a_i x_i \right),
\]
where $ \omega_{\log} $ is the log canonical bundle of $ C $. Following Chiodo \cite{Chi08+}, we take the derived pushforward of the universal line bundle associated to $ L $ from the universal curve to $ \Mbar_{g;a}^{r,-1} $, to obtain a vector bundle $ \mathcal{V}^{r,-1}_{g;a} $ (defined precisely in \cref{eqn:V:vb}). We consider the top Chern class of $ \mathcal{V}^{r,-1}_{g;a} $, push it forward along the forgetful map to $ \Mbar_{g,n} $, and rescale it by a factor to obtain the \textit{Theta class} $ \Theta^r_{g,n} $ that lives in $ H^{\bullet}(\Mbar_{g,n}) $. The case $r=2$ was considered by Norbury in \cite{Nor23}. A detailed description of the Theta class is in \cref{sec:thetadef}.

For every $ (g,n) $, the cohomology class $ \Theta^r_{g,n} $ depends on the primary fields, and hence, we can view it as a collection of maps
\[
	\Theta^r_{g,n} \colon  V^{\otimes n} \longrightarrow H^{\bullet}(\Mbar_{g,n})
\]
from the $n$-th tensor power of the vector space $V = \Q \braket{ v_1, v_2, \ldots, v_{r-1} }$. The evaluation at $ v_a $ corresponds to the Theta class with primary fields $a$. For different $ (g,n) $, these collections of maps satisfy various properties respecting the structure of $ \Mbar_{g,n} $ and the natural morphisms between them.  A convenient notion to keep track of these properties is the language of cohomological field theories (CohFTs for short) introduced by Kontsevich and Manin \cite{KM94}. Notice that the vector associated to the primary field $a = 0$ is excluded from $V$. This is fundamental for the CohFT properties to hold (see \cref{rem:v0} for more details).

The CohFT structure equips the vector space $ V $ with a product known as the quantum product. The quantum product turns $ V $ into a commutative associative algebra, and when this algebra is semisimple, the CohFT is called semisimple. Semisimple CohFTs form a very special class, as they are completely understood by results of Teleman \cite{Tel12}. Non-semisimple CohFTs, on the other hand, are rather poorly understood. Our first main result, proved in \cref{sec:deformedTheta}, is that the Theta class $ \Theta^r_{g,n} $ satisfies the axioms of a CohFT, but is not semisimple. Norbury \cite{Nor23} proved the following result in the special case $ r =2 $, and our theorem extends it to all $ r \geq 2 $.

\begin{introthm}[CohFT properties]
	The Theta class $( \Theta^r_{g,n} )_{2g-2+n > 0}$ is a non-semisimple CohFT of rank $(r-1)$ on $ (V,\eta) $, with the non-degenerate pairing defined as
	\[
		\eta \colon V \times V \longrightarrow \Q\,,
		\qquad\qquad
		\eta(v_a,v_b) = \delta_{a+b,r}\,.
	\]
	The CohFT does not admit a flat unit. On the other hand, it admits a modified unit $ v_{r-1} $, that is
	\[
		\Theta^{r}_{g,n+1}(v_{a_1} \otimes \cdots \otimes v_{a_n} \otimes v_{r-1})
		=
		\psi_{n+1} \cdot p^* \Theta^{r}_{g,n}(v_{a_1} \otimes \cdots \otimes v_{a_n})\, ,
	\]
	where $ p \colon \Mbar_{g,n+1} \to \Mbar_{g,n} $ is the forgetful map that forgets the last marked point.
\end{introthm}

It is worth noting that, although neither the Witten $r$-spin class nor the Theta class are semisimple, the construction of the latter is significantly easier. While abundant examples of non-semisimple CohFTs can be obtained from Gromov--Witten theory, the Theta class is the simplest example of a non-semisimple CohFT that we are aware of.

The Theta class $ \Theta^2_{g,n} $  for $ r = 2 $ was first considered by Norbury \cite{Nor23} and  has already seen various applications. For example, it appears in the context of Jackiw--Teitelboim supergravity \cite{SW20}, which in turn is related to the enumerative geometry of super Riemann surfaces \cite{Nor}. The moduli space of super hyperbolic Riemann surfaces with geodesic boundaries $\widehat{\mc{M}}_{g,n}(L)$ is a symplectic supermanifold and its supervolumes can be calculated as a tautological integral \cite{Nor}:
\[
	\mathrm{Vol}\bigl( \widehat{\mc{M}}_{g,n}(L) \bigr)
	=
	(-1)^n \, 2^{g-1+n} \int_{\overline{\mc{M}}_{g,n}}
	\Theta^2_{g,n}
	\exp\left( 2\pi^2\kappa_1 + \frac{1}{2} \sum_{i=1}^n L_i^2 \psi_i \right) .
\]
Other applications include Gromov--Witten theory of $ \P^1 $ coupled to the Theta class and the Legendre ensemble \cite{Nor22}, and integrable hierarchies of BKP type \cite{Ale23}.

%–––––––––––––––––––––––––––––––––––––––––––%
\subsection*{Tautological relations}
%–––––––––––––––––––––––––––––––––––––––––––%	
The analysis of the Witten class in \cite{PPZ15,PPZ19} relies on the key property that the Dubrovin--Frobenius manifold is generically semisimple. The authors work at a semisimple point, use the Teleman reconstruction theorem to understand the CohFT completely, and prove tautological relations by taking a limit to the non-semisimple point. In contrast, the Dubrovin--Frobenius manifold associated to the Theta class is \textit{nowhere semisimple}. Thus, the above line of attack fails.

In order to circumvent this issue, we will exploit one key difference between the Witten class and the Theta class concerning the range of primary fields. The Witten $ r $-spin class satisfies Ramond vanishing: this means that setting any of the primary fields $ a = 0 $ forces the class to be zero. On the other hand, the Theta class does not vanish upon setting any of the primary fields $ a = 0 $ (although they must be excluded for the CohFT axioms to be satisfied).

Having a non-vanishing cohomology class upon setting the primary fields to zero, allows us to deform the Theta class along the direction $v_0$:
\[
	\Theta^{r,\epsilon}_{g,n}(v_{a_1} \otimes \cdots \otimes v_{a_n})
	\coloneqq
	\sum_{m \geq 0} \frac{\epsilon^m}{m!} p_{m,*} \Theta^{r}_{g,n+m}(v_{a_1} \otimes \cdots \otimes v_{a_n}\otimes v_0^{\otimes m}) \, ,
\]
where $ p_{m} \colon \Mbar_{g,n+m} \to \Mbar_{g,n} $ is the map that forgets the last $ m $ marked points. We stress again that, as $ v_0 $ is not part of the vector space underlying the Theta class, the above deformation is not a shift along any direction of the associated Dubrovin--Frobenius manifold. Instead we view the deformed Theta class as a family of CohFTs parametrised by $ \epsilon \in \C$, which coincides with the Theta class in degree
\[
	D^r_{g;a} =  \frac{(r+2)(g-1) + n + \sum_{i=1}^n a_i}{r}
\]
and vanishes in all degrees higher than that.

We prove that for any $ \epsilon \neq 0 $, the deformed Theta class is a semisimple, homogeneous CohFT of rank $r-1$. Then, we compute all the ingredients of the Teleman reconstruction theorem to find an expression for $ \Theta^{r,\epsilon} $ in terms of tautological classes. Taking the limit $ \epsilon \to 0 $ back to the Theta class allows us to produce a collection of tautological relations on $ \Mbar_{g,n} $. Our result is the first application of Teleman's reconstruction theorem for CohFTs without a flat unit that we are aware of. Teleman's reconstruction theorem for CohFTs without a flat unit does not  apply for $ n = 0 $ (as we explain in \cref{subsubsec:Teleman}), which results in the exception for $ n = 0 $ in \cref{thm:intro:rels} and \cref{cor:intro:r=2}. We summarize these results from \cref{sec:reconstructionTheta} below.

The Givental--Teleman reconstruction recipe for homogeneous CohFTs without a flat unit requires the data of an $ R $-matrix $ R(u) = \Id + \sum_{m\geq 1} R_m u^m \in \End(V)\bbraket{u} $, a vacuum vector $\mathbf{v}(u) \in V\bbraket{u} $, and the degree zero part of the CohFT in question, called the topological field theory $ w_{g,n} \colon V^{\otimes n} \to H^0(\Mbar_{g,n})$. Using Teleman's characterisation of the $R$-matrix and the vacuum vector as solutions to differential equations \cite{Tel12}, we calculate both explicitly. The $ R $-matrix elements for the deformed Theta class are 
\begin{equation*}
	R^{-1}(u)_a^b
	=
	\epsilon^{\frac{a-b}{r-1}}
	\sum_{ \substack{m \ge 0 \\ b+m \equiv a \pmod{r-1}}}
		P_m(r,a-1) \left( \frac{u}{(1-r) \epsilon^{\frac{r}{r-1}}} \right)^m \, .
\end{equation*}
The coefficients $P_m(r,a)$, first found in \cite{PPZ19} and interpreted as coefficients of the asymptotic expansion of the hyper-Airy function and its derivatives in \cite{CCGG24}, are computed recursively as
\begin{equation*}
\begin{cases}
	P_m(r,a) - P_m(r,a-1) = r \left( m - \frac{1}{2} - \frac{a}{r} \right) P_{m-1}(r,a-1) \, ,
	\qquad
	\text{for } a = 1,\dots,r-2 \, , \\
	P_m(r,r-1) = P_m(r,0)  \, .
\end{cases}
\end{equation*}
The vacuum vector for the deformed Theta class is defined as
\begin{equation*}
	\begin{cases}
		\mathbf{v}(u)^a
		=
		- \epsilon^{\frac{a+1}{1-r}}
		\sum_{m \geq 0} 
		H_m(r,a) \biggl( \frac{u}{\epsilon^{\frac{r}{r-1}}} \biggr)^{(r-1)m + r-2 - a} \, \qquad
		\text{for } a = 1,\dots,r-2\,, \\
		\mathbf{v}(u)^{r-1}
		=
		- \epsilon^{\frac{r}{1-r}}
		\sum_{m \geq 0} 
		H_m(r,0) \biggl( \frac{u}{\epsilon^{\frac{r}{r-1}}} \biggr)^{(r-1)m + r-2 }
	\end{cases}
\end{equation*} 
The coefficients $H_m(r,a)$ are interpreted as coefficients of the asymptotic expansion of the hyper-Scorer functions and defined by
\begin{equation*}
		H_m(r,a) = \frac{(rm+r-1-a)!}{m! r^m} \, ,
		\qquad
		\text{for } a = 0,\dots,r-2\,.
\end{equation*}
From the $R$-matrix and the vacuum, we can define the translation $ T(u) = u \bigl( \mathbf{1} - R^{-1}(u) \mathbf{v}(u) \bigr)$. Lastly, the topological field theory $ w_{g,n} $ is
\[
	w_{g,n}(v_{a_1} \otimes \cdots \otimes v_{a_n}) = (-1)^n \epsilon^{\frac{(r+2)(g-1) + n + \sum_i a_i}{r-1}} (r-1)^g \cdot \delta\,,
\]
where $\delta$ equals $1$ if $3g-3+n+\sum_i a_i \equiv 0 \pmod{r-1}$ and $0$ otherwise. 

\begin{introthm}[Tautological relations]\label{thm:intro:rels}
	The deformed Theta class can be expressed  as 
	\[
		\Theta^{r, \epsilon}_{g,n} = R T w_{g,n} \, , 
	\] for all stable (g,n), except for $ (g,n) = (g,0) $ and degree $ d = 3g-3 $,
	and the Theta class is the term of degree $ d = D^{r}_{g;a}$:
	\[
		\Theta^{r}_{g,n}(v_{a_1} \otimes \cdots \otimes v_{a_n}) = \bigl[(R T w_{g,n})(v_{a_1} \otimes \cdots v_{a_n}) \bigr]^{D^r_{g;a}} \, . 
	\]
	All the terms of degree $ d > D^r_{g;a} $ vanish and thus produce relations among tautological classes:
	\[
		\bigl[(R T w_{g,n})(v_{a_1} \otimes \cdots v_{a_n}) \bigr]^{d} = 0 \in H^{2d}(\Mbar_{g,n})\,,
		\qquad \text{ for } d > D^{r}_{g;a} \, (\text{except } (g,n) = (g,0) \text{ and } d = 3g-3) \,.
	\]
\end{introthm}

While our result is not valid in Chow, as Teleman's result has not been extended to Chow field theories, we expect the above relations to hold in the Chow ring. Moreover, we expect them to be implied by Pixton's relations, although we do not have a proof of this statement. In \cite{Jan18}, Janda proves that all tautological relations that one can obtain by taking a limit from a semisimple point to a point on the discriminant of a Dubrovin--Frobenius manifold are implied by Pixton's relations. However, our relations in \cref{thm:intro:rels} do not fall into this class of tautological relations, as we have emphasised previously.

For $ r \geq 3 $, we provide a description of the deformed Theta class and the relations in \cref{thm:intro:rels} in terms of weighted stable graphs in \cref{prop:stablegraphs}. When $ r = 2 $ the statement takes a very simple form, involving $ \kappa $-classes only, and proves a conjecture of Kazarian--Norbury \cite[conjectures~1 and 4]{KN24}. More precisely, consider $s_m$ for $m > 0$ defined uniquely via
\begin{equation*}
	\exp\left(-\sum_{m > 0}  s_m u^m \right) = \sum_{k \geq 0} (-1)^k (2k+1)!!  u^{k}\,.
\end{equation*}
	
\begin{introcor}[{Kazarian--Norbury conjecture}]\label{cor:intro:r=2}
	We have the following vanishing relations among $\kappa$-classes:
	\begin{equation*}
		\left[ \exp\biggl( \sum_{m > 0} s_m \kappa_m \biggr) \right]^{d} = 0 \in H^{2d}(\Mbar_{g,n})\,,
		\qquad
		\text{for } d > 2g - 2 + n \, (\text{except } (g,n) = (g,0) \text{ and } d = 3g-3)\,.
	\end{equation*}
	Moreover, in degree $ d = 2g-2+n $, we get Norbury's Theta class:
	\begin{equation*}
		\Theta^2_{g,n} = \left[ \exp\biggl( \sum_{m > 0} s_m \kappa_m \biggr) \right]^{2g-2+n}  \in H^{2(2g-2+n)}(\Mbar_{g,n}) \,.
	\end{equation*}
	Finally, for $ n = 0 $, and $ d = (3g-3) $, we have the following statement, for $ g \geq 2 $,
	\begin{equation*}
		\int_{\Mbar_{g} } \exp\biggl( \sum_{m > 0} s_m \kappa_m \biggr) = (-1)^g \frac{B_{2g}}{(2g)(2g-2)} = |\chi(M_g)|,
	\end{equation*} where $ B_{2g} $ denotes the $ (2g) $-th Bernoulli number and $ \chi(\mathcal M_g) $ is the orbifold Euler characteristic of $ \mathcal{M}_g $.
\end{introcor}

See also \cite{YZ} for an independent proof of the last equation via the Hodge--BGW correspondence. \Cref{cor:intro:r=2} provides a geometric explanation for the vanishing relations among the $\kappa$-classes observed and conjectured by Kazarian and Norbury \cite{KN24}. These are \textit{universal} polynomial relations (i.e. independent of the genus and the number of marked points) that hold among the $\kappa$-classes. The first polynomials in $\exp( \sum_{m > 0} s_m \kappa_m )$ are:
\begin{multline*}
	1
	+
	3 \kappa_1 %t
	+
	\frac{3}{2} \left( 3 \kappa_1^2-7 \kappa_2\right) %t^2
	+
	\frac{3}{2} \left( 3 \kappa_1^3-21 \kappa_2 \kappa_1+46 \kappa_3\right) %t^3
	+
	\frac{9}{8} \left( 3 \kappa_1^4-42 \kappa_2 \kappa_1^2+184 \kappa_3 \kappa_1+49 \kappa_2^2-562 \kappa_4\right) %t^4
	\\
	+
	\frac{9}{40} \left( 9 \kappa_1^5-210 \kappa_2 \kappa_1^3+1380 \kappa_3 \kappa_1^2+735 \kappa_2^2 \kappa_1-8430 \kappa_4 \kappa_1-3220 \kappa_2 \kappa_3+32216 \kappa_5\right) %t^5
	+
	\cdots .
	%\mathrm{O}( t^6 )
\end{multline*}
The statement for $ n = 0 $ is also conjectured in \textit{loc.~cit.} and we prove it using the identification with the Eynard--Orantin topological recursion in \cref{prop:Fg}. We also comment on a curious similarity to Gromov--Witten potentials at the conifold point (called the conifold gap in \cite{HK07}) in \cref{rem:n=0} that was pointed out to us by F.~Janda.

%–––––––––––––––––––––––––––––––––––––––––––%
\subsection*{\texorpdfstring{$W$}{W}-constraints and integrability}
%–––––––––––––––––––––––––––––––––––––––––––%
Given that  the Theta class is a negative spin analogue of the Witten class, we expect it to satisfy a version of the Witten $ r $-spin conjecture. The original formulation of the Witten $ r $-spin conjecture states that the descendant potential of the Witten $ r $-spin class is the tau function of the $ r $-KdV hierarchy satisfying the string equation \cite{Wit93}. An equivalent version \cite{AvM92} states that the descendant potential is a highest weight vector in a certain representation of a $ \mathcal{W} $-algebra (which is a generalisation of the Virasoro algebra), or in other words, that the descendant potential is the unique solution to a set of $ \mathcal{W} $-constraints.

A very useful tool to understand the connections between CohFTs and $ \mathcal{W} $-constraints is the Eynard--Orantin topological recursion \cite{EO07}. The topological recursion is a universal formalism that takes as input an algebraic curve along with some extra data called a spectral curve, and recursively constructs multidifferentials known as correlators on the underlying curve. From the spectral curve, one can build a semisimple CohFT such that the multidifferentials can be expressed in terms of descendant integrals of this CohFT \cite{DOSS14}. Conversely, however, the answer to whether one can (and if so how to) obtain a global spectral curve from a given semisimple CohFT is unanswered in general. While there is a partial answer in \cite{DNOS19}, our situation where the Dubrovin--Frobenius manifold does not have a flat unit vector field is not covered by their results.

Nevertheless, in \cref{sec:SCdTheta}, we find a global spectral curve whose topological recursion correlators encode the descendant theory of the Theta class.

\begin{introthm}[Topological recursion]\label{intro:thm:TR}
	The CohFT associated to the $1$-parameter family of spectral curves $ \mathcal{S}_{\epsilon} $ on $\P^1$ given by 
	\begin{equation*}
		x(z) = \frac{z^r}{r} -  \epsilon z\,,
		\qquad\quad
		y(z) = - \frac{1}{z}\,,
		\qquad\quad
		B(z_1,z_2) = \frac{dz_1 dz_2}{(z_1 - z_2)^2}\,,
	\end{equation*}
	is the deformed Theta class $ \Theta^{r, \epsilon} $. More precisely, the  correlators corresponding to the spectral curve $ \mathcal{S}_{\epsilon} $ are
	\begin{equation*}
		\omega_{g,n}(z_1,\dots,z_n)
		=
		\sum_{a_1,\dots,a_n = 1}^{r-1}
		\int_{\overline{\mathcal{M}}_{g,n}}
		\Theta^{r,\epsilon}_{g,n}({v}_{a_1} \otimes\cdots\otimes {v}_{a_n}) \prod_{i=1}^n \sum_{k_i \ge 0}
		\psi_i^{k_i} d\xi^{k_i,a_i}(z_i)\,,
	\end{equation*}
	where the $ d \xi^{k,a}(z) $ are certain explicit differentials given in \cref{eqn:xi}.
\end{introthm}

In order to prove the theorem, we compute the CohFT associated to this spectral curve using the prescription of \cite{DOSS14}. We find that the $ R $-matrix and the translation $ T $ of the corresponding CohFT can be expressed in terms of the asymptotic expansion of certain solutions of the hyper-Airy differential equation (and the associated inhomogeneous equation):
\[
	f^{(r-1)} (t) = (-1)^{r-1} t \, f(t)\,,
\]
which for $ r = 3 $ reduces to the classical Airy ODE. Our computations rely on the calculation of \cite{CCGG24} in the context of the shifted Witten class. Finally, we match the $ R $-matrix and translation thus obtained with the ones we computed using Teleman's formulae in \cref{thm:intro:rels}, to finish the proof.

The fact that the $ R $-matrix for the deformed Theta class (\cref{thm:intro:rels}) essentially matches the $ R $-matrix for the $ e_1 $-shifted Witten class studied in \cite{PPZ19} is intriguing. From the perspective of the topological recursion, \cref{intro:thm:TR} provides an explanation for this occurrence: the function $ x(z) $ for the spectral curve is exactly the same in both cases \cite{CCGG24}. However, we do not know  of a purely algebro-geometric reason for this phenomenon and this deserves further investigation. On a related note, the Dubrovin--Frobenius manifold associated to the deformed Theta class has not been studied so far in the literature to our knowledge. Perhaps a singularity theory  understanding of the Dubrovin--Frobenius manifold could help explain why the $ R $-matrices match.

By taking the parameter $ \epsilon \to 0 $ in the above theorem, we obtain as an immediate corollary that the descendant integrals of the Theta class are computed by the Bouchard--Eynard topological recursion on the  \textit{$ r $-Bessel spectral curve}, which is defined on $ \P^1 $ by
\begin{equation*}
	x(z) = \frac{z^r}{r} \,,
	\qquad\quad
	y(z) = - \frac{1}{z}\,,
	\qquad\quad
	B(z_1,z_2) = \frac{dz_1 dz_2}{(z_1 - z_2)^2}\,.
\end{equation*}
The Bouchard--Eynard topological recursion was analysed thoroughly in the context of higher Airy structures in \cite{BBCCN24} and  proved to be equivalent to a set of $ \mathcal{W} $-constraints in general. The $ \mathcal{W} $-algebra we are interested in here is $ \mathcal{W}^k(\mathfrak{gl}_r) $ at the self-dual level $ k = 1 - r $. Putting together the identification of the correlators of the $ r $-Bessel spectral curve with descendant integrals of the Theta class and the results of \cite{BBCCN24}, we get one of our main results (\cref{thm:Wconst}) which can be viewed as a direct analogue of the Witten $ r $-spin conjecture to negative spin:

\begin{introthm}[{$\mc{W}$-constraints}]\label{thm:intro:Z}
	The descendant potential of the Theta class
	\[
		Z^{\Theta^r}
		=
		\exp\left(
			\sum_{\substack{ g \geq 0, n \geq 1 \\ 2g-2+n > 0}} \frac{\hbar^{g-1}}{n!}
			\sum_{a_1,\ldots, a_n =1}^{r-1}
				\int_{\Mbar_{g,n}} \Theta^{r}_{g,n}\left(v_{a_1} \otimes \cdots \otimes v_{a_n}\right) \prod_{i=1}^n \sum_{k_i \geq 0} \psi_i^{k_i} \,\, (rk_i+a_i)!^{(r)} t_{a_i + r k_i} \right)
	\]
	is the unique solution to the following set of $ \mathcal{W} $-constraints
	\[
		H^i_k \, Z^{\Theta^r} = 0\,, \; \text{ for all } k \geq -i + 2 \; \text{ and } \; i = 1,2,\ldots, r\,,
	\]
	where the $ H^i_k $ are  differential operators defined in \cref{eqn:H:cnstrnts} that form a representation of the $ \mathcal{W}^{-r+1}(\mathfrak{gl}_r) $ algebra.
\end{introthm}

The above statement is extremely powerful, as it allows one to calculate any descendant integral by a recursion on the integer $ 2g-2+n $, i.e.~a topological recursion. This theorem answers a question posed in \cite{BBCCN24}, and the result was conjectured by the first-named author along with Borot and Bouchard.

It has been observed in many cases in the literature that $ \mathcal{W} $-constraints  and integrability properties of the corresponding  solutions go hand-in-hand. Examples include the Witten $ r $-spin conjecture \cite{AvM92}, its extensions to singularities of type $ D $ and $ E $ \cite{FJR13} and open intersection theory \cite{Ale17}. Motivated by these results, we formulate the following integrability conjecture.

\begin{introconj}[{$r$-KdV integrability}] \label{intro:conj}
	The descendant potential $ Z^{\Theta^r} $ is a tau function of the $ r $-KdV integrable hierarchy. Moreover, it coincides with the $ r $-spin Brézin--Gross--Witten tau function.
\end{introconj}

The $r$-spin Brézin--Gross--Witten ($ r $-BGW) tau function, introduced by Mironov, Morozov and Semenoff \cite{MMS96} and studied further in the recent work of Alexandrov and Dhara \cite{AD25}, is a generalisation of the Brézin--Gross--Witten tau function \cite{GW80,BG80}. It can be regarded as the second constituent of complex matrix models theory \cite{AMM09} together with the $r$-Witten--Kontsevich model, and \cref{intro:conj} gives an enumerative-geometric interpretation to it, generalising a conjecture of Norbury \cite{Nor23} to higher spin.

Another interesting feature of the BGW matrix integral (i.e. for $r = 2$) is its connection to single monotone Hurwitz numbers \cite{GGN14, Nov} in the strong coupling limit (the so-called character phase), which corresponds to $\epsilon\rightarrow 1$. By combining the ELSV-type formula for such Hurwitz numbers found in \cite{ALS16} along with the expression we find in \cref{cor:intro:r=2} for the  deformed Theta class $\Theta^{2,\epsilon}$, we get the following ELSV-type formula:
\[
	h_{g;\mu_1,\dots,\mu_n}^{\le}
	=
	(-1)^n\prod_{i=1}^{n} \binom{2\mu_i}{\mu_i}
	\int_{\overline{\mc{M}}_{g,n}}
		\Theta_{g,n}^{2,1} \,
		\prod_{j=1}^{n} \sum_{k_j \ge 0}
			\psi_j^{k_j} \frac{(2(\mu_j + k_j) - 1)!!}{(2\mu_j - 1)!!}\, .
\]
On the other hand, the connection between the BGW integral and the Theta class for $\epsilon = 0$ appearing in \cref{intro:conj} is concerned with the weak coupling limit (the Kontsevich phase). Strikingly, our analysis of the Theta class provides a family of cohomology classes that interpolates between the strong and weak coupling limits of the BGW integral.

In \cref{thm:intro:Z}, we have proved the most difficult part of the $ r $-KdV integrability conjecture. In addition, by analysing certain commutation relations of the $ \mathcal{W} $-algebra $ \mathcal{W}^{-r+1}(\mathfrak{gl}_r) $, we reduce the conjecture to the following equivalent statement: \cref{intro:conj} holds if and only if the $ r $-BGW tau function $Z^{r\textup{-BGW}}$ satisfies the \textit{string equation}:
\[
	H^r_{-r+2} \, Z^{r\textup{-BGW}} = 0 \,.
\]
Finally, the string equation for $ r = 2 $ (and indeed, the full set of Virasoro constraints) is already known thanks to \cite{GN92}, and was derived using the formalism of Kac--Schwarz operators in \cite{Ale18}. For $ r = 3 $, we prove the string equation using Kac--Schwarz operators. This proves \cref{intro:conj} in the case of $ r =2 $ and $ r= 3 $ (\cref{thm:r23}).

\begin{introthm}
	\Cref{intro:conj} is true  for $ r = 2 $ (Norbury's conjecture \cite{Nor23}) and for $ r =3 $.
\end{introthm}

%–––––––––––––––––––––––––––––––––––––––––––%
\subsection*{Outline of the paper}
%–––––––––––––––––––––––––––––––––––––––––––%

We provide a brief outline of the paper.

\begin{itemize}
	\item In \cref{sec:Theta}, we define the (deformed) Theta classes and prove that they form a CohFT.

	\item In \cref{sec:Teleman}  we investigate the Dubrovin--Frobenius manifold associated to the deformed Theta class, and apply the reconstruction theorem to find an expression for $ \Theta^r $ in terms of tautological classes and obtain tautological relations.

	\item In \cref{sec:TR}, we find global spectral curves corresponding to the deformed Theta class, and show that the $ r $-Bessel curve computes descendant integrals of the Theta class.

	\item Finally, in \cref{sec:W:cnstrs:rKdV}, we prove that the descendant potential of the Theta class is the unique solution to $ \mathcal{W} $-constraints. Then, we  conjecture that the $r$-BGW tau function matches with the descendant potential $ Z^{\Theta^r} $, and prove it for $ r=2 $ and $ 3 $. 
\end{itemize}
\addtocontents{toc}{\protect\setcounter{tocdepth}{2}}

%–––––––––––––––––––––––––––––––––––––––––––%
\subsection*{Acknowledgements}
%–––––––––––––––––––––––––––––––––––––––––––%

We thank Gaëtan Borot, Vincent Bouchard, Séverin Charbonnier, Alessandro Chio\-do, Bertrand Eynard, Felix Janda, Omar Kidwai, Danilo Lewański, Paul Norbury, Rahul Pandharipande, Sergey Shadrin, Johan van de Leur, and Di Yang for many helpful discussions. We also thank Alexander Alexandrov and Saswathi Dhara for discussions and sharing their results in \cite{AD25} with us before posting.

N.~K.~C.~is supported by the Max-Planck-Gesellschaft. A.~G. is supported by the Institut de Physique Th\'{e}orique Paris (IPhT), CEA, Universit\'{e} de Saclay. This work is partly a result of the ERC-SyG project ``Recursive and Exact New Quantum Theory'' (ReNewQuantum) which received funding from the European Research Council (ERC) under the European Union's Horizon 2020 research and innovation programme under grant agreement No 810573. We thank the Institut de Math\'ematiques de Jussieu - Paris Rive Gauche for its hospitality.

%–––––––––––––––––––––––––––––––––––––––––––%
\section{The Theta class}\label{sec:Theta}
%–––––––––––––––––––––––––––––––––––––––––––%

In this section, we introduce the main characters of our paper -- the Theta class $\Theta^r$ and its deformation $ \Theta^{r,\epsilon} $ -- and study their properties in detail.

%–––––––––––––––––––––––––––––––––––––––––––%
\subsection{Cohomological field theories} 
%–––––––––––––––––––––––––––––––––––––––––––%

In order to make our paper as self-contained as possible, we recall the notion of cohomological field theories (CohFTs), and describe some of their properties. For an excellent introduction to CohFTs, we refer the reader to the ICM 2018 address of Pandharipande \cite{Pan19}.

\begin{definition}\label{def:cohft}
	Let $V$ be a finite dimensional $\Q$-vector space equipped with a non-degenerate symmetric $2$-form $\eta$. A \emph{cohomological field theory} on $(V,\eta)$ consists of a collection $\Omega = (\Omega_{g,n})_{2g-2+n>0}$ of elements
	\begin{equation}
		\Omega_{g,n} \in H^{\bullet}(\Mbar_{g,n}) \otimes (V^{*})^{\otimes n}
	\end{equation}
	satisfying the following axioms.
	\begin{enumerate}
		\item[i)] {\sc Symmetry.} Each $\Omega_{g,n}$ is $S_n$-invariant, where the action of the symmetric group $S_n$ permutes  the marked points of $\Mbar_{g,n}$ and the copies of $(V^{*})^{\otimes n}$ simultaneously.

		\item[ii)] {\sc Gluing axiom.} With respect to the gluing maps
		\begin{equation}
		\begin{aligned}
			q \colon& \Mbar_{g-1,n+2} \longrightarrow \Mbar_{g,n}\,, \\
			r \colon& \Mbar_{g_1,n_1+1} \times \Mbar_{g_2,n_2+1}  \longrightarrow \Mbar_{g,n}\,,
			\qquad
			g_1 + g_2 = g, \; n_1 + n_2 = n\,,
		\end{aligned}
		\end{equation}
		we have
		\begin{equation}
		\begin{aligned}
			q^{*} \Omega_{g,n}(v_1 \otimes \cdots \otimes v_n)
			& =
			\Omega_{g-1,n+2}(v_1 \otimes \cdots \otimes v_n \otimes \eta^{\dag})\,, \\
			r^{*} \Omega_{g,n}(v_1 \otimes \cdots \otimes v_n)
			& =
			(\Omega_{g_1,n_1+1} \otimes \Omega_{g_2,n_2}) \left( \bigotimes_{i = 1}^{n_1} v_i \otimes \eta^{\dag} \otimes \bigotimes_{j = 1}^{n_2} v_{n_1+j}  \right),
		\end{aligned}
		\end{equation}
		where $\eta^{\dag} \in V^{\otimes 2}$ is the bivector dual to $\eta$.
	\end{enumerate}
	Often the vector space comes with a distinguished element $\mathbb 1 \in V$, which satisfies the following axiom:
	\begin{enumerate}
		\item[iii)] {\sc Unit axiom.} Consider the forgetful map
		\begin{equation}
			p \colon \Mbar_{g,n+1} \longrightarrow \Mbar_{g,n}\,.
		\end{equation}
		Then
		\begin{equation}
		\begin{aligned}
			p^{*} \Omega_{g,n}(v_1 \otimes \cdots \otimes v_n)
			=
			\Omega_{g,n+1}(v_1 \otimes \cdots \otimes v_n \otimes \mathbb{1}) \,.
		\end{aligned}
		\end{equation}
		Moreover, it is compatible with the pairing: $\Omega_{0,3}(v_1 \otimes v_2 \otimes \mathbb{1}) = \eta(v_1,v_2)$.
	\end{enumerate}
	In this case, $\Omega$ is called a cohomological field theory with the \emph{flat unit} $ \mathbb 1 $.

	Another possibility (cf.~\cite{Nor23}) is to substitute the unit axiom with a modified version. Suppose again that the vector space comes with a distinguished element $\nu \in V$. Then we can also ask for:
	\begin{enumerate}
		\item[iii')] {\sc Modified unit axiom.} Under the forgetful map,
		\begin{equation}
		\begin{aligned}
			\psi_{n+1} \cdot p^{*} \Omega_{g,n}(v_1 \otimes \cdots \otimes v_n)
			=
			\Omega_{g,n+1}(v_1 \otimes \cdots \otimes v_n \otimes \nu)\, .
		\end{aligned}
		\end{equation}
	\end{enumerate}
	In this case, $\Omega$ is called a cohomological field theory with a \emph{modified unit} $ \nu $.
\end{definition}

A CohFT determines a product $\bullet$ on $V$, called the \emph{quantum product}: $v_1 \bullet v_2$ is defined as the unique vector in $ V $ such that for all $v_3 \in V$ the following holds:
\begin{equation}
	\eta(v_1 \bullet v_2,v_3) = \Omega_{0,3}(v_1 \otimes v_2 \otimes v_3)\,.
\end{equation}
Commutativity and associativity of $\bullet$ follow from (i) and (ii) respectively, and thus $ V $ equipped with the quantum product carries the structure of  a commutative associative $ \Q $-algebra. If the CohFT has a flat unit, the quantum product is unital, with $\mathbb{1} \in V$ being the identity by (iii). An important class of CohFTs are the \textit{semisimple} ones: a CohFT $\Omega$ on $(V,\eta)$ is called semisimple if $(V,\bullet)$ is a semisimple algebra, i.e. if there exists a basis $(e_i)$ of idempotents
\begin{equation}
	e_i \bullet e_j = \delta_{i,j} \, e_i\,,
\end{equation}
after an extension of the base field to $\C$.

The degree $0$ part of a CohFT
\begin{equation}
	w_{g,n} = \deg_0{\Omega_{g,n}} \in H^{0}(\Mbar_{g,n}) \otimes (V^{*})^{\otimes n} \cong (V^{*})^{\otimes n}
\end{equation}
is also a CohFT and is called a $2d$ \textit{topological field theory} (TFT for short). It is uniquely determined by the values of $w_{0,3}$ and by the pairing $\eta$ (or equivalently, by the associated quantum product and the pairing $\eta$) by a repeated application of the gluing axiom (ii).  

Examples of CohFTs include the pushforward to the moduli space of curves of the virtual fundamental class in Gromov--Witten theory, the Witten $ r $-spin class and the Hodge class.

%–––––––––––––––––––––––––––––––––––––––––––%
\subsection{Chiodo classes and Theta classes} \label{sec:thetadef}
%–––––––––––––––––––––––––––––––––––––––––––%

An important class of CohFTs that generalise the Hodge class are called Chiodo classes. These are defined using the moduli space of spin curves, and  we refer to \cite{Jar00,Chi08} for details on the construction of this moduli space. 

\begin{definition}\label{defn:Chiodo:class}
	For a fixed positive integer $r$, and integers $s,a_1,\dots,a_n$ satisfying the modular constraint
	\begin{equation}
		\sum_{i=1}^n a_i \equiv s (2g-2+n) \pmod{r}\,,
	\end{equation}
	consider the moduli space $\Mbar_{g;a}^{r,s}$ of objects $(C,x_1,\dots,x_n,L)$, where $(C,x_1,\dots,x_n)$ is a stable curve of genus $g$ with $n$ marked points, and $L $ is a line bundle on $ C $ such that
	\begin{equation}
		L^{\otimes r} \cong \omega_{\log}^{\otimes s}\left( -\sum_{i=1}^n a_i x_i \right).
	\end{equation}
	Here $\omega_{\log} \coloneqq \omega(\sum_{i=1}^n x_i)$ is the log canonical bundle. This moduli space, called the \emph{moduli space of spin curves}, is naturally equipped with a universal curve $ \overline{\mathcal{C}}_{g;a}^{r,s} $ and a universal line bundle $ \mathcal{L}_{g;a}^{r,s} $ on the universal curve:
	\begin{equation}
		\pi \colon \overline{\mathcal{C}}_{g;a}^{r,s} \longrightarrow \Mbar_{g;a}^{r,s}\,,
		\qquad\qquad
		\mathcal{L}_{g;a}^{r,s} \longrightarrow \overline{\mathcal{C}}_{g;a}^{r,s}\,.
	\end{equation}
	By forgetting the extra data of the line bundle $ L $, we also have a forgetful map $f \colon \Mbar_{g;a}^{r,s} \to \Mbar_{g,n}$ to the moduli space of stable curves.

	Define the \emph{Chiodo class} as
	\begin{equation}\label{eqn:Chiodo:class}
		C_{g,n}^{r,s}(a_1,\dots,a_n)
		\coloneqq
		f_{\ast} c(- R^{\bullet}\pi_{\ast}\mathcal{L}_{g;a_1,\dots,a_n}^{r,s})
		\in H^{\bullet}(\Mbar_{g,n})\,.
	\end{equation}
	Here $R^{\bullet}\pi_{\ast}\mathcal{L}_{g;a}^{r,s}$ is the derived pushforward of $\mathcal{L}_{g;a}^{r,s}$, and $c$ is its total Chern class.
\end{definition}

We note that in general the derived pushforward $ R^{\bullet}\pi_{\ast}\mathcal{L}_{g;a}^{r,s} $ is a complex with multiple cohomology sheaves. However, in the specific range $ -r + 1 \leq s \leq -1 $, an easy Riemann--Roch calculation  (cf.~\cite{GLN23}) shows that $ R^{0}\pi_{\ast}\mathcal{L}_{g;a}^{r,s} $ vanishes, and  the derived pushforward is an honest vector bundle on the moduli space of spin curves. Let us denote this vector bundle as
\begin{equation}\label{eqn:V:vb}
	\mathcal{V}^{r,s}_{g,a} \coloneqq R^{1}\pi_{\ast}\mathcal{L}_{g;a}^{r,s}\, ,
\end{equation}
for $ -r+1 \leq s \leq -1  $.

In \cite{Chi08+}, Chiodo computed the Chern characters of $R^{\bullet}\pi_{\ast}\mathcal{L}_{g;a}^{r,s}$ in terms of tautological classes. From this formula, one can derive many interesting properties satisfied by the class $C_{g,n}^{r,s}$. For instance, as long as $0 \le a_i < r$, it forms a CohFT \cite{LPSZ17}; it can be expressed as a sum over stable graphs \cite{JPPZ17}; and it enjoys several properties regarding shifts in the parameters, pullback with respect to the forgetful map, as well as string and dilaton equations \cite{GLN23}. We recall the properties that will be useful in the text here.

\begin{proposition}[{\cite{LPSZ17,GLN23}}]\label{prop:Chiodo:properties}
	Fix $g,n \geq 0$ integers such that $2g - 2 + n > 0$. Let $r$ and $s$ be integers with $r$ positive, and $a_1, \ldots, a_n$ integers satisfying the modular constraint $a_1 + \cdots + a_n \equiv (2g-2+n)s \pmod{r}$. Chiodo's classes satisfy the following properties.
	\begin{enumerate}
		\item\label{prop:CohFT:Chiodo}
		Let $W = \operatorname{span}_{\Q} (v_0,\ldots, v_{r-1})$. Then the collection of maps
		\begin{equation}
			C^{r,s}_{g,n} \colon W^{\otimes n} \longrightarrow H^{\bullet}(\Mbar_{g,n})\,,
			\qquad\quad
			v_{a_1} \otimes \cdots \otimes v_{a_n} \longmapsto C^{r,s}_{g,n}(a_1,\dots,a_n)
		\end{equation}
		is a CohFT on $W$ with pairing 
		\begin{equation}\label{eqn:Chiodo:pairing}
			\eta_{C}(v_a,v_b) = \frac{1}{r} \, \delta_{a+b \equiv 0 \pmod{r}}\,.
		\end{equation} 
		If $0 \le s < r$, the CohFT admits a flat unit $v_s$.

		\item\label{prop:shift:ai}
		Shift in $a_i$:
		\begin{equation}
			C^{r,s}_{g,n}(a_1, \dots, a_i + r, \dots, a_n)
			=
			\left( 1 + \frac{a_i}{r}\psi_i\right) \cdot C^{r,s}_{g,n}(a_1, \dots, a_n)\,.
		\end{equation}

		\item\label{prop:pullback}
		Pullback property: for $0 \le a_1, \ldots, a_n < r$,
		\begin{equation}
			C^{r,s}_{g,n+1}(a_1, \dots, a_n, s)
			=
			p^{*}C^{r,s}_{g,n}(a_1, \dots, a_n)\,.
		\end{equation}
	\end{enumerate}
\end{proposition}

We note that combining property~(\labelcref{prop:shift:ai}) and then property~(\labelcref{prop:pullback}) of \cref{prop:Chiodo:properties} gives the equation:
\begin{equation}\label{eqn:modunit}
	C^{r,s}_{g,n+1}(a_1, \dots, a_n, s + r) = \left( 1+ \frac{s}{r} \psi_{n+1} \right) \cdot p^{*}C^{r,s}_{g,n}(a_1, \dots, a_n) \,.
\end{equation}
In the following, we will focus on the case $s = -1$. As we have already mentioned, in this case the Chiodo class is the pushforward of the total Chern class of an honest vector bundle $\mathcal{V}^{r,-1}_{g;a}$ on the moduli space of spin curves. Our main interest in this paper is in the top degree of the Chiodo class.  Again, let $W = \operatorname{span}_{\Q} (v_0,\ldots, v_{r-1})$, and define the collection of maps 
\begin{equation}
	\Upsilon^r_{g,n} \colon W^{\otimes n} \longrightarrow  H^{\bullet}(\overline{\mathcal M}_{g,n})\,,
\end{equation}
as follows:
\begin{equation}\label{eqn:thetadef}
	\Upsilon^r_{g,n} (v_{a_1} \otimes \cdots \otimes v_{a_n})
	\coloneqq
	(-1)^n \, r^{\frac{2g-2+|a| +n}{r}} \, f_* c_{\textup{top}} (\mathcal{V}^{r,-1}_{g;a})
	\in H^{\bullet}(\overline{\mathcal M}_{g,n})\,,
\end{equation}
where we define $|a| \coloneqq \sum_{i=1}^n a_i$. In other words, $\Upsilon^r_{g,n}$ is (up to the prefactor) the top degree of $C^{r,-1}_{g,n}$. By an easy Riemann--Roch calculation, we conclude that the (complex) degree of the cohomology class is 
\begin{equation}\label{eqn:degUpsilon}
	\deg \Upsilon^r_{g,n} (v_{a_1} \otimes \cdots \otimes v_{a_n}) = \frac{(r+2)(g-1) + n + |a|}{r}\,.
\end{equation}
We consider the restriction of the maps $\Upsilon^r_{g,n}$ to the vector subspace spanned by $v_{a}$ for $ 1\leq a \leq r-1$.

\begin{definition}\label{def:Theta}
	Define  $V \coloneqq \operatorname{span}_{\Q} (v_1, \ldots, v_{r-1})$, and define the \textit{Theta class} $ \Theta^r $ as the restriction 
	\begin{equation}\label{eqn:thetaCohft}
		\Theta^{r}_{g,n} \coloneqq \Upsilon^r_{g,n} \big|_{V}\,.
	\end{equation} To be precise, the arguments $ v_a $ of the Theta class are only allowed to take  values $ 1 \leq a \leq r-1 $.
	The class $\Theta^{r}_{g,n}$ has pure degree
	\begin{equation}\label{eqn:degtheta}
		D^r_{g;a} \coloneqq \frac{(r+2)(g-1) + n + |a|}{r}\,.
	\end{equation}
\end{definition}

The case $r = 2$ coincides with the Theta class introduced by Norbury \cite{Nor23}: $\Theta_{g,n}^{2}(v_1^{\otimes n}) = \Theta_{g,n}$.

We will prove shortly that the cohomology classes $\Theta^{r}_{g,n}$ form a CohFT. However, it is not a semisimple cohomological field theory. When one encounters a non-semisimple CohFT, a standard trick (see for example \cite{PPZ15,PPZ19}) is to pass to the associated Dubrovin--Frobenius manifold. One can associate a CohFT to every point on the  Dubrovin--Frobenius manifold (see \cref{sec:DFmanifolds}). When the associated Dubrovin--Frobenius manifold is generically semisimple, one can shift along a semisimple direction, to work at a semisimple point instead. Then, one can study the semisimple CohFT at that point using the Givental--Teleman reconstruction theorem. Finally, by taking the limit back to the non-semisimple point, one can analyse the non-semisimple CohFT of interest. However, we stress that the Dubrovin--Frobenius manifold associated to the Theta class $\Theta^{r}_{g,n}$ is \emph{not} generically semisimple -- see \cref{rem:nss} for details. Thus, the usual strategy fails. 

To bypass this issue, we construct a family of semisimple CohFTs instead, depending on a non-zero parameter $\epsilon $, such that the limit $\epsilon \to 0$ recovers the Theta class $ \Theta^{r}$.

%–––––––––––––––––––––––––––––––––––––––––––%
\subsection{Deformed Theta class} \label{sec:deformedTheta}
%–––––––––––––––––––––––––––––––––––––––––––%

In this section, we define a deformation $ \Theta^{r,\epsilon} $ of the Theta class by shifting it along the vector $v_0$. We stress that the direction $\epsilon \cdot v_0$ is not in the vector space $V$ (where the Theta class is defined), but in the vector space $W$. Hence, this is \textit{different} from working at a semisimple point of the Dubrovin--Frobenius manifold associated to the Theta class $ \Theta^{r} $. 

\begin{definition}
	We define the \textit{deformed Theta class} $ \Theta^{r,\epsilon}_{g,n} $
	\begin{equation}
		\Theta^{r,\epsilon}_{g,n} \colon V^{\otimes n} \longrightarrow H^{\bullet}(\overline{\mathcal M}_{g,n})
	\end{equation}
	as follows:
	\begin{equation}\label{eqn:deformedthetadef}
		\Theta^{r,\epsilon}_{g,n} (v_{a_1} \otimes  \cdots \otimes  v_{a_n})
		\coloneqq
		\sum_{m \geq 0} \frac{\epsilon^m}{m!} p_{m,*} \Upsilon^r_{g,n+m}(v_{a_1}\otimes  \cdots\otimes v_{a_n} \otimes v_0^{\otimes m})\,.
	\end{equation}
\end{definition}

Before analysing the deformed Theta class, we need to justify that the sum in the above definition is finite and thus the class $\Theta^{r,\epsilon}_{g,n}$ is well-defined. This follows from the degree calculation
\begin{equation*}
	\deg p_{m,*} \Upsilon^r_{g,n+m}( v_{a_1} \otimes  \cdots \otimes  v_{a_n} \otimes  v_0^{\otimes m})
	=
	\frac{(r+2)(g-1) + n + m + |a|}{r} - m
	=
	D^r_{g;a} - \frac{(r-1)m}{r}\,.
\end{equation*}
In addition, this implies that the class $\Theta^{r,\epsilon}_{g,n}$ is equal to the class $\Theta^{r}_{g,n}$ in top degree, with possibly some correction terms in strictly smaller degree. We also note that this is analogous to the situation studied in \cite{PPZ15} for the shifts of the Witten $ r $-spin class.

\begin{proposition}\label{prop:cohft}
	The collection $\Theta^{r,\epsilon}_{g,n}$ for $2g-2+n > 0$ satisfies the axioms of an $(r-1)$-dimensional cohomological field theory on $V$ with pairing
	\begin{equation}\label{eqn:Theta:pairing}
		\eta\left( v_a, v_b \right) = \delta_{a+b,r}\,.
	\end{equation}
\end{proposition}

\begin{proof}
	The first axiom of a CohFT is the $ S_n $-equivariance. This is immediate from the definition of the deformed Theta class.

	For the gluing axiom, let $\mathbb{v} = v_{a_1} \otimes \cdots \otimes v_{a_n}$ be a generic tensor in $V^{\otimes n}$. We prove the statement for the gluing map $q$ of non-separating kind, and omit the proof for the gluing maps of separating kind as it is completely analogous. We recall that the Chiodo classes form a CohFT (see \cref{prop:Chiodo:properties}). Using the gluing axiom for the Chiodo class and rescaling by the proper minus sign and power of $r$, we get
	\begin{equation*}
		(-1)^{n+m} r^{\frac{2g-2+|a| +n +m}{r}}
		q^* C^{r,-1}_{g,n + m} (\mathbb{v} \otimes  v_0^{\otimes  m})
		=
		(-1)^{n+m} r^{\frac{2g-2+|a| +n + m}{r}}
		C^{r,-1}_{g-1,n+m+2} (\mathbb{v} \otimes  v_0^{\otimes m} \otimes  \eta_C^\dagger)\,.
	\end{equation*}
	Here $\eta_C$ is the pairing of the Chiodo CohFT from \cref{eqn:Chiodo:pairing}. Now, we want to take the degree $ \left(D^r_{g;a} + \frac{m}{r}\right) $ part of the above equation. This forces us to keep the part of the bivector $\eta_C^\dagger$ that is of the form $v_b \otimes v_c$ such that $b + c = r$. Thus, we see that neither $b$ nor $c$ can be $0$. We also note that the part of the bivector $\eta^\dagger_C$ that survives is indeed $r \cdot \eta^\dagger$ as we defined above in \cref{eqn:Theta:pairing}. Putting all of this together, we get
	\[
		q^* \Upsilon^r_{g,n+m}(\mathbb{v} \otimes v_0^{\otimes  m} )
		=
		\Upsilon^r_{g-1,n+m+2} (\mathbb{v} \otimes v_0^{\otimes  m} \otimes  \eta^\dagger )\,.
	\]
	Applying $\sum_{m \geq 0} \frac{\epsilon^m}{m!} p_{m,*}$ to the above equation and base changing yields the thesis.
\end{proof}

Notice that the above proof goes through without any modification even if $ \epsilon = 0 $. Thus, setting $ \epsilon =0 $ in \cref{prop:cohft}, we see that the Theta class is a cohomological field theory. Moreover, it satisfies the modified unit axiom.

\begin{theorem}\label{thm:theta:cohft}
	The Theta class $\Theta^{r}_{g,n}$ for $2g-2+n > 0$ satisfies the axioms of an $(r-1)$-dimensional cohomological field theory on $(V,\eta)$ and the modified unit axiom with distinguished vector $v_{r-1}$:
	\begin{equation}
		\psi_{n+1} \cdot p^* \Theta^{r}_{g,n}(v_{a_1} \otimes \cdots \otimes v_{a_n})
		=
		\Theta^{r}_{g,n+1}(v_{a_1} \otimes \cdots \otimes v_{a_n} \otimes v_{r-1}) \,.
	\end{equation}
\end{theorem}

\begin{proof}
	The only additional statement here is the modified unit axiom. Let $\mathbb{v} \coloneqq v_{a_1} \otimes \cdots \otimes v_{a_n} \in V^{\otimes n}$ denote an arbitrary element. By taking degree $D^r_{g;a} + 1$ of \cref{eqn:modunit} with $s = -1$ and rescaling by the proper minus sign and power of $r$, we get the thesis. 
\end{proof}

The definition of the pairing in \cref{prop:cohft} shows that the basis $ (v_a)_{a\in[r-1]} $ for the (deformed) Theta class is the flat basis. In \cref{sec:TR}, we will also encounter the canonical basis of this CohFT, which will be denoted $( e_i )_{i\in[r-1]} $.

\begin{remark}\label{rem:v0}
	The proof of \cref{prop:cohft} shows that the vector $ v_0 $ cannot be added to the vector space $ V $ while still ensuring that the (deformed) Theta class is a CohFT. The reason is that the proof of the gluing axiom requires the pairing $ \eta $ to be such that $ \eta(v_a,v_b) = \delta_{a+b,r} $. If $ a $ (or $ b $) was zero, the $ 2 $-form would no longer be non-degenerate. Thus, we cannot include $ v_0 $ in the underlying vector space.
\end{remark}

\begin{remark}\label{rem:Upsilon:modunit}
	We also note that the argument in the proof of \cref{thm:theta:cohft} shows that in general the collection of maps $\Upsilon^r_{g,n}$ defined on the space $W^{\otimes n}$ satisfies the modified unit axiom (albeit not forming a CohFT):
	\begin{equation}
		\Upsilon^r_{g,n+1}(\mathbb{v} \otimes v_{r-1}  )
		=
		\psi_{n+1} \cdot p^* \Upsilon^r_{g,n}(\mathbb{v})
	\end{equation}
	for an arbitrary element $\mathbb{v} \in W^{\otimes n}$. This modified unit  property for $ \Upsilon^r $ will be used often in the following section.
\end{remark}

\begin{remark}\label{rem:generals:def}
	In this remark, we discuss an extension of our work to variations of the Theta class. As mentioned before, the Chiodo class $C^{r,s}$ is the (pushforward along the forgetful map of the) total Chern class of an honest vector bundle for any $s$ satisfying $-r+1 \leq s \leq -1 $. In this paper, we are always working with the case $ s = -1 $. However, it makes perfect sense to consider the following generalisations of the Theta class for any $ -r+1 \leq s \leq -1 $. First define
	\begin{equation}
		\Upsilon^{r,s}_{g,n}(v_{a_1} \otimes \cdots \otimes v_{a_n})
		\coloneqq
		s^{-(2g-2+n)} r^{\frac{-(2g-2+n)s+|a|}{r}} f_* c_{\textup{top}} ( \mathcal{V}^{r,s}_{g;a})
	\end{equation}
	for $ 0 \leq a_i \leq r-1$, and then consider its restriction to the vector space where the $ a_i $ are restricted to be $ 1 \leq a_i \leq r-1 $ just as in \cref{def:Theta}. Schematically, $\Theta^{r,s}_{g,n} \coloneqq \Upsilon^{r,s}_{g,n}|_{a_i \neq 0}\,$. Analogous to the definition of the deformed Theta class, we can  define the deformed $ \Theta^{r,s} $-classes which we denote by $ \Theta^{r,s, \epsilon} $. Then the proof of \cref{prop:cohft} goes through without any modification, and thus  $ \Theta^{r,s, \epsilon} $ is a CohFT. The proof of \cref{thm:theta:cohft} goes through as well with one minor difference -- the $ \Theta^{r,s} $-classes form a $ (r-1) $ dimensional CohFT with the modified unit being $ v_{s+r} $.
\end{remark}

%–––––––––––––––––––––––––––––––––––––––––––%
\section{Givental--Teleman reconstruction}\label{sec:Teleman}
%–––––––––––––––––––––––––––––––––––––––––––%

In this section, we want to study  the deformed Theta CohFT further using the techniques of the Givental--Teleman reconstruction theorem. We will show that the deformed Theta class is semisimple and use the Givental--Teleman reconstruction theorem in order to find an expression in terms of tautological classes. In addition, this method will yield vanishing relations in the tautological ring.

%–––––––––––––––––––––––––––––––––––––––––––%
\subsection{Dubrovin--Frobenius manifolds} \label{sec:DFmanifolds}
%–––––––––––––––––––––––––––––––––––––––––––%

We start with some generalities on \emph{Dubrovin--Frobenius manifolds} and \emph{potentials} \cite{Dub96}. We stress that our Dubrovin--Frobenius manifolds are not assumed to have a flat unit vector field. The Dubrovin--Frobenius manifolds we are interested in will have a unit vector field by construction (as they will come from a CohFT), but this will not be flat with respect to the metric. This is precisely because the CohFTs that we are interested in do not have a flat unit. 

Given a CohFT $\Omega_{g,n}$ on $V$, we can naturally endow a formal neighbourhood of the origin in $ V $ with the structure of a Dubrovin--Frobenius manifold by restricting to genus $0$. Under certain convergence assumptions on the genus zero part of the CohFT $\Omega_{0,n}$ (see \cite{Jan18} for a related general discussion in the context of CohFTs with flat unit), we can equip a neighbourhood of $ V $ with the structure of a Dubrovin--Frobenius manifold. To be precise, assume that we have a $ d $-dimensional vector space $ V $ with the flat basis $ v_1, \ldots, v_d $ underlying our CohFT $ \Omega_{g,n} $. Then, we can define the potential $ F $ as 
\begin{equation}\label{eqn:defF}
	F (t_1, \ldots, t_d)
	=
	\sum_{\substack{k_1+ \cdots + k_d = n \\ n\geq 3}}
		\left( \int_{\overline{\mathcal M}_{0,n}} \Omega_{0,n} (v_1^{\otimes k_1} \otimes \cdots \otimes v_d^{\otimes k_d} ) \right)
		\prod_{i=1}^d \frac{t_i^{k_i}}{k_i!} \, ,
\end{equation}
where we view $ t_i $ as the dual coordinate to the basis element $ v_i $. If the sum in \cref{eqn:defF} converges in a domain $ U \subset V $, $ U $ inherits the structure of a Dubrovin--Frobenius manifold with flat coordinates $ (t_1, \ldots, t_d) $. All the information of this Dubrovin--Frobenius manifold is encoded in the potential $F (t_1, \ldots, t_d) $.

We can equip the tangent space at every point $ p $ on the Dubrovin--Frobenius manifold $ V $ with an associative algebra structure given by the \textit{quantum product}, 
\[
	\de_{i} \bullet_p \de_{j}
	=
	\sum_{k,\ell =1}^d \left.\left( \frac{\partial^3 F}{\partial t_i \partial t_j \partial t_k } \right)\right|_{p} \eta^{k,\ell} \de_{\ell}\,,
\]
where we introduced the following notation for the vector fields, $\de_a \coloneqq \frac{\partial}{\partial t_a} \in H^0(U, TU) $. Here, we also see that all terms of total degree $ < 3 $ are irrelevant in the potential, and we can drop them. In the sequel, we will only consider the potential up to these lower degree terms.

We can often equip the Dubrovin--Frobenius manifold $ U $ with an additional grading using the notion of an Euler field. An \textit{Euler field} on a Dubrovin--Frobenius manifold $ U $ with flat coordinates $( t_i )_{i \in [d]}$, is an affine vector field $ E $ satisfying the following conditions.
\begin{itemize}
	\item The vector field $ E $ has the form $E = \sum_i (\alpha_i t_i  + \beta_i ) \de_i\,$.

	\item
	The metric $ \eta $ and the quantum product $ \bullet $ are eigenfunctions of the Lie derivative $\mc{L}_E $ with weights $ 2- \delta $ and $1$ respectively, where $ \delta $ is a rational number called the \textit{conformal dimension}.
\end{itemize}

The Euler field $ E $ on the Dubrovin--Frobenius manifold $ U $ can be used to define an action of $ E $ on the CohFT $ \Omega_{g,n} $ as follows:
\begin{equation}\label{eqn:Eactiongen}
\begin{split}
	(E.\Omega)_{g,n} \left(\de_{a_1} \otimes \cdots \otimes \de_{a_n}\right)
	& \coloneqq
	\left( \deg + \sum_{l=1}^n \alpha_{a_l} \right) \Omega_{g,n} \left(\de_{a_1} \otimes \cdots \otimes \de_{a_n}\right) \\
	& \qquad +
	p_* \Omega_{g,n+1} \left(\de_{a_1} \otimes  \cdots \otimes \de_{a_n} \otimes \sum_i \beta_i \de_i \right).
\end{split}
\end{equation}

\begin{definition}
	We say that the CohFT $ \Omega $ is \textit{homogeneous} if there exists an Euler field $E$ such that
	\begin{equation}
		(E.\Omega)_{g,n} = \bigl( (g-1)\delta + n \bigr) \Omega_{g,n}\,.
	\end{equation}
\end{definition}

After this brief digression, we return to the CohFT we are interested in -- the deformed Theta class $\Theta^{r,\epsilon}$. As we have already mentioned, all the information of the associated Dubrovin--Frobenius manifold is encoded in the potential. Using the definition of the deformed Theta class $ \Theta^{r,\epsilon} $, the potential can be expressed as
\begin{equation}\label{eqn:Frob}
	F^\epsilon(t_1, \ldots, t_{r-1})
	=
	\sum_{\substack{m + k_1+ \cdots + k_{r-1}  = n \\ n\geq 3}}
	\left(
		\int_{\overline{\mathcal M}_{0,n }}
		\Upsilon^r_{0,n } (v_0^{\otimes m} \otimes v_1^{\otimes k_1} \otimes \cdots \otimes v_{r-1}^{\otimes k_{r-1}} )
	\right)
	\frac{\epsilon^m}{m!} \prod_{a=1}^{r-1} \frac{t_a^{k_a}}{k_a!}\, .
\end{equation}
We note that, for any term in the above sum to be non-zero, the degree of the class has to coincide with the dimension of the moduli space: $n - 3$. Equivalently, this can be written as the following degree condition:
\begin{equation}\label{eqn:degc}
	(r-1)m + \sum_{a=1}^{r-1} (r-1-a) k_a = 2r-2\, .
\end{equation}
Although we do not have a closed form for the potential for any $ r $, we can show that the  potential converges in the neighbourhood $ U $ of $ 0 $ in $ V $ defined as $U \coloneqq \{ (t_1,\dots,t_{r-1}) \mid  |t_{r-1}| < 1 \} \subset V\,$.

\begin{lemma}
	The potential $ F^\epsilon(t_1, \ldots, t_{r-1}) $ associated to the deformed Theta class $ \Theta^{r,\epsilon} $ defines a Dubrovin--Frobenius manifold structure on $ U $. 
\end{lemma}

\begin{proof}
	The degree condition \labelcref{eqn:degc} forces every $ k_a $ with $a = 1,\ldots,r-2$ to be bounded by a number that is independent of $n$, so that the potential is a polynomial in $ t_1, \ldots, t_{r-2} $. Thus, we only need to understand the convergence properties of the potential in $t_{r-1}$. Let us consider a term in the potential \labelcref{eqn:Frob} with fixed $(m,k_1, \ldots, k_{r-2} )$ satisfying \cref{eqn:degc}, and denote $\mathbb v \coloneqq v_0^{\otimes m} \otimes v_1^{\otimes k_1} \otimes \cdots \otimes v_{r-2}^{\otimes k_{r-2}}$ and $ n'= m + \sum_{a=1}^{r-2} k_a $. As the degree condition is independent of $k_{r-1}$ we need to analyse the following series
	\begin{equation*}
		\sum_{k_{r-1}}
		\left(
			\int_{\overline{\mathcal M}_{0,n' + k_{r-1} }}
			\Upsilon^r_{0,n' + k_{r-1} } ( \mathbb v \otimes v_{r-1}^{\otimes k_{r-1}})
		\right)
		\frac{\epsilon^m}{m!} \prod_{a=1}^{r-1} \frac{t_a^{k_a}}{k_a!}\, .
	\end{equation*}
	We can now repeatedly apply the modified unit axiom for $\Upsilon^r$ (cf. \cref{rem:Upsilon:modunit}) and the projection formula, together with $p_{\ast} \psi_{n+1} = \kappa_0$, to find
	\begin{equation*}
		\sum_{k_{r-1}}
		\left(
			\int_{\overline{\mathcal M}_{0,n'}}
			\Upsilon^r_{0, n'} ( \mathbb v )
		\right)
		\frac{(n' + k_{r-1}-3)!}{(n'-3)!}
		\frac{\epsilon^m}{m!}
		\prod_{a=1}^{r-1} \frac{t_a^{k_a}}{k_a!}\, .
	\end{equation*}
	By applying the ratio test for convergence, we see that the above series is absolutely convergent for $ |t_{r-1}| < 1 $.
\end{proof}

\begin{remark}
	We compute the potential explicitly in the case of $ r = 2 $ and $ r=3 $, and present the results. The proofs are completely straightforward calculations. For $ r=2 $, we have
	\[
		F^{\epsilon}(t_1) = \frac{\epsilon^2 }{2}\log(1 - t_1)
	\]
	and for $ r=3 $, we have 
	\[
		F^{\epsilon}(t_1,t_2) =  - \frac{t_1^4}{12(1-t_2)^2} - \frac{\epsilon t_1^2}{2(1-t_2)} + \frac{\epsilon^2 }{2} \log(1-t_2)\,.
	\]
	We also note that the potential for general $ r $ always contains the term $\frac{\epsilon^2 }{2}\log(1 - t_{r-1})$ as a summand.
\end{remark}

Our next goal is  to compute the quantum product of $ \Theta^{r,\epsilon} $. This requires only the terms of degree 3 in $ t_1, \ldots, t_{r-1} $ in the potential and we calculate the required integrals below.

\begin{lemma}
	Assume that $0 < a,b,c \leq r-1$. Then, we have the following values for the integrals of Chiodo classes:
	\begin{equation}\label{eqn:Chiodo:int}
	\begin{aligned}
		\int_{\Mbar_{0,3}} \Upsilon^r_{0,3}(v_a \otimes v_b \otimes v_c) = - \delta_{ a+b+c,r-1} \, ,\\
		\int_{\Mbar_{0,4}} \Upsilon^r_{0,4}(v_0 \otimes v_a \otimes v_b \otimes v_c) = - \delta_{a+b+c,2r-2} \, ,\\
		\int_{\Mbar_{0,5}} \Upsilon^r_{0,5}(v_0^{\otimes 2} \otimes v_a \otimes v_b \otimes v_c) = -2 \cdot \delta_{a,b,c,r-1} \, .
	\end{aligned}
	\end{equation}
	In addition, any integral for $n \geq 0$ with at least $3$ insertions of $v_0$ vanishes:
	\begin{equation}\label{eqn:Chiodo:vanish}
		\int_{\Mbar_{0,3+n}} \Upsilon^r_{0,3+n}(v_0^{\otimes 3} \otimes \mathbb v) = 0
	\end{equation}
	for any $n \geq 0$ and any $ \mathbb v \in W^{\otimes n} $.
\end{lemma}

\begin{proof}
	In all three cases that we consider in \cref{eqn:Chiodo:int}, the degree condition \labelcref{eqn:degc} gives the Kronecker delta conditions. Thus, we only need to compute the values of the Chiodo classes in those cases. We can do so using the formula derived in \cite[corollary~4]{JPPZ17} from Chiodo's results \cite{Chi08+}. For $ n =3 $, with $ a+b+c = r-1 $, we get $C^{r,-1}_{0,3}(a,b,c) = 1/r$. When $n = 4$, and $a + b + c = 2r-2$,  the value of the Chiodo integral is
	\begin{equation*}
		\int_{\Mbar_{0,4}} C^{r,-1}_{0,4}(0,a,b,c) = - \frac{1}{r^2}\, .
	\end{equation*}
	When $n=5$ with two zero insertions, we get $a+b+c = 3r-3$, which implies that $a=b=c=r-1$, and a repeated application of \cref{eqn:modunit} shows that
	\begin{equation*}
		\int_{\Mbar_{0,5}} C^{r,-1}_{0,5}(0,0,r-1,r-1,r-1) = \frac{2}{r^3}\,.
	\end{equation*}
	The degree calculation \labelcref{eqn:degc} shows that as soon as we have at least three insertions of $v_0$, the integral in \cref{eqn:Chiodo:vanish} vanishes.
\end{proof}

It is straightforward to see that the terms computed in the above lemma are the only ones that contribute to triple derivatives of $ F^{\epsilon}(t_1, \ldots, t_{r-1}) $ at  the origin $ 0  \in  U $. Thus, we can compute  $ \Theta^{r,\epsilon}_{0,3} $ as 
\begin{equation}
	\left.\left( \frac{\partial^3 F^\epsilon}{\partial t_a \partial t_b \partial t_c } \right)\right|_{p = 0}
	=
	\Theta^{r,\epsilon}_{0,3}(v_a \otimes v_b \otimes v_c)
	=
	\begin{cases}
		-1 \,, & a+b+c = r-1\,, \\
		-\epsilon\,, & a+b+c = 2r-2\,, \\
		-\epsilon^2\,, &  a=b=c = r-1\,.
	\end{cases}
\end{equation}
Consequently, the quantum product at the origin of $ U $, and thus for the deformed Theta class $ \Theta^{r,\epsilon} $, is 
\begin{equation}\label{eqn:quantum:prod}
	v_a \bullet v_b
	=
	\begin{cases}
		-v_{a+b+1}\,, & 2\leq a+b< r-1\,, \\
		-\epsilon v_{a+b+2-r}\,, & r-1 \leq a+b< 2r-2\,, \\
		-\epsilon^2 v_{1}\,, &  a+b  = 2r-2\,.
	\end{cases}
\end{equation}

\begin{proposition}
	The CohFT $\Theta^{r,\epsilon}$ is semisimple if and only if $\epsilon \ne 0$. 
\end{proposition}

\begin{proof}
	From the expression \labelcref{eqn:quantum:prod} of the quantum product, we can see that for $\epsilon = 0$ we have $v_{r-1}^2 = 0$. In particular, the CohFT at $\epsilon = 0$ has nilpotents, hence it is not semisimple. 
	
	For $\epsilon \ne 0$, one can easily check that the following constitutes a basis of normalised idempotents:
	\[
		e_k = - \frac{1}{r-1} \sum_{a=1}^{r-1} \theta^{-k(a+1)} \epsilon^{\frac{a+1}{1-r}} v_a\,,
		\qquad
		k=1,\dots,r-1\,,
	\]
	where $\theta = e^{ \frac{2\pi\iu}{r-1} }$\,.
\end{proof}

As the Theta class is of pure degree, we expect that the Dubrovin--Frobenius manifold associated to it admits an Euler field. Indeed, we have the following result for the Dubrovin--Frobenius manifold $ U $ associated to the deformed Theta class.

\begin{proposition}
	The vector field 
	\begin{equation}
			E \coloneqq \frac{r-1}{r} \de_{r-1}  - \sum_{a=1}^{r-1} \frac{a}{r} t_a \de_a 
	\end{equation}
	is an Euler field for the Dubrovin--Frobenius manifold $ U $, with conformal dimension $\delta = 3$.
\end{proposition}

\begin{proof} 
	We start by checking that the vector field $E$ is conformal, i.e. the metric $\eta$ is an eigenfunction of $\mc{L}_E$ with eigenvalue $2 - \delta$. Indeed
	\[
		(\mc{L}_E \eta)(\de_a,\de_b) = E\bigl( \eta(\de_a,\de_b) \bigr) - \eta\bigl( [E,\de_a],\de_b \bigr) - \eta\bigl( \de_a,[E,\de_b] \bigr)\,.
	\]
	The first term vanishes, while the last two terms can be simplified using $[E,\de_k] = \frac{k}{r} \de_k$. Thus, we find
	\[
		(\mc{L}_E \eta)(\de_a,\de_b) = - \frac{a+b}{r} \eta(\de_a,\de_b) = - \eta(\de_a,\de_b)\,,
	\]
	as $\eta(\de_a,\de_b) = \delta_{a+b,r}$. In particular, the vector field $E$ is conformal, with conformal dimension $\delta = 3$.

	To check the Euler property, i.e.~that the quantum product $\bullet$ is an eigenfunction of $\mc{L}_E$ with eigenvalue $1$, we can use \cite[proposition 2.2.2]{Man99}: $ E $ satisfying the Euler property is equivalent to $E.F^\epsilon = 0$ (up to terms of degree $< 3$). For simplicity of notation, let us denote $\mathbb{v} = v_0^{\otimes m} \otimes v_1^{\otimes k_1} \otimes \ldots \otimes  v_{r-1}^{\otimes k_{r-1}}$, and assume that $m +|k| = m + \sum_{a=1}^{r-1} k_a = n$. Notice that for $ \int_{\overline{\mathcal M}_{0,n}} \Upsilon^r_{0,n} (\mathbb{v}) $ to be non-zero, the degree condition \labelcref{eqn:degc} can be equivalently expressed as
	\[
		\sum_{a=1}^{r-1} a k_{a} = (r-1)(n-2)\, .
	\]
	Then, due to the generalised Euler formula for quasi-homogeneous polynomials, we have
	\begin{equation*}
		\sum_{a=1}^{r-1} \frac{a}{r} t_a \de_a
		\left( \int_{\overline{\mathcal M}_{0,n}} \Upsilon^r_{0,n} (\mathbb{v}) \right)
		\frac{\epsilon^{m}}{m!} \prod_{a=1}^{r-1} \frac{t_a^{k_a}}{k_a!}
		=
		\frac{(r-1)}{r} (n-2)
		\left( \int_{\overline{\mathcal M}_{0,n}} \Upsilon^r_{0,n} (\mathbb{v}) \right)
		\frac{\epsilon^{m}}{m!} \prod_{a=1}^{r-1} \frac{t_a^{k_a}}{k_a!}\,.
	\end{equation*}
	Thus, in order to prove that $ E.F^{\epsilon} = 0 $, we reduce to proving the following equation:
	\begin{equation}\label{eqn:Er-1}
		\sum_{\substack{ n\geq 3 \\ m + |k|=n }}
			\frac{\partial}{\partial t_{r-1}}
				\left( \int_{\overline{\mathcal M}_{0,n}} \Upsilon^r_{0,n} ( \mathbb{v}) \right)
				\frac{\epsilon^{m}}{m!}
				\prod_{a=1}^{r-1} \frac{t_a^{k_a}}{k_a!}
		=
		\sum_{\substack{ n\geq 3 \\  m +|k|=n }}
			(n-2) \left( \int_{\overline{\mathcal M}_{0,n}} \Upsilon^r_{0,n} (\mathbb{v} ) \right)
			\frac{\epsilon^{m}}{m!}
			\prod_{a=1}^{r-1} \frac{t_a^{k_a}}{k_a!}\, .
	\end{equation}
	Let us start with the left-hand side of the above equation, and calculate the derivative:
	\begin{equation*}
	\begin{split}
		\sum_{\substack{ n\geq 3 \\  m +|k|=n }} \frac{\partial}{\partial t_{r-1}}
		\left( \int_{\overline{\mathcal M}_{0,n}} \Upsilon^r_{0,n} (\mathbb{v}) \right)
			\frac{\epsilon^{m}}{m!}
			\prod_{a=1}^{r-1} \frac{t_a^{k_a}}{k_a!}
		& =
		\sum_{\substack{ n\geq 3 \\ |k| = n \\ k_{r-1} \geq 1 \\}}
			\left( \int_{\overline{\mathcal M}_{0,n}} \Upsilon^r_{0,n} (\mathbb{v} ) \right)
				\frac{\epsilon^{m}}{m!}
				\frac{t_{r-1}^{k_{r-1}-1}}{(k_{r-1}-1)!}
				\prod_{a=1}^{r-2} \frac{t_a^{k_a}}{k_a!} \\
		& =
		\sum_{\substack{ n\geq 3 \\ m +|k| = n }}
			\left( \int_{\overline{\mathcal M}_{0,n+1}} \Upsilon^r_{0,n+1} (\mathbb{v} \otimes  v_{r-1} ) \right)
				\frac{\epsilon^{m}}{m!}
				\frac{t_{r-1}^{k_{r-1}}}{k_{r-1}!}
				\prod_{a=1}^{r-2} \frac{t_a^{k_a}}{k_a!} \\
		& =
		\sum_{\substack{ n\geq 3 \\ m+ |k| = n }}
			(n-2) \left( \int_{\overline{\mathcal M}_{0,n}} \Upsilon^r_{0,n} (\mathbb{v} ) \right)
				\frac{\epsilon^{m}}{m!}
				\prod_{a=1}^{r-1} \frac{t_a^{k_a}}{k_a!}\,. 
	\end{split}
	\end{equation*}
	In the last equality, we applied the modified unit axiom and the projection formula, together with $p_{\ast} \psi_{n+1} = (n-2)$. Thus we have the result.
\end{proof}

\begin{remark}\label{rem:nss}
	In this remark, we note that the calculation above proves that the Dubrovin--Frobenius manifold $ U $ at $ \epsilon = 0 $ is \textit{not} generically semisimple. We assume that $ \epsilon  = 0 $ throughout this remark and claim that at any point $ p \in U $, the  element $ v_{\textup{nil}} $
	\begin{equation*}
		v_{\textup{nil}}
		\coloneqq
		\sum_{a=1}^{r-1} t_a \de_a - \de_{r-1} \,  \in H^0(U, TU)
	\end{equation*}
	is nilpotent. When $ \epsilon = 0 $, \cref{eqn:Er-1} reduces to the following equation 
	\begin{equation*}
		\sum_{\substack{ n\geq 3 \\ |k|=n }}
		\frac{\partial}{\partial t_{r-1}}
		\left( \int_{\overline{\mathcal M}_{0,n}} \Upsilon^r_{0,n} ( \mathbb{v}) \right)
		\prod_{a=1}^{r-1} \frac{t_a^{k_a}}{k_a!}
		=
		\sum_{\substack{ n\geq 3 \\|k|=n }}
		(n-2) \left( \int_{\overline{\mathcal M}_{0,n}} \Upsilon^r_{0,n} (\mathbb{v} ) \right)
		\prod_{a=1}^{r-1} \frac{t_a^{k_a}}{k_a!}\,,
	\end{equation*}
	where $ \mathbb v  = v_1^{\otimes k_1} \otimes \ldots \otimes  v_{r-1}^{\otimes k_{r-1}} $. The right-hand side can be further simplified by using $ |k| = n $ and the degree condition \cref{eqn:degc} (with $ m=0 $) in order to get 
	\begin{equation*}
		\sum_{\substack{ n\geq 3 \\|k|=n }}
		(|k|-2) \left( \int_{\overline{\mathcal M}_{0,n}} \Upsilon^r_{0,n} (\mathbb{v} ) \right)
		\prod_{a=1}^{r-1} \frac{t_a^{k_a}}{k_a!} =  \sum_{a=1}^{r} ( t_a \de_a-2)\Bigg( \sum_{\substack{ n\geq 3 \\|k|=n }}
		 \left( \int_{\overline{\mathcal M}_{0,n}} \Upsilon^r_{0,n} (\mathbb{v} ) \right)
		\prod_{a=1}^{r-1} \frac{t_a^{k_a}}{k_a!} \Bigg).
	\end{equation*}
	Thus, we finally get $\de_{r-1} F^0 = \left(t_a \de_{a} - 2\right) F^0$. By successively differentiating the above equation with respect to $ t_b $ and $ t_c $, and then multiplying on the right by $ \de_{r-c} $ and summing over all $ 1 \leq c \leq r-1  $, we get $v_{\textup{nil}} \bullet \de_b  = 0$ for all $ 1 \leq b \leq r-1 $. Consequently, $ v_{\textup{nil}}\bullet v_{\textup{nil}} = 0 $.
\end{remark}

This Euler field $E$ also makes the deformed Theta class a homogeneous CohFT.

\begin{proposition}
	The deformed Theta class $ \Theta^{r,\epsilon}$ is a homogeneous CohFT.
\end{proposition}

\begin{proof}
	We need to calculate the action of the Euler field on $ \Theta^{r,\epsilon}_{g,n} $. Let us consider a specific summand of it, say $ p_{m,*} \Upsilon^r_{g,n+m}(v_{a_1}\otimes  \cdots\otimes v_{a_n} \otimes v_0^{\otimes m}) $. Again, let us use the notation $ \mathbb{v} = v_{a_1}\otimes  \cdots\otimes v_{a_n} $ in order to keep the formulas readable. Up to permutation, we can rewrite it as $\mathbb{v} = v_{1}^{\otimes k_1} \otimes  \cdots\otimes v_{r-1}^{\otimes k_{r-1}}$. Then, we have
	\begin{equation*}
	\begin{split}
		E. \bigl( p_{m,*}
		&
		\Upsilon^r_{g,n+m}(\mathbb{v} \otimes v_0^{\otimes m}) \bigr) \\
		& =
		\left(\deg - \sum_{a=1}^{r-1} \frac{a}{r} k_a\right)
			p_{m,*} \Upsilon^r_{g,n+m} (\mathbb{v} \otimes v_0^{\otimes m})
		+
		\frac{r-1}{r} p_* p_{m,*} \Upsilon^r_{g,n+m+1}(\mathbb{v} \otimes v_0^{\otimes m}\otimes v_{r-1}) \\ 
		& =
		\biggl(	\frac{(r+2)(g - 1) + n + m + |a|}{r} - m - \sum_{i=1}^{r-1} \frac{a}{r} k_a \\
		&\qquad\qquad\qquad\qquad\qquad
			+ \frac{r-1}{r}(2g-2+n+m) \biggr)
		p_{m,*} \Upsilon^r_{g,n+m} (\mathbb{v} \otimes v_0^{\otimes m})  \\ 
		& =
		( 3g - 3 + n )
		p_{m,*} \Upsilon^r_{g,n+m}(\mathbb{v} \otimes  v_0^{\otimes m})\,.
	\end{split}
	\end{equation*}
	Here we used once again the modified unit axiom and the projection formula. Multiplying by $\frac{\epsilon^m}{m!}$ and summing over all $m \geq 0$, we get the action on the deformed Theta class, which proves the proposition.
\end{proof}

%–––––––––––––––––––––––––––––––––––––––––––%
\subsection{Reconstruction and tautological relations}
%–––––––––––––––––––––––––––––––––––––––––––%

In \cite{Giv01} Givental defined certain actions on Gromov--Witten potentials by $R$-matrices and translations, and these actions were lifted to cohomological field theories in the work of Teleman \cite{Tel12}. A careful proof that the resulting collection of cohomology classes satisfies the cohomological field theory axioms can be found in \cite{PPZ15}. A description of the orbit structure was given by Teleman in the specific case of homogeneous semisimple CohFTs. Here we recall the basic definitions.

%––––––––––––––––––––––––––––––––––––––––––––
\subsubsection{$R$-matrix action.}
%––––––––––––––––––––––––––––––––––––––––––––
Fix a vector space $V$ with a symmetric bilinear form $\eta$. An \emph{$R$-matrix} is an $\End(V)$-valued power series that is the identity in degree $0$
\begin{equation}
	R(u) = \Id + \sum_{k \ge 1} R_k u^k\,,
	\qquad
	R_k \in \End(V)\,,
\end{equation}
and satisfying the symplectic condition
\begin{equation}
	R(u) R^{\dag}(-u) = \Id\,.
\end{equation}
Here $R^{\dag}$ is the adjoint with respect to $\eta$. The inverse matrix $R^{-1}(u)$ also satisfies the symplectic condition. In particular, we can consider the $V^{\otimes 2}$-valued power series\footnote{The reason why we use $R^{-1}$ instead of $R$ is  that it defines a left action on the set of CohFTs. Beware that some authors use a different notation.}
\begin{equation}
	E(u,v)
	=
	\frac{\Id \otimes \Id - R^{-1}(u) \otimes R^{-1}(v)}{u + v} \eta^{\dag} \in V^{\otimes 2}\bbraket{u,v}\,.
\end{equation}

\begin{definition}
	Consider a CohFT $\Omega$ on $(V,\eta)$, together with an $R$-matrix. We define a collection of cohomology classes
	\begin{equation}
		R\Omega_{g,n} \in H^{\bullet}(\Mbar_{g,n}) \otimes (V^{*})^{\otimes n}
	\end{equation}
	as follows. Let $\ms{G}_{g,n}$ be the finite set of stable graphs of genus $g$ with $n$ legs (cf. \cite{PPZ15} for the definition and the notation). For each $\Gamma \in \ms{G}_{g,n}$, define a contribution $\mathrm{Cont}_{\Gamma} \in H^{\bullet}(\Mbar_{\Gamma}) \otimes (V^{*})^{\otimes n}$ by the following construction:
	\begin{itemize}
		\item place $\Omega_{g(v),n(v)}$ at each vertex $v$ of $\Gamma$,

		\item place $R^{-1}(\psi_{i})$ at each leg $i$ of $\Gamma$,

		\item at every edge $e = (h,h')$ of $\Gamma$, place $E(\psi_{h},\psi_{h'})$.
	\end{itemize}
	Define $R\Omega_{g,n}$ to be the sum of contributions of all stable graphs, after pushforward to the moduli space weighted by automorphism factors:
	\begin{equation}
		R\Omega_{g,n} = \sum_{\Gamma \in \ms{G}_{g,n}} \frac{1}{|\Aut{(\Gamma)}|} \, \xi_{\Gamma,\ast}\mathrm{Cont}_{\Gamma}\,.
	\end{equation}
\end{definition}

\begin{proposition}
	The data $R\Omega = (R\Omega_{g,n})_{2g-2+n > 0}$ form a CohFT on $(V,\eta)$. Moreover, the $R$-matrix action on CohFTs is a left group action.
\end{proposition}

%––––––––––––––––––––––––––––––––––––––––––––
\subsubsection{Translations}
\label{subsubsec:translation}
%––––––––––––––––––––––––––––––––––––––––––––

There is also another action on the space of CohFTs: a \emph{translation} is a $V$-valued power series vanishing in degree $0$ and $1$:
\begin{equation}
	T(u) = \sum_{d \ge 1} T_d u^{d+1}\,,
	\qquad
	T_d \in V\,.
\end{equation}

\begin{definition}
	Consider a CohFT $\Omega$ on $(V,\eta)$, together with a translation $T$. We define a collection of cohomology classes
	\begin{equation}
		T\Omega_{g,n} \in H^{\bullet}(\Mbar_{g,n}) \otimes (V^{*})^{\otimes n}
	\end{equation}
	by setting
	\begin{equation}
		T\Omega_{g,n}(v_1 \otimes \cdots \otimes v_n)
		=
		\sum_{m \ge 0} \frac{1}{m!} p_{m,\ast} \Omega_{g,n+m} \bigl(
			v_1 \otimes \cdots \otimes v_n \otimes
			T(\psi_{n+1}) \otimes \cdots \otimes T(\psi_{n+m})
		\bigr)\,.
	\end{equation}
	Here $p_{m} \colon \Mbar_{g,n+m} \to \Mbar_{g,n}$ is the map forgetting the last $m$ marked points. Notice that the vanishing of $T$ in degree $0$ and $1$ ensures that the above sum is actually finite.
\end{definition}

\begin{proposition}
	The data $T\Omega = (T\Omega_{g,n})_{2g-2+n > 0}$ form a CohFT on $(V,\eta)$. Moreover, translations form an abelian group action on CohFTs.
\end{proposition}

%––––––––––––––––––––––––––––––––––––––––––––
\subsubsection{Teleman reconstruction theorem}
\label{subsubsec:Teleman}
%––––––––––––––––––––––––––––––––––––––––––––

Given a homogeneous semisimple CohFT, we can use Teleman's reconstruction theorem \cite{Tel12} to determine the higher genus part starting from the genus zero part (the topological field theory). More precisely, the theorem states the following.

\begin{theorem}[\cite{Tel12}]\label{thm:Teleman}
	Let $\Omega_{0,n}$ be a genus zero homogeneous semisimple CohFT. Then:
	\begin{itemize}
		\item When $ n > 0 $, there exists a unique homogeneous CohFT $\Omega_{g,n}$ that extends $\Omega_{0,n}$ in higher genera.
		
		\item The extended CohFT is obtained by first applying a translation action followed by an $R$-matrix action on the topological field theory $ w_{g,n} $ (determined by $\Omega_{0,3}$).
		
		\item The translation and the $R$-matrix are uniquely specified by the associated Dubrovin--Frobenius manifold and the Euler field.
	\end{itemize}
\end{theorem}

\begin{remark}
	We note here that the usual (re)-statement of Teleman's reconstruction theorem (for instance in \cite{PPZ15}), does not impose the restriction $ n > 0 $ on the number of marked points for the reconstructed CohFT. However, this is a small, but crucial, caveat in the reconstruction theorem, as pointed out in \cite[remark~3.2]{Tel12}.
\end{remark}

Let us give a summary of Teleman's reconstruction theorem using our notations and conventions. As $ \Omega_{0,n} $ is a CohFT on $ V $,  there is a (possibly formal) neighbourhood $U$ of the origin in $V$ that admits a Dubrovin--Frobenius manifold structure. This Dubrovin--Frobenius manifold admits a (possibly non-flat) unit vector field, denoted $\mathbf{1}$. Moreover, we assume that there is an Euler field $ E $. Let $ v $ be a tangent vector at a point $p \in U$ on this Dubrovin--Frobenius manifold. Then, we define the \textit{Hodge grading operator} $ \mu \in \End(T_p U) $ as
\begin{equation}\label{eqn:mudef}
	\mu (v) \coloneqq [E,v] + \left(1-\frac{\delta}{2}\right)v \,.
\end{equation}
We also define the operator $\phi \in \End(T_p U)$ of \textit{quantum multiplication by the Euler field} as
\begin{equation}
	\phi(v) \coloneqq E \bullet_p v\, .
\end{equation}
The matrices $ R_m $ for $ m \geq 0 $ satisfy the following equation
\begin{equation}\label{eqn:Rdef}
	[R_{m+1}, \phi ] = (m+\mu) R_m \,.
\end{equation}
At a semisimple point of this Dubrovin--Frobenius manifold $ U $, the above equation determines the $ R $-matrix uniquely starting with the initial condition $ R_0 = \operatorname{Id} $. In order to obtain the $ R $-matrix of \cref{thm:Teleman}, we will work at the origin which is semisimple by assumption.

The other piece of data required for reconstruction is called the \textit{vacuum vector} of the theory, that is a vector field-valued formal power series in $u$ denoted $ \mathbf{v}(u) $. The vacuum satisfies the following differential equation\footnote{We remark that the extra sign in \cref{eqn:vdef} is due to an opposite sign convention as compared to \cite{Tel12}.}
\begin{equation}\label{eqn:vdef}
	\frac{ d \mathbf{v}(u)}{du} + \frac{\mu + \delta/2}{u}{\mathbf{v}(u)} = - \frac{\phi}{u^2} \left(\mathbf{v}(u) - \mathbf{1}\right), 
\end{equation}
where $ \mathbf{1} $ is the unit vector field. Again, the above equation determines the vacuum uniquely at a semisimple point on the Dubrovin--Frobenius manifold. Using the vacuum and the $R$-matrix computed at the origin of the Dubrovin--Frobenius manifold $U$, the translation $T(u)$ in \cref{thm:Teleman} is defined as 
\begin{equation}\label{eqn:Tv}
	T(u)
	=
	u \left( \mathbf{1} - R^{-1}(u) \mathbf{v}(u) \right) .
\end{equation}
Then, the Teleman reconstruction theorem states that for all $ n > 0 $, the CohFT $\Omega_{g,n}$ is given by
\begin{equation}
	\Omega_{g,n} = R T w_{g,n}\,.
\end{equation}

Finally, we address the reconstruction of $ \Omega_{g,0} $. As it turns out, Teleman's reconstruction can be extended by elementary arguments for $ \Omega_{g,0} $ in almost all degrees. To be precise, we have the following statement.

\begin{proposition}\label{prop:n=0}
	Let $(V,\eta,\Omega)$ be a semisimple, homogeneous CohFT (not necessarily with a flat unit). Let $\delta$ be its conformal dimension. Then Teleman's reconstruction theorem still holds for $n = 0$, except in degree $\delta(g-1)$.

	In addition, if the CohFT $(V,\eta,\Omega)$ has a flat unit, the reconstruction theorem for $ n = 0 $  holds without any degree restrictions.
\end{proposition}

\begin{proof}
	Let us denote the reconstructed CohFT by $ \widetilde{\Omega}_{g,n} \coloneqq R T w_{g,n}$. We already know due to \cref{thm:Teleman} that $\Omega_{g,1} = \widetilde{\Omega}_{g,1}$. Let $E_0$ be the Euler vector field at the origin of the Dubrovin--Frobenius manifold associated to $ \Omega $. The homogeneity property implies that
	\begin{equation*}
		\deg \Omega_{g,0} + p_{\ast} \Omega_{g,1}(E_0) = \delta(g-1) \, \Omega_{g,0}.
	\end{equation*}
	The reconstructed CohFT $\widetilde{\Omega}$ is also homogeneous by \cite[proposition 8.5]{Tel12}. Alternatively, the homogeneity can be proved directly using the expression of $\widetilde{\Omega}$ as a sum over stable graphs. Using homogeneity and equality in genus $1$, we get
	\begin{equation*}
		\bigl( \delta(g-1) - \deg \bigr) \, \Omega_{g,0}
		=
		p_{\ast} \Omega_{g,1}(E_0)
		=
		p_{\ast} \widetilde{\Omega}_{g,1}(E_0)
		=
		\bigl( \delta(g-1) - \deg \bigr) \, \widetilde{\Omega}_{g,0} \,,
	\end{equation*}
	and thus, $\Omega_{g,0} = \widetilde{\Omega}_{g,0}$, except in degree $\delta(g-1)$.
	
	Now, assume that the CohFT $ \Omega $ has a flat unit $ \mathbb 1 $.  We thank D. Zvonkine for pointing out the following simple proof. We have
	\begin{equation*}
		\frac{1}{2g-2}p_{\ast} \bigl( \psi_1 \cdot \Omega_{g,1}(\bm{1}) \bigr)
		=
		\frac{1}{2g-2} p_{\ast} \bigl( \psi_1 \cdot p^{\ast} \Omega_{g,0} \bigr)
		=
		 \, \Omega_{g,0}.
	\end{equation*}
	As the reconstructed CohFT $\widetilde{\Omega}$ is also a CohFT with flat unit (see \cite[proposition 2.12]{PPZ15}), the above equation holds for $ \widetilde{\Omega}$ as well, and we get that $ \Omega_{g,0} = \widetilde{\Omega}_{g,0} $ for all $ g \geq 2 $. 
\end{proof}

In addition, a simple adaptation of \cite[proposition 2.12]{PPZ15} shows that the pull-back with respect to the forgetful map of any reconstructed CohFT can be recovered through the vacuum vector.

\begin{lemma}\label{lem:pull-back}
	The reconstructed CohFT $ \widetilde{\Omega}_{g,n} = R T w_{g,n}$ satisfies the pull-back property
	\begin{equation}
		p^* \widetilde{\Omega}_{g,n}(v_1 \otimes \cdots v_n)
		=
		\widetilde{\Omega}_{g,n+1}(v_1 \otimes \cdots v_n \otimes \bm{v}(\psi_{n+1}))
	\end{equation}
	for all stable $(g,n)$, including $n = 0$.
\end{lemma}

%––––––––––––––––––––––––––––––––––––––––––––
\subsubsection{Reconstruction of the deformed Theta class}\label{sec:reconstructionTheta}
%––––––––––––––––––––––––––––––––––––––––––––

In order to compute the ingredients of the Givental--Teleman reconstruction procedure for the Dubrovin--Frobenius manifold associated to the deformed Theta class $ \Theta^{r,\epsilon}_{g,n} $, let us choose a slightly different basis of $V$:
\begin{equation}
	\hat{v}_a
	=
	- \epsilon^{\frac{a+1}{1-r}} v_a\,,\qquad \forall a \in [r-1]\,.
\end{equation}
Then, the metric and the quantum product at the origin become
\begin{equation}
	\eta (\hat{v}_a , \hat{v}_b )
	=
	\epsilon^{\frac{r+2}{1-r}} \delta_{a+b,r} \, ,
	\qquad\quad
	\hat{v}_a \bullet \hat{v}_b =
	\begin{cases}
		\hat{v}_{a+b+1}\, , & 2\leq a+b< r-1\, , \\
		\hat{v}_{a+b+2-r}\, , & r-1 \leq a+b< 2r-2\, , \\
		\hat{v}_{1}\, , &  a+b  = 2r-2\, .
	\end{cases}
\end{equation}
Now, let us compute the topological field theory $w_{g,n} \coloneqq \deg_0 \Theta^{r,\epsilon}_{g,n}$.

\begin{lemma}\label{lem:tft}
	The topological field theory $ w_{g,n} $ of the deformed Theta class is 
	\begin{equation}
		w_{g,n}(\hat{v}_{a_1} \otimes \cdots \otimes \hat{v}_{a_n})
		=
		\epsilon^{\frac{r+2}{r-1}(g-1) } (r-1)^g \cdot \delta\, ,
	\end{equation}
	where $\delta$ equals $1$ if $3g-3+n + |a| \equiv 0 \pmod{r-1}$ and $0$ otherwise.
\end{lemma}

\begin{proof}
	First we compute $ w_{0,3} $ in the basis $ \hat{v}_a $. This gives 
	\[
		w_{0,3}(\hat{v}_a \otimes \hat{v}_b \otimes \hat{v}_c)
		=
		\epsilon^{\frac{r+2}{1-r}} \delta_{a+b+c \equiv 0 \pmod{r-1}}\, .
	\]
	In order to compute the topological field theory, we can restrict to  a completely degenerate curve of type $ (g,n) $ -- this has $ 3g-3+n $ nodes and $ 2g-2+n $ rational components -- and then use the gluing axiom of the CohFT \cref{def:cohft}. On this degenerate curve, we need to place insertions $1 \leq  a_i \leq r-1 $ at every branch of every  node such that the following conditions are satisfied:
	\begin{itemize}
		\item The definition of the pairing implies that the sum of the insertions at the two branches of the node equal $ r $.
		
		\item The expression of $ w_{0,3} $ implies that the sum of the insertions on every rational component is $ 0 $ modulo $ r-1 $.
	\end{itemize}
	Such a placement is impossible unless $3g - 3 + n + |a| \equiv 0 \pmod{r-1}$. Let us assume this and let us choose an insertion for one branch of one node for every independent cycle of the dual graph of the curve, which in turn fixes all the other insertions uniquely. This choice contributes a factor of $ (r-1)^g $. To conclude, at every rational component we have a factor of $\epsilon^{\frac{r+2}{1-r}}$, and at every node the inverse of the pairing contributes a factor of $\epsilon^{\frac{r+2}{r-1}}$. Putting all these contributions together gives the result.
\end{proof}

We can calculate the Hodge grading operator on the basis $( \hat{v}_a )_{a \in[r-1]} $ of tangent vectors at the origin $(t_1,\dots,t_{r-1}) = (0,\dots,0)$ and the operator of quantum multiplication by the Euler field as
\begin{equation}\label{eqn:mu:phi}
	\mu
	=
	\frac{1}{2r}
	\begin{pmatrix}
		-(r-2) & 0 & \cdots & \cdots& 0 \\
		0 & -(r-4) & 0 & & 0 \\
		\vdots & \ddots & \ddots & \ddots & \vdots \\
		0 & & 0 & r-4 & 0 \\
		0 & \cdots & \cdots & 0 & r-2
	\end{pmatrix},
	\qquad
	\phi
	=
	- \frac{r-1}{r} \epsilon^{\frac{r}{r-1}}
	\begin{pmatrix}
		0& \cdots & \cdots & 0& 1 \\
		1 & 0 & 0 & \cdots & 0 \\
		0 & 1 & 0 & \cdots & 0 \\
		\vdots & \ddots & \ddots & \ddots & \vdots \\
		0 & \cdots & 0 & 1  & 0
	\end{pmatrix}.
\end{equation}

\begin{lemma}\label{lem:Rmatrix}
	The $ R $-matrix elements in the basis $(\hat{v}_1, \dots, \hat{v}_{r-1} )$ computed from the Teleman reconstruction theorem for the deformed Theta class are
	\begin{equation}
		(R^{-1}_m)_a^b = \frac{P_m(r,a-1)}{ ((1-r) \epsilon^{\frac{r}{r-1}})^m }\,,
		\qquad
		\text{if }\; b+m \equiv a \pmod{r-1}\, ,
	\end{equation}
	and $0$ otherwise. Here the coefficients $P_m(r,a)$ are computed recursively as
	\begin{equation}\label{eqn:PPZ:polys}
	\begin{cases}
		P_m(r,a) - P_m(r,a-1) = r \left( m - \frac{1}{2} - \frac{a}{r} \right) P_{m-1}(r,a-1)\, ,
		\qquad
		\text{for } a = 1,\dots,r-2\,, \\
		P_m(r,r-1) = P_m(r,0) \, ,
	\end{cases}
	\end{equation}
	with initial condition $P_0 = 1$.
\end{lemma}

\begin{proof}
	The expression for the Hodge grading operator $ \mu $ and the one for the Euler field $ \phi $ match the ones in \cite[section 4.5]{PPZ19}, except for a factor of $-r$ in the latter. Thus the proof is identical to the one presented there.
\end{proof}

It is intriguing that the $ R $-matrix for the deformed Theta class essentially matches the $ R $-matrix for the $ e_1 $-shifted Witten class studied in \cite{PPZ19}. We do not know of a good algebro-geometric reason for this occurrence and this  deserves further investigation.

Finally, we need to compute the vacuum vector of the theory. The unit vector field at the origin is $ \mathbf{1} = \hat{v}_{r-2} $ for $r \ge 3$ and $ \mathbf{1} = \hat{v}_{1} $ for $r = 2$, as one can easily check from the quantum product. We define the formal power series $ \ms{H}(r,a;u) $ for $ a = 0 ,\ldots, r-1 $ as follows. First, define
\begin{equation}\label{eqn:H:coeff}
	\begin{cases}
		H_k(r,a) \coloneqq \frac{(rk+r-1-a)!}{k! r^k} \, ,
		\qquad
		\text{for } a = 0,\dots,r-2\,, \\
		H_k(r,r-1) \coloneqq H_k(r,0) \, .
	\end{cases}
\end{equation}
Then we define the formal power series
\begin{equation}\label{eqn:H:def}
	\begin{cases}
	\ms{H}(r,a;u)
	\coloneqq
	\sum_{k \geq 0} H_k(r,a) \left(\frac{u}{\epsilon^{\frac{r}{r-1}}}\right)^{(r-1)k + r-2 - a} \, ,
	\qquad \text{for } a = 0,\dots,r-2 \,,\\
	\ms{H}(r,r-1;u)
	\coloneqq H(r,0;u).
	\end{cases}
\end{equation}
For $r = 3$, the above functions are related to the asymptotic expansion of the Scorer function, which are solutions of the inhomogeneous Airy ODE.

\begin{lemma}\label{lem:vac}
	The vacuum in the Teleman reconstruction theorem for the deformed Theta class is 
	\begin{equation}
		\mathbf{v}(u) = \sum_{a=1}^{r-1} \ms{H}(r,a;u) \hat{v}_a 
	\end{equation}
	and thus the translation is $T(u) = u \bigl( \mathbf{1} - R^{-1}(u) \mathbf{v}(u) \bigr)$.
\end{lemma}

\begin{proof}
	With the matrices given by \cref{eqn:mu:phi}, components of the vacuum ODE translates as
	\[
		\frac{d \ms{H}(r,a;u)}{du} + \frac{r+a}{r u} \ms{H}(r,a;u)
		=
		\frac{r-1}{r u^2} \epsilon^{\frac{r}{r-1}} \bigl( \ms{H}(r,a-1;u) - \delta_{a,r-2} \bigr)\,,
	\]
	for $ 1 \leq a \leq r-1 $, with the convention $ \ms{H}(r,0;u) = \ms{H}(r,r-1;u)$. Let us plug in our formulae for $ \ms{H}(r,a;u) $ and compute the left-hand side of the above equation:
	\begin{equation*}
	\begin{split}
		\frac{d \ms{H}(r,a;u)}{du} + \frac{r+a}{r u} \ms{H}(r,a;u)
		& =
		\frac{r-1}{r} \frac{1}{\epsilon^{\frac{r}{r-1}}}
		\sum_{k\geq 0 + \delta_{a,r-2}} \frac{(rk+r-a)!}{k! r^k} \left( \frac{u}{\epsilon^{\frac{r}{r-1}}} \right)^{(r-1)k + r-2 - a-1} \\
		& =
		\frac{r-1}{r u^2} \epsilon^{\frac{r}{r-1}} \bigl( \ms{H}(r,a-1;u) - \delta_{a,r-2} \bigr) \,.
	\end{split}
	\end{equation*}
	This proves the lemma for $r \ge 3$. For $r = 2$ the differential equation for the vacuum collapses to
	\[
		\frac{d \mathbf{v}(u)}{du} + \frac{3}{2} \frac{\mathbf{v}(u)}{u} = \frac{\epsilon^2}{2} \frac{\mathbf{v}(u) - \hat{v}_1}{u^2}\, ,
	\]
	which is still solved by $\mathbf{v}(u) = \ms{H}(2,1;u) \hat{v}_1$.
\end{proof}

Finally, we obtain an expression for the deformed Theta class $ \Theta^{r,\epsilon}_{g,n} $ in terms of tautological classes using the Teleman reconstruction theorem. Alternatively, we can express the same class as the translation of a CohFT (with flat unit) obtained by the action of the unit preserving $ R $-matrix on the topological field theory $ w_{g,n} $. Let us consider the following translation 
\begin{equation}\label{eqn:Ttilde}
	\widetilde{T}(u) = u \bigl( \mathbf{1} - \mathbf{v}(u) \bigr)\,.
\end{equation}

\begin{theorem}\label{thm:dTheta}
	The deformed Theta class $ \Theta^{r,\epsilon}_{g,n} $ has the following expression in terms of tautological classes 
	\begin{equation}
		\Theta^{r,\epsilon}_{g,n}  = R T w_{g,n} = \widetilde{T} (R. w_{g,n} ) \, ,
	\end{equation}
	for all stable $ (g,n) $, except  $(g,n) = (g,0)$ and degree  $d = 3g-3$. We computed the topological field theory $ w_{g,n} $ in \cref{lem:tft}, the $ R $-matrix in \cref{lem:Rmatrix} and the translation $ T  $ in \cref{lem:vac}.
\end{theorem} 

\begin{proof}
	The first equality follows directly from the Teleman reconstruction theorem (\cref{thm:Teleman}) and the extension to $ n = 0 $ proved in \cref{prop:n=0}. In order to get the second equality, notice that $R T = \widetilde{T} + u( R\,\mathbf{1} - \mathbf{1})$. Then, by applying \cite[proposition~2.9]{PPZ15} and using the fact that translations act as an abelian group gives $\Theta^{r,\epsilon}_{g,n} = ( \widetilde{T} + u( R\,\mathbf{1} - \mathbf{1}) ) R w_{g,n} = \widetilde{T} (R. w_{g,n} )$.
\end{proof}

\begin{remark}\label{rem:n=0} 
	As the deformed Theta class has conformal dimension $ \delta  = 3 $, \cref{prop:n=0} says that the Teleman reconstruction theorem does not hold for $ n = 0 $ and degree $ 3g-3  $, which explains the exception in the above theorem.  For $ n = 0 $, while the deformed Theta class $ \Theta^{r, \epsilon}_{g} $ has no terms in degree $ 3g-3 $, the class $ R T w_{g} $ does. In fact, we have
	\[
		R T w_{g} = \frac{(\operatorname{const})}{\epsilon^{2g-2}} + \Theta^{r, \epsilon}_{g},
	\]
	where the constant is a class in $ H^{2(3g-3)}(\Mbar_{g}) \cong \Q $. The above expression resembles the form of the Gromov--Witten potential at the conifold point, known as the conifold gap \cite{HK07}. Furthermore, we calculate the constant appearing in the above expression for $ r = 2 $ in \cref{prop:Fg} to be $ (-1)^g \frac{B_{2g-2}}{2g(2g-2)} $ which  matches with the expected constant for the Gromov--Witten potential at the conifold point. 
\end{remark}

\begin{remark}
	We observe that the deformed Theta class can be expressed as a translation of the shifted Witten class when $ r =3 $. To be precise, \cite{PPZ15} shows that  the unit preserving $ R $-matrix action  on the TFT $ w_{g,n} $ gives the shifted Witten class. Indeed, \cite[section 4]{PPZ19} shows that
	\begin{equation}
		R.w_{g,n} = (-3)^{\frac{(g-1)+|a|-n}{3}} (-\epsilon)^{2g-2+n} \, W^{\,3, \gamma }_{g,n}\,,
	\end{equation}
	with $ \gamma = \epsilon (-3)^{-2/3} e_{1} $. Thus, up to an overall factor, the deformed Theta class can be expressed as a translation of the shifted Witten class for $ r=3 $ using \cref{thm:dTheta}:
	\begin{equation}
		\Theta^{3, \epsilon}_{g,n} = (-3)^{\frac{(g-1)+|a|-n}{3}} (-\epsilon)^{2g-2+n} \widetilde{T}  W^{\,3, \gamma }_{g,n}\,.
	\end{equation}
\end{remark}

\begin{remark}
	One can find an alternative expression for the deformed Theta class $ \Theta^{r,\epsilon} $ in terms of tautological classes using Chiodo's Grothendieck--Riemann--Roch formula in \cite{Chi08+}. It would be interesting to compare this expression with the one that we find in \cref{thm:dTheta}. 
\end{remark}

As we know that the deformed Theta class $ \Theta^{r,\epsilon}_{g,n} (v_{a_1}\otimes \cdots \otimes v_{a_n}) $ only has terms in degrees $ d $ less than or equal to $ D^r_{g;a} $ by construction, we get the following vanishing result in the tautological ring of $ \Mbar_{g,n} $:

\begin{corollary}\label{cor:vanishing}
	The terms $\bigl[ (R T w_{g,n} )(v_{a_1}\otimes \cdots \otimes v_{a_n}) \bigr]^{d} \in H^{2d}(\Mbar_{g,n})$ vanish for any $d > D^r_{g;a}$ (except for $(g,n) = (g,0)$ and degree  $d = 3g-3$). In addition, we have 
	\begin{equation}
		\Theta^{r}_{g,n} (v_{a_1}\otimes \cdots \otimes v_{a_n})
		=
		\bigl[ (R T w_{g,n})(v_{a_1}\otimes \cdots \otimes v_{a_n}) \bigr]^{D^r_{g;a}}\,.
	\end{equation}
\end{corollary}

For an explicit description  of the reconstructed deformed Theta class in \cref{thm:dTheta} as a sum over decorated stable graphs and consequently the tautological relations in \cref{cor:vanishing}, see \cref{prop:stablegraphs}.

We note that while we expect the above tautological relations to be implied by Pixton's relations, we do not have a proof of this result. Janda \cite{Jan18} shows that the  tautological relations one obtains by taking the limit to the discriminant locus of a generically semisimple Dubrovin--Frobenius manifold are implied by Pixton's relations. While we see our relations in \cref{cor:vanishing} as morally fitting into the same class of relations, it is not covered by \cite{Jan18} for the following reasons. First, as we have emphasised previously, our $ \epsilon $-deformation is not a shift to a semisimple point of a generically semisimple Dubrovin--Frobenius manifold, but rather a family of Dubrovin--Frobenius manifolds that collapses to a non-generically semisimple one at $ \epsilon = 0 $. Second, \textit{loc.~cit.} only treats Dubrovin--Frobenius manifold with a flat unit vector field, and our unit vector field is not flat.

As mentioned in the introduction, the construction of the (deformed) Theta class is substantially easier than the construction of its positive spin analogue -- the  Witten $ r $-spin class. Thus, it is an interesting line of investigation to ask whether our relations in \cref{cor:vanishing} imply Pixton's relations \cite{Pix13}. We leave this for future work.

When $ r = 2 $, the Teleman reconstruction takes a strikingly simple form, giving tautological relations between $\kappa$-classes and a simple expression for the Theta class studied by Norbury \cite{Nor23}. Both properties have been recently conjectured by Kazarian--Norbury \cite[(1) in conjecture~1, conjecture~4]{KN24}. Consider the rational numbers $s_m$ for $m > 0$ defined uniquely via
\begin{equation}\label{eqn:K:series}
	\exp\left(-\sum_{m > 0}  s_m u^m \right) = \sum_{k \geq 0} (-1)^k (2k+1)!!  u^{k}\,.
\end{equation}

\begin{corollary}[Kazarian--Norbury conjecture]\label{cor:KN}
	We have the following vanishing relations among $\kappa$-classes,
	\begin{equation}\label{eqn:K:vanishing}
		\left[ \exp\biggl( \sum_{m > 0} s_m \kappa_m \biggr) \right]^{d} = 0 \in H^{2d}(\Mbar_{g,n})\,,
		\qquad \text{for } d > 2g-2 + n\, (\text{except } (g,n) = (g,0) \text{ and }  d = 3g-3)\,.
	\end{equation}
	Moreover, in degree $ d = 2g-2+n $, we get the Theta class:
	\begin{equation}
		\Theta_{g,n} = \left[ \exp\biggl( \sum_{m > 0} s_m \kappa_m \biggr) \right]^{2g-2+n}  \in H^{2(2g-2+n)}(\Mbar_{g,n}) \,.
	\end{equation}
\end{corollary}

\begin{proof}
	When $ r=2 $, the $R$-matrix computed in \cref{lem:Rmatrix} simplifies to $ R(u) = 1 $. The translation is then given by \cref{lem:vac} as
	\begin{equation*}
		T(u) = - \sum_{k \geq 1} (-1)^k (2k+1)!! (-\epsilon^{2})^{-k} u^{k+1} \hat{v}_1\,.
	\end{equation*}
	Taking into account the TFT $w_{g,n}(\hat{v}_1^{\otimes n}) = (-\epsilon^2)^{2g-2}$ and rescaling to the basis $v_1$, \cref{thm:dTheta} specialised to the case $ r=2 $ gives
	\begin{equation*}
		\Theta^{2,\epsilon}_{g,n}( v_1^{\otimes n}) = (-\epsilon^2)^{2g-2+n} \exp\bigg( \sum_{m > 0} s_m (- \epsilon^2)^{-m} \kappa_m \bigg)\,,
	\end{equation*}
	where the $s_m$ are defined through \labelcref{eqn:K:series}. Now, applying \cref{cor:vanishing} gives the vanishing result \labelcref{eqn:K:vanishing} upon setting $ \epsilon = \sqrt{-1} $, and the expression of $\Theta_{g,n}$ in terms of $\kappa$-classes (which is independent of $\epsilon$).
\end{proof}

For $ n = 0 $, Kazarian and Norbury conjecture an explicit formula for the top degree of the reconstructed class $ [\, \exp( \sum_{m > 0} s_m \kappa_m ) \,]^{3g-3} $ that we prove in \cref{prop:Fg} using the identification with topological recursion on a certain spectral curve.

%––––––––––––––––––––––––––––––––––––––––––––
\subsubsection{Deformed Theta class in terms of stable graphs}
%––––––––––––––––––––––––––––––––––––––––––––

When $ r \geq 3 $, we can express the relations and the expression for the $ \Theta^r $-class as a sum of decorated stable graphs as follows. We refer to \cite[section 4.6]{PPZ19} for the analogous statement in the case of Witten classes. We denote the set of stable graphs of genus $ g $ with $ n+k $ legs by $ \mathsf{G}_{g,n+k} $. The first $ n $ legs are ordinary legs that carry $ \psi $-classes, and the last $ k $ legs are dilaton legs which carry $\psi$-classes that will get pushed forward to $ \kappa $-classes.

\begin{definition} 
	Fix $0 < a_1,\dots,a_n \le r-1$. Consider a stable graph $ \Gamma \in \mathsf{G}_{g,n+k}$. We define a \textit{weighting} of $ \Gamma $ with boundary conditions $a = (a_1,\dots,a_n)$ as an assignment $\bm{a}$ to the set of half-edges
	\begin{equation}
		H(\Gamma) \longrightarrow \{1, \ldots, r-1\}\,, \qquad  h \longmapsto a_h
	\end{equation}
	that satisfies the following conditions:
	\begin{itemize}
		\item if $ h $ and $ h' $ are two half-edges that are connected to form an edge, then $ a_h + a_{h'} = r $,

		\item if $ h $ corresponds to the ordinary leg marked with $i \in \set{1,\dots,n}$, then $ a_h = a_i $.
	\end{itemize}
	We denote by $\mathsf{W}_{\Gamma}(a)$ the set of weightings of $ \Gamma $ with boundary conditions $a$.
\end{definition}

Let $\bm{a} \in \mathsf{W}_{\Gamma}(a)$. To every vertex $ v $, assign a formal variable $ \gamma_v $ such that $ \gamma_v^{r-1} = 1$. Then, we have the following weights associated to $ \Gamma $.

For every edge $ e = (h,h') \in E(\Gamma) $, assign the \textit{edge weight}:
\begin{equation}
	\Delta(e)
	=
	\gamma_{v}^{\vphantom{a_h'}a_h} \gamma_{v'}^{a_{h'}}
	\left(
		\frac{1 - \sum_{m, l \geq 0} P_m(r,a_h - 1) P_l(r,a_{h'} - 1)\
		(\gamma_{v}^{-1} \psi_{h})^m (\gamma_{v'}^{-1} \psi_{h'})^l}{\psi_h + \psi_{h'}}
	\right) ,
\end{equation}
where $v$ and $v'$ are the vertices paired to $h$ and $h'$ respectively. To the ordinary leg marked with $i \in \set{1,\dots,n}$, assign the \textit{ordinary-leg weight}:
\begin{equation}
	L(i) = \gamma_v^{a_i} \sum_{m \geq 0} P_m(r, a_i -1) (\gamma_v^{-1} \psi_i)^m\,,
\end{equation}
where $v$ is the vertex paired to the ordinary leg $i$. To the dilaton leg marked with $j \in \set{n+1,\dots,n+k}$, assign the \textit{dilaton-leg weight}:
\begin{equation}
	K(j)
	=
	\psi_j \gamma_v^{a_j} \left( \delta_{a_j, r-2} - \sum_{m,l \geq 0} P_m(r,a_j-1) H_{l}(r,a_j) (1-r)^{(r-1)l + r-2-a_j}\gamma_v^{-m} \psi_j^{m + (r-1)l + r-2-a_j}\right),
\end{equation}
where $v$ is the vertex paired to the dilaton leg $j$.

By $ \set{ \Pi }_{\gamma} $, we denote the term of degree $ 0 $ in $ \Pi $ in the variables $ \gamma_v $.

\begin{proposition}\label{prop:stablegraphs}
	The deformed Theta class $ \Theta_{g,n}^{r,\epsilon}(v_{a_1} \otimes \cdots \otimes v_{a_n})$ with $\epsilon = 1$ (except for the special case $ n = 0 $ in degree $ (3g-3 $)) can be expressed as the following sum over decorated stable graphs:
	\begin{equation}
		\sum_{k \geq 0} \sum_{\Gamma \in \mathsf{G}_{g,n+k}} \sum_{\bm{a} \in \mathsf{W}_{\Gamma}(a)}
		\frac{ (-1)^{\deg} (r-1)^{g-h^1(\Gamma) - \deg} }{|\Aut(\Gamma)|}
		\xi_{\Gamma,*} \left\{
			\prod_{v} \gamma_v^{3g(v)-3+n(v)}
			\prod_ {e} \Delta(e)
			\prod_{i=1}^n L(i)
			\prod_{j = n+1}^{n+k} K(j)
		\right\}_\gamma .
	\end{equation}
	In particular, the degree $D^r_{g;a}$ part of the above mixed degree class is the expression for the class $\Theta^r_{g,n}$ and the classes in degree $ d > D^r_{g;a} $ vanish identically (again, excluding $ n = 0 $ and $ d = 3g-3 $).
\end{proposition}

\begin{proof}
	The expression is a rewriting of the $ R $-matrix and translation  action of \cref{thm:dTheta}. The powers of $ \gamma $ at the vertices $ v $ keep track of the conditions $\pmod{r-1} $. The basis vector $ v_{a} $ corresponds to $ \gamma^{a} $, and the pairing $\eta^{a_{h},a_{h'}} $ corresponds to $\gamma_{v}^{\vphantom{a_h'}a_h} \gamma_{v'}^{a_{h'}}$ such that $a_{h}+a_{h'} = r$. Adding the weight $ \gamma_v^{3g(v)-3+n(v)} $ at every vertex $ v $ and extracting the coefficient of $ \gamma_v^0 $ enforces the condition 
	\[
		3g(v) - 3 + n(v) + \sum_{h \mapsto v} a_h \equiv 0 \pmod{r-1} 
	\]
	at every vertex $ v $. The coefficient of the $R$-matrix $ R^{-1}_m $ takes $ v_a $ to a multiple of $ v_{b} $ such that $ b = a - m \pmod{r-1} $. Thus, we need to add a factor of $ \gamma^{-m} $ everywhere the $R$-matrix appears. Finally, we have removed a factor of $ (-1)^m(r-1)^m $ in the denominator of the $ R $-matrix and a factor of $(r-1)^{g(v)}$ from the topological field theory at every vertex $ v $. These can be combined together into the factor $(-1)^{\deg} (r-1)^{g-h^1(\Gamma) - \deg }$ at the cost of rescaling the $ H_l(r,a_i) $ by $ (1-r)^{(r-1)l + r-2-a_i} $.
\end{proof}

%––––––––––––––––––––––––––––––––––––––––––––
\section{Topological recursion and the spectral curve}
\label{sec:TR}
%––––––––––––––––––––––––––––––––––––––––––––

Topological recursion, TR for short, is a universal procedure that associates a collection of symmetric multidifferentials to a spectral curve, which is a curve with some extra data \cite{EO07}. What makes TR especially useful is its applications to enumerative geometry: many counting problems are solved by TR, in the sense that the sought numbers are coefficients of the multidifferentials when expanded in a specific basis \cite{Eyn14}.

\begin{definition}\label{defn:spectral:curve}
	A \emph{spectral curve} $\mathcal{S} = (\Sigma,x,y,\omega_{0,2})$ consists of
	\begin{itemize}
		\item a Riemann surface $\Sigma$ (not necessarily compact or connected);

		\item a function $x \colon \Sigma \to \C$ such that its differential $dx$ is meromorphic and has finitely many simple zeros  $\alpha_1,\dots,\alpha_{r-1}$ called ramification points; 

		\item a meromorphic function $y \colon \Sigma \to \C$ such that $dy$ does not vanish at the zeros of $dx$;

		\item a symmetric bidifferential $\omega_{0,2}$ on $\Sigma \times \Sigma$, with a double pole on the diagonal with biresidue $1$, and no other poles.
	\end{itemize}
	The \emph{topological recursion} produces symmetric multidifferentials (also called \emph{correlators}) $\omega_{g,n}$, for $ g \geq 0 $ and $ n \geq 1 $ such that $ (2g-2+n)> 0 $ on $\Sigma^n$. The correlators are defined recursively on $2g - 2 + n > 0$ as
	\begin{equation}\label{eqn:TR}
	\begin{split}
		\omega_{g,n}(z_1,\dots,z_n) \coloneqq \sum_{i=1}^{r-1} \Res_{z = \alpha_i} K_i(z_1,z) \bigg( &
			\omega_{g-1,n+1}(z,\sigma_i(z),z_2,\dots,z_n) \\
			& +
			\sum_{\substack{g_1+g_2 = g \\ J_1 \sqcup J_2 = \{2,\dots,n\}}}^{\text{no $(0,1)$}}
				\omega_{g_1,1+|J_1|}(z,z_{J_1})
				\omega_{g_2,1+|J_2|}(\sigma_i(z),z_{J_2})
		\bigg)\,,
	\end{split}
	\end{equation}
	where $K_i$, called the topological recursion kernels, are locally defined in a neighbourhood $U_i$ of $\alpha_i$ as
	\begin{equation}\label{eqn:TR:kernel}
		K_i(z_1,z) \coloneqq \frac{\frac{1}{2} \int_{w = \sigma_i(z)}^z \omega_{0,2}(z_1,w)}{\bigl( y(z) - y(\sigma_i(z)) \bigr) d x(z)}\,,
	\end{equation}
	and $\sigma_i \colon U_i \to U_i$ is the Galois involution near the ramification point $\alpha_i \in U_i$. It can be shown that $\omega_{g,n}$ is a symmetric meromorphic multidifferential on $\Sigma^n$, with poles only at the ramification points.
\end{definition}

There is an extension of the topological recursion to the case of higher ramification points, meaning that the zeroes of $ dx $ are of higher order, known as the Bouchard--Eynard topological recursion \cite{BHLMR14}. While this version of TR will be relevant for us, we do not go into the details of this formalism in this paper. Instead, we refer the reader to the papers \cite{BHLMR14,  BE17, BBCCN24}, where the Bouchard--Eynard topological recursion is defined and investigated.

%––––––––––––––––––––––––––––––––––––––––––––
\subsubsection{Identification with CohFTs}
\label{sec:DOSS}
%––––––––––––––––––––––––––––––––––––––––––––

When the Riemann surface underlying the spectral curve is compact, we can represent the topological recursion correlators on a basis of auxiliary differentials with coefficients given by intersection numbers on the moduli space of curves of a CohFT and $\psi$-classes (see \cite{DOSS14,DNOS19}). In this section, we will work with CohFTs over $\C$.

We fix a global constant $C \in \C^{\times}$. Choose local coordinates $\zeta_i$ in the neighbourhood $U_i$ of the ramification point $ \alpha_i $ such that $\zeta_i(\alpha_i) = 0$ and $x - x(\alpha_i) = - \frac{\zeta_i^2}{2}$. Consider the auxiliary functions $\xi^i \colon \Sigma \to \C$ and the associated differentials $d\xi^{k,i}$ defined as
\begin{equation}
	\xi^i(z)
	\coloneqq
	\int^z \left.\frac{\omega_{0,2}(\zeta_{i},\cdot)}{d\zeta_{i}}\right|_{\zeta_{i} = 0}\,,
	\qquad
	d\xi^{k,i}(z)
	\coloneqq
	d\biggl( \frac{d^k}{dx(z)^k} \xi^{i}(z) \biggr)\,.
\end{equation}
We also set $\Delta^i \coloneqq \frac{dy}{d\zeta_i}(0)$ and $h^i \coloneqq C \Delta^i$. We define a unital, semisimple TFT on $V \coloneqq \C\braket{e_1,\dots,e_r}$ by setting
\begin{equation}
	\eta(e_i,e_j) \coloneqq \delta_{i,j}\,,
	\qquad
	\mathbf{1}
	\coloneqq
	\sum_{i=1}^r h^i e_i\,,
	\qquad
	w_{g,n}(e_{i_1} \otimes \cdots \otimes e_{i_n})
	\coloneqq
	\delta_{i_1,\ldots,i_n} (h^i)^{-2g+2-n}\,.
\end{equation}
Define the $R$-matrix $R \in \End(V)\bbraket{u}$ and the translation $T \in u^2 V\bbraket{u}$ in the basis $(e_i)$ by setting
\begin{equation}
	R^{-1}(u)^j_i
	\coloneqq
	- \sqrt{\frac{u}{2\pi}} \int_{\gamma_j}
		d\xi^i \,
		e^{\frac{1}{u}(x - x(\alpha_j))} \,,
	\qquad
	T(u)^i
	\coloneqq
	u h^i - C \sqrt{\frac{u}{2\pi}} \int_{\gamma_i} dy \, e^{\frac{1}{u}(x - x(\alpha_i))} \,.
\end{equation}
Here $\gamma_i$ is the Lefschetz thimble passing through the saddle point $a_i$, and the equations are intended as equalities between formal power series in $u$, where on the right-hand side we take the formal asymptotic expansion as $u \to 0$. Through the Givental action, we can then define a cohomological field theory
\begin{equation}
	\Omega_{g,n} \coloneqq RTw_{g,n} \in H^{\bullet}(\overline{\mathcal{M}}_{g,n}) \otimes (V^{*})^{\otimes n}
\end{equation}
from the data $(w,R,T)$. We point out that the resulting CohFT does not necessarily have a flat unit. The link with the topological recursion correlators is given by the following result due to Dunin-Barkowski, Orantin, Shadrin, and Spitz \cite{DOSS14}.

\begin{theorem}[\cite{DOSS14}]\label{thm:Eyn:DOSS}
	Assume that we have a compact spectral curve $ (\Sigma, x,y,\omega_{0,2})$. Then its topological recursion correlators are given by
	\begin{equation}
		\omega_{g,n}(z_1,\dots,z_n)
		=
		C^{2g-2+n}
		\sum_{i_1,\dots,i_n = 1}^{r-1} \int_{\overline{\mathcal{M}}_{g,n}}
			\Omega_{g,n}(e_{i_1} \otimes\cdots\otimes e_{i_n}) \prod_{j=1}^n \sum_{k_j \ge 0}
				\psi_j^{k_j} d\xi^{k_j,i_j}(z_j)\,.
	\end{equation}
\end{theorem}

%––––––––––––––––––––––––––––––––––––––––––––
\subsection{Hyper-Airy functions}
\label{sec:hAiry}
%––––––––––––––––––––––––––––––––––––––––––––

Before computing the spectral curve associated to the deformed Theta class, we take a quick detour to study certain generalisations of Airy functions that will appear. The well-known Airy function $ \Ai(t) $ is a solution to the differential equation $u''(t) = tu(t)$, and admits an integral representation and an asymptotic formula as $t \to \infty$ in the region $|\arg(t)| < \pi$:
\begin{equation}
	\Ai(t)
	=
	\frac{1}{2\pi\iu} \int_{C} e^{\frac{w^3}{3} - tw} dw
	\sim
	\frac{e^{-\frac{2t^{3/2}}{3}}}{2\sqrt{\pi}} t^{-1/4} \, \sum_{k \ge 0}
	\frac{(6k)!}{(2k)! (3k)!}
	\left( -\frac{1}{576 t^{3/2}} \right)^k.
\end{equation}
Here $C$ is the path starting at $e^{-\frac{\pi}{3}} \infty$ and ending at $e^{\frac{\pi}{3}} \infty$.

The generalisations of the Airy function for higher $ r $ that we are interested in are called \textit{hyper-Airy functions}, and are solutions to the differential equation
\begin{equation}\label{eqn:hyperAiryeq}
	u^{(r-1)}(t) = (-1)^{r-1} t u(t)\,.
\end{equation}
for any $ r \geq 3 $.  We can define $r-1$ (independent) solutions via contour integrals as
\begin{equation}\label{eqn:hyperAiry}
	\widetilde{\Ai}_{r,k}(t) \coloneqq \frac{\theta^{k\frac{r-2}{2}}}{2\pi\iu} \int_{\tilde{C}_k} e^{\frac{w^{r}}{r} - t w} dw\,,
	\qquad\quad
	k = 0,\dots,r-2\,,
\end{equation}
where $\theta = e^{\frac{2\pi\iu}{r-1}}$ and $\tilde{C}_k$ is the Lefschetz thimble passing through the critical point $w = t^{1/(r-1)} \theta^k$.  The hyper-Airy functions were studied in some detail in \cite{CCGG24}. In particular, the reader can find a careful description of the Lefschetz thimbles in appendix B of \textit{loc.~cit.}

We also consider the derivatives (and anti-derivatives) of the hyperAiry function 
\begin{equation}\label{eqn:hyperAirya}
	\widetilde{\Ai}_{r,k}^{(a)}(t) =  \frac{\theta^{k\frac{r-2}{2}}}{2\pi\iu} \int_{\tilde{C}_k} (-w)^a e^{\frac{w^{r}}{r} - t w} dw\,, 
\end{equation}
where we assume that $ a \in \Z $. The hyper-Airy functions and their (anti-)derivatives admit an asymptotic expansion. In order to state this expansion, we need to recall the polynomials $ P_m(r,a) $ that we have already encountered in \cref{eqn:PPZ:polys}, and extend their definition to the case $ a = -1 $. Precisely, the polynomials $P_m(r,a)$ are defined recursively for any $ -1 \leq a \leq r-1 $ as
\begin{equation}\label{eqn:Ppoly}
	\begin{cases}
		P_m(r,a) - P_m(r,a-1) = r \left( m - \frac{1}{2} - \frac{a}{r} \right) P_{m-1}(r,a-1)\,,
		\qquad
		\text{for } a = 1,\dots,r-2\,, \\
		P_m(r,0) = P_m(r,r-1)\,,
	\end{cases}
\end{equation}
with the initial condition $P_0 = 1$.

\begin{proposition}\label{prop:asymp}
	The asymptotic expansions as $t \to \infty$ of the hyper-Airy function and its (anti-)derivatives, for $ -1 \leq a \leq r-2 $ are given by
	\begin{equation}
		\widetilde{\Ai}_{r,k}^{(a)}(t)
		\sim
		\frac{(-\theta^k)^a}{\sqrt{2\pi(r-1)}}
		e^{
			- \frac{r-1}{r}
			\theta^k
			t^{\frac{r}{r-1}}
		}
		t^{-\frac{r-2a-2}{2(r-1)}}
		\sum_{m \ge 0}
		P_m(r,a) \left( \theta^{-k} \frac{t^{-\frac{r}{r-1}}}{r-1}  \right)^m.
	\end{equation}
\end{proposition}

\begin{proof}
	First we use the steepest descent method to compute the leading behaviour. Then, using the ODE \labelcref{eqn:hyperAiryeq}, we find  the recursion relation \labelcref{eqn:Ppoly}. It is identical to \cite[lemma~3.3, proposition~3.4]{CCGG24}, where we only note that the proof given there goes through unchanged for $ a = -1 $.
\end{proof}

For later convenience, let us introduce the following notation for the asymptotic expansion of the hyper-Airy functions and their (anti-)derivatives.

\begin{definition}
	Define the formal power series for $ -1 \leq a \leq r-1 $
	\begin{equation}
		\mathsf{A}_{r}^{(a)}(u)
		\coloneqq
		\sum_{m \ge 0}
		P_m(r,a) \left( \frac{u}{r(r-1)}\right)^m,
	\end{equation}
	so that the asymptotic expansions as $t \to \infty$ of the hyper-Airy function and its (anti-)derivatives are written as
	\begin{equation}
		\widetilde{\Ai}_{r,k}^{(a)}(t)
		\sim
		\frac{(-\theta^k)^a}{\sqrt{2\pi(r-1)}}
		e^{ - \frac{r-1}{r} \theta^k t^{\frac{r}{r-1}} }
		t^{-\frac{r-2a-2}{2(r-1)}}
		\mathsf{A}_{r}^{(a)}\bigl( \theta^{-k} \, r t^{- \frac{r}{r-1}} \bigr) \, .
	\end{equation}
\end{definition}

%––––––––––––––––––––––––––––––––––––––––––––
\subsection{The spectral curve for the deformed Theta class}\label{sec:SCdTheta}
%––––––––––––––––––––––––––––––––––––––––––––

In this section, we will show that the descendant potential of the deformed Theta class can be constructed by the topological recursion on a certain  spectral curve. To start, let us consider the following one-parameter family of spectral curves on $ \P^1 $:
\begin{equation}\label{eqn:tS}
	x(z) = z^r - r \lambda^{r-1} z\,,
	\qquad\quad
	y(z) = - \frac{1}{z}\,,
	\qquad\quad
	\omega_{0,2}(z_1,z_2) = \frac{dz_1 dz_2}{(z_1 - z_2)^2}\,,
\end{equation}
where $ \lambda \neq 0 $\,. Now, we will compute all the ingredients of the DOSS formula from \cref{thm:Eyn:DOSS} for this spectral curve as explained in the previous section. The ramification points of the spectral curve are given by $ x'(z) = 0 $, the solutions of which are $\alpha_k \coloneqq \theta^k \lambda$ for $k = 1,\ldots, r-1$, where $ \theta $ is a primitive $ (r-1) $-th root of unity. The value of $ x $ at the ramification points is given by $x_k \coloneqq x(\alpha_k) = -(r-1)\theta^k \lambda^{r} $. We will choose local coordinates $ \zeta_k $ around the ramification point $ z = \alpha_k $, such that 
\begin{equation}
	x(z) - x_k = -\frac{\zeta_k^2(z)}{2}\,,
\end{equation}
and we choose a branch of the square root such that 
\begin{equation}
	\zeta_k (z) = -i \sqrt{r(r-1)} (\theta^k \lambda)^{\frac{r-2}{2}} (z - \alpha_k) + O\bigl( (z - \alpha_k)^2 \bigr)\,.
\end{equation}
Using the recipe of \cref{sec:DOSS}, we compute $ \Delta^k = (-i \sqrt{r(r-1)} (\theta^k \lambda)^{\frac{r+2}{2}} )^{-1}$. We choose a global constant $ C = -i \sqrt{r}  \lambda^{\frac{r+2}{2}} $, so that we get $ h^k =  \frac{ \theta^{- k \frac{r+2}{2}} }{\sqrt{r-1}} $. Then the vector space, the pairing and the unit are given by
\begin{equation} 
	V = \C\braket{e_1,\dots,e_{r-1}}\,,
	\qquad
	\eta(e_{k},e_{l}) = \delta_{k,l}\,,
	\qquad
	\mathbf{1} = \frac{1}{\sqrt{r-1}} \sum_{k=1}^{r-1} \theta^{- k\frac{r+2}{2}} e_k \, ,
\end{equation}
and the TFT in the canonical basis $(e_i)$ is given by
\begin{equation} 
	w_{g,n}(e_{k_1}\otimes \cdots \otimes e_{k_n})
	=
	\delta_{k_1,\dots,k_n} (r-1)^{g-1+\frac{n}{2}} \theta^{k(r+2)(g-1+\frac{n}{2})} \, .
\end{equation}
Let us compute the other ingredients of the DOSS formula.

\begin{lemma}
	In the canonical basis $ (e_1,\ldots, e_{r-1}) $, the following holds.
	\begin{itemize}
		\item
		The auxiliary functions are given by
		\begin{equation}
			\xi^{k}_{\textup{C}}(z)
			=
			 \frac{\left( i \sqrt{r(r-1)} (\theta^k \lambda)^{\frac{r-2}{2}}\right)^{-1}}{z - \alpha_k}\,.
		\end{equation}

		\item
		The $R$-matrix elements in the canonical basis (denoted $R^{-1}_{\textup{C}}$) are given by
		\begin{equation}
			R^{-1}_{\textup{C}}(u)_i^j
			=
			\frac{1}{r-1}
			\sum_{s=0}^{r-2}
				\theta^{(i-j)\frac{r-2s-2}{2}}
				\mathsf{A}_r^{(s)} \left( \theta^{-j} \frac{u}{\lambda^r} \right) \, .
		\end{equation}

		\item
		The coefficients of the translation in the canonical basis (denoted $ T_{\textup{C}}(u) $) are given by 
		\begin{equation}
			T_{\textup{C}}(u)^k
			=
			\frac{\theta^{-k(r+2)/2}}{\sqrt{r-1}}	\left(u -  \lambda^r r  \theta^{k} \left( \ms{A}^{(r-2)}_r\left( \theta^{-k} \frac{u}{\lambda^r} \right) - \ms{A}^{(-1)}_r\left( \theta^{-k} \frac{u}{\lambda^r} \right) \right) \right)\,.
		\end{equation}
	\end{itemize}
\end{lemma}
 
\begin{proof}
	The calculation of the auxiliary functions and the $R$-matrix do not depend on $ y $ and thus are identical to the one in \cite[lemma 4.5]{CCGG24}. Thus, we only compute the translation here.

	We take the definition of the translation (we omit the subscript ``C'' in the proof) and integrate by parts to get
	\begin{equation*}
	\begin{split}
		T(u)^k
		& =
		u h^k - \iu \sqrt{\frac{u r}{2\pi}} \lambda^{\frac{r+2}{2}} \int_{\gamma_k} dy \, e^{\frac{1}{u}(x - x_k)} \\
		& =
		u h^k + \iu \sqrt{\frac{u r}{2\pi}} \lambda^{\frac{r+2}{2}}  \int_{\gamma_k} y \, e^{\frac{1}{u}(x - x_k)} \frac{dx}{u} \, .
	\end{split}
	\end{equation*}
	With the change of variables $ z = \left(\frac{u}{r}\right)^{1/r} w $ and setting $ t = \lambda^{r-1} \left(\frac{u}{r}\right)^{\frac{1}{r}-1} $, we find
	\begin{equation*}
	\begin{split}
		T(u)^k
		& =
		u h^k
		+
		\iu  e^{-\frac{x_k}{u}} r u \sqrt{\frac{1}{2\pi }} t^{\frac{r+2}{2(r-1)}}
		\int_{\tilde{C}_k}
			\frac{\left(
				w^{r-1} - t 
			\right)}{w}
			e^{
					\frac{w^r}{r} - t w
				}
			dw \\
		& =
		u h^k - u r \sqrt{2\pi}  \theta^{-k(r-2)/2} e^{-x_k/u} t^{\frac{r+2}{2(r-1)}} \left(  (-1)^{r-2} \widetilde{\Ai}^{(r-2)}_{r,k}(t) + t \widetilde{\Ai}^{(-1)}_{r,k}(t)   \right) .
	\end{split}
	\end{equation*}
	In the last line, we recognise the integral representations of the (anti-)derivatives of the hyper-Airy functions as discussed in \cref{sec:hAiry}. We plug in the asymptotic expansion of the functions $ \widetilde{\Ai}_{r,k}^{(a)}(t) $ for the cases $ a = -1 $  and $ a = r-2 $ to simplify $ T(u) $ to 
	\begin{equation*}
		T(u)^k
		=
		u h^k - \frac{ \lambda^r r}{\sqrt{r-1}} \theta^{-kr/2} \left(\ms{A}^{(r-2)}_r\left( \theta^{-k} \frac{u}{\lambda^r} \right) - \ms{A}^{(-1)}_r\left( \theta^{-k} \frac{u}{\lambda^r} \right) \right) .
	\end{equation*}
	This completes the proof.
\end{proof}

We have the following changes of basis between the flat one $( v_1, \ldots, v_{r-1} )$ and the canonical one $( e_1, \ldots, e_{r-1} )$:
\begin{equation}
	v_a  = \frac{1}{\sqrt{r-1}} \sum_{k=1}^{r-1} \theta^{-k \left(\frac{r}{2} -a\right)} e_k \,, \qquad 
	e_k  =\frac{1}{\sqrt{r-1}}\sum_{a=1}^{r-1} \theta^{k \left(\frac{r}{2} -a\right)} v_a\,. 
\end{equation} 
In the following, we identify the parameter $ \lambda $ in the spectral curve with the deformation parameter $ \epsilon $ in the deformed Theta class as $\lambda^{r-1} = \epsilon$, and carry out the change of basis computations.

\begin{lemma}\label{lem:DOSS:flat}
	In the flat basis $(v_1,\dots,v_{r-1})$, the following holds.
	\begin{itemize}
		\item
		The pairing and the TFT are expressed as
			\begin{equation}
			\eta(v_{a},v_{b})
			=
			\delta_{a+b,r}\,,
			\qquad
			w_{g,n}(v_{a_1} \otimes \cdots \otimes v_{a_n})
			=
			(r-1)^g \cdot \delta\,,
		\end{equation}
		where $\delta$ is equal to $1$ if $r-1$ divides $3g-3+n+|a|$ and $0$ otherwise. The unit for the TFT is given by $ \mathbf 1 = v_{r-2} $.
		
		\item
		The $R$-matrix elements in the flat basis (denoted $R^{-1}_{\textup{F}}$) are given by
		\begin{equation}\label{eqn:Rcomp}
			R_{\textup{F}}^{-1}(u)^{b}_{a}
			=
			\sum_{\substack{m \geq 0\\ a-b \equiv m \pmod{r-1}}} P_m(r,a-1) \left( \frac{1}{r(r-1)} \frac{u}{\epsilon^{\frac{r}{r-1}}} \right)^{m} \, .
		\end{equation}
		
		\item
		The translation is given by 
		\begin{equation}\label{eqn:Tcomp}
			T_{\textup{F}}(u)
			=
			u(\mathbf 1 - R^{-1}(u)\mathbf v(u))\,, \qquad 
			\mathbf{v}(u) = \sum_{a=1}^{r-1} \ms{H}\left(r,a;-\frac{u}{r}\right) v_a\, ,
		\end{equation}
		where we recall the formal power series $ \ms{H}(r,a) $ defined in \cref{eqn:H:def}.
	\end{itemize} 
\end{lemma}

\begin{proof}
	Throughout the proof, we use the parameter $ \lambda  $ (recall the  identification $ \lambda^{r-1} = \epsilon $) for convenience. As the computation of the pairing and the TFT is straightforward, and in any case very similar to the computations below, we omit it.
	
	Here is  a proof for the $R$-matrix elements:
	\[
	\begin{split}
		R_{\textup{F}}^{-1}(u)^{b}_{a}
		& =
		\sum_{i,j = 1}^{r-1}\sum_{s = 0}^{r-2}
		\frac{1}{(r-1)^2}
		\theta^{-i\left(\frac{r}{2}-a\right)}
		\theta^{j\left(\frac{r}{2}-b\right)}
		\theta^{(i-j)\frac{r-2s-2}{2}}
		\mathsf{A}_r^{(s)}\left( \theta^{-j} \frac{u}{\lambda^r} \right) \\
		& =
		\sum_{\substack{m \geq 0\\ a-b \equiv m \pmod{r-1}}} P_m(r,a-1) \left( \frac{1}{r(r-1)} \frac{u}{\lambda^r} \right)^{m} \, ,
	\end{split}
	\]
	where we have only used the definition of the formal power series $\ms{A}^{(s)}_r $ and standard sum over roots of unity calculations.
	
	As for the translation, we find that the coefficients of $T(u)$ in the flat basis are given by
	\begin{align*}
		T_{\textup{F}}(u)^a
		& =
		\sum_{k = 1}^{r-1}
		\frac{1}{r-1}
		\theta^{k\left(\frac{r}{2}-a\right)} \theta^{-k(r+2)/2}
		\left(u -  \lambda^r r  \theta^{k} \left(
			\ms{A}^{(r-2)}_r\left( \theta^{-k} \frac{u}{\lambda^r} \right) - \ms{A}^{(-1)}_r\left( \theta^{-k} \frac{u}{\lambda^r} \right) \right)
		\right) \\
		&= 
		u\Bigg( \delta_{a,r-2} -   \frac{1}{r-1}  \sum_{\substack{m \geq 0 \\ m \equiv  -a-1 \pmod{r-1}}} \bigl( P_{m+1}(r,r-2)- P_{m+1}(r,-1) \bigr) \left( \frac{u}{\lambda^r r(r-1)} \right)^m  \Bigg)\,,
	\end{align*}
	where we again plug in the definition of the formal power series $\ms{A}^{(a)}_r $ and carry out standard sum over roots of unity calculations. The above equation can be recast equivalently as 
	\begin{equation*}
		\delta_{a,r-2} - \frac{ T_{\textup{F}}(u)^a}{u}
		=
		\frac{1}{r-1} \sum_{\substack{m \geq 0 \\ m \equiv  -a-1 \pmod{r-1}}}
			\bigl( P_{m+1}(r,r-2)- P_{m+1}(r,-1) \bigr) \left( \frac{u}{\lambda^r r(r-1)} \right)^m.
	\end{equation*} 
	The relation~\labelcref{eqn:Tcomp}, which we are trying to show, can equivalently be expressed as $ \mathbf 1 - \frac{ T(u)}{u} = R^{-1}(u) \mathbf v (u) $.	In order to prove this (in flat coordinates), we plug in the expression for $ P_m(r,-1) $ proved in \cref{lem:P-1} into the above equation to get
	\begin{multline*}
		\delta_{a,r-2} - \frac{ T_{\textup{F}}(u)^a}{u}
		= \\
		- \frac{1}{r-1}  \sum_{\substack{m \geq 0 \\ m \equiv  -a-1 \pmod{r-1}}} \Biggl(
			\sum_{\substack{b,c \ge 0 \\ b + c = r-2}} \; \sum'_{\substack{i,\ell \ge 0 \\ i + (r-1)\ell = m+1}}
				(1-r)^{b + (r-1)\ell} \frac{(r\ell + b)!}{\ell! r^{\ell}}
				P_{i-b}(r,c) \Biggr)
				\left( \frac{u}{\lambda^r r(r-1)} \right)^{m} .
	\end{multline*}
	where the $ ' $ on the sum indicates that we drop the term  $ \ell = 0 = b $. Now, by using the definition of the function $ \ms{H}(r,a;u) $, we are able to simplify the above expression to 
	\begin{equation*}
		\sum_{\substack{c,b \ge 0 \\ c + b = r-2}} \sum_{\substack{i \ge b \\ i \equiv -a \pmod{r-1}}}
		\ms{H}\left(r,r-1-b;-\frac{u}{r}\right)
		P_{i-b}(r,c) \left( \frac{u}{\lambda^r r(r-1)} \right)^{i-b}.
	\end{equation*}
	Here, we note that the term that we dropped in the previous expression due to the prime in the sum is precisely the term that is required to identify $ \ms{H}(r,r-1;u) $ with $ \ms{H}(r,0;u) $ in \cref{eqn:H:def}. By making the change of variables  $ b \to r-2-b $, and $ i \to i+r-2-b $, we get 
	\begin{equation*}
		\sum_{b = 0}^{r-2} \sum_{\substack{i  \ge 0 \\ i \equiv b-a +1 \pmod{r-1}}}
			\ms{H}\left(r,b+1;-\frac{u}{r}\right)
			P_{i}(r,b) \left( \frac{u}{\lambda^r r(r-1)} \right)^{i}  
		=
		\sum_{b = 0}^{r-2}
			R^{-1}_{\textup{F}}(u)^{a}_{b+1} \ms{H}\left(r,b+1;-\frac{u}{r}\right).
	\end{equation*}
	This proves \cref{eqn:Tcomp}.
\end{proof}

Finally, we are ready to identify the deformed Theta class as the CohFT associated to a certain family of global spectral curves. Before stating the result, it is worth remarking here that given a semisimple CohFT, there is a standard procedure to construct a local spectral curve using the equivalence of \cite{DOSS14}. This procedure immediately gives a local spectral curve for the deformed Theta class $ \Theta^{r, \epsilon} $. However, this is insufficient for us as we need to take the limit as $\epsilon \to 0 $, and the local TR does not behave well in families. In general, it is not clear whether one can construct a global spectral curve from any local spectral curve. There is a partial answer in \cite{DNOS19}, but we note that it does not apply to our situation, where we do not have a Dubrovin--Frobenius manifold with a flat unit vector field. 

Here is the main result of this section.

\begin{theorem}\label{thm:SC:Wspin:v1}
	The CohFT associated to the $1$-parameter family of spectral curves $  \mathcal{S}_{\epsilon} $ on $\P^1$ given by 
	\begin{equation}\label{eqn:SC:Wspin:v1} 
		x(z) = \frac{z^r}{r} -  \epsilon z\,,
		\qquad\quad
		y(z) = - \frac{1}{z}\,,
		\qquad\quad
		\omega_{0,2}(z_1,z_2) = \frac{dz_1 dz_2}{(z_1 - z_2)^2}\,,
	\end{equation}
	is the deformed Theta class $ \Theta^{r, \epsilon} $. More precisely, the TR correlators corresponding to the spectral curve $ \mathcal{S}_{\epsilon} $ are
	\begin{equation}\label{eqn:SC:Wspin:v2} 
		\omega_{g,n}(z_1,\dots,z_n)
		=
		\sum_{a_1,\dots,a_n = 1}^{r-1}
		\int_{\overline{\mathcal{M}}_{g,n}}
			\Theta^{r,\epsilon}_{g,n}({v}_{a_1} \otimes\cdots\otimes {v}_{a_n})
			\prod_{j=1}^n \sum_{k_j \ge 0}
				\psi_j^{k_j} d\xi^{k_j,a_j}(z_j) \, ,
	\end{equation}
	where
	\begin{equation}\label{eqn:xi}
		\xi^a(z) = \frac{z^{r-a-1}}{z^{r-1} - \epsilon} \, ,
		\qquad
		d \xi^{k,a}(z) \coloneqq d\left( \left( \frac{1}{(\epsilon - z^{r-1})}\frac{d}{dz} \right)^k \xi^a(z) \right).
	\end{equation}
\end{theorem}

\begin{proof}
	First of all, we note that $ x $ in the spectral curve $ \mathcal{S}_\epsilon $ is rescaled by $ \frac{1}{r} $ as compared to \cref{eqn:tS}. This gives an overall rescaling factor of $ r^{2g-2+n} $, as one can immediately deduce from the definition of topological recursion \labelcref{eqn:TR}.
	
	The $ R $-matrix and translation of \cref{lem:DOSS:flat} correspond exactly to those computed in \cref{lem:Rmatrix} and \cref{lem:vac} using the Teleman reconstruction theorem for the deformed Theta class, up to a scaling of $ u $ by $ (-r)^{-1} $. As $ u $ is keeping track of the degrees we get a factor of $ (-r)^{-\deg} $. The TFT computed in \cref{lem:DOSS:flat} differs from the one computed in \cref{lem:tft} (after base-changing from the $ \hat{v}_a $ basis to the $ v_a $ basis) by an overall constant of $ \lambda^{D^r_{g,a}} $.

	Thus, the CohFT associated to the spectral curve $ \mathcal S_\epsilon $ is 
	\begin{equation*}
		r^{2g-2+n} \lambda^{-D^r_{g,a}}	(-r)^{-\deg}\cdot	\Theta^{r,\epsilon}_{g,n} (v_{a_1} \otimes \cdots \otimes  v_{a_n}) \,.
	\end{equation*}
	Now, we need to also write the basis of differentials in the flat basis. Upon changing basis from the $ e_i $ to the $ v_a $, we get
	\[
		\ \sum_{i=1}^{r-1} \frac{\theta^{i(\frac{r}{2}-a)}}{\sqrt{r-1}}\xi^i_{\textup{C}}(z)
		= - \frac{\iu \lambda^{a-\frac{r}{2}} }{\sqrt{r}} \frac{z^{r-a-1}}{z^{r-1} - \lambda^{r-1}}
		= - \frac{\iu \lambda^{a-\frac{r}{2}} }{\sqrt{r}} \xi^{a}(z)\,.
	\]
	In addition, in the definition of $ d \xi^{k,a}(z) $ in \cref{eqn:xi}, we have removed a factor of $ (-r)^{-1} $ (in comparison to the ones obtained from the spectral curve \labelcref{eqn:tS}), and this contributes an overall factor of $ (-r)^{-|k|} $. The global constant $ C $ contributes a factor of $ (-i\sqrt{r} \lambda^{\frac{r+2}{2}})^{2g-2+n} $.

	Putting all these factors together, we see that they cancel: 
	\begin{equation*}
		r^{2g-2+n} \lambda^{-D^r_{g,a}}	(-r)^{-\deg} 
		\left(
			- \frac{\iu \lambda^{-\frac{r}{2}} }{\sqrt{r}}
		\right)^n \lambda^{ |a|} (-r)^{-|k|}
		(-i\sqrt{r} \lambda^{\frac{r+2}{2}})^{2g-2+n}
		=
		1 \,,
	\end{equation*}
	where we impose the constraint $ \deg + |k| = 3g - 3 + n$ in order to get a non-vanishing integral over $ \Mbar_{g,n} $. Thus, we get the result.
\end{proof} 

\begin{remark}
	Notice that the function $ x $ of the spectral curve appearing in \cref{thm:SC:Wspin:v1} coincides with the function $ x $ in the spectral curve associated to the $ e_1 $-shifted Witten class (see \cite[theorem~A, part~(2)]{CCGG24} or \cite[theorem~7.1]{DNOS19}. This explains the fact that the $ R $-matrix for both these CohFTs are the same, as the $ R $-matrix prescription of \cite{DOSS14} depends only on the function $ x $ (and not on $ y $). 
\end{remark}

%––––––––––––––––––––––––––––––––––––––––––––
\subsubsection{The free energies}
%––––––––––––––––––––––––––––––––––––––––––––

In addition to producing the multidifferentials $ \omega_{g,n} $, the topological recursion also produces the \textit{free energies}, $ F_g $, which are the $ n = 0 $ analogues of the $ \omega_{g,n} $. Given a spectral curve $ (\Sigma,x,y,\omega_{0,2}) $, we define the $ F_g $ for $ g \geq 2 $ as
\begin{equation}\label{eqn:def:Fg}
	F_{g}
	\coloneqq
	-\frac{1}{2g-2} \sum_{i=1}^r \Res_{z = \alpha_i} \left[
		\left( \int_o^z \omega_{0,1}(\cdot) \right) \omega_{g,1}(z)
	\right] ,
\end{equation}
where $ o \in \Sigma $ is an arbitrary base-point for integration. The $ F_g $ are not covered by the results of \cref{thm:Eyn:DOSS} and hence we do not know whether the $ F_g $ are integrals of the CohFT associated to the spectral curve. However, in this section, we prove by direct calculation that in the case of the spectral curve for the deformed Theta class $ \mathcal S_\epsilon $ with $ r = 2 $, this is indeed the case. 

Recall from \cref{thm:dTheta} and the following \cref{rem:n=0} that the Teleman reconstruction theorem for the deformed Theta class $ \Theta^{r,\epsilon} $ does not hold for $ n = 0 $ and degree $ 3g-3 $. Moreover, in this case the $ F_g $ have been calculated explicitly in \cite{IKT19}, and thus we prove a conjecture of Kazarian--Norbury \cite{KN24} involving certain integrals of $\kappa$-classes over $\Mbar_g$. An independent proof of the same conjecture has been given in \cite{YZ} through the Hodge--BGW correspondence.

\begin{proposition}\label{prop:Fg}
	The free energies of the topological recursion for the spectral curve $ \mathcal{S}_\epsilon $ with $ r = 2 $ for $ \epsilon \neq 0 $ are given by 
	\begin{equation}\label{eqn:Fg:Theta}
		F_g
		=
		\frac{(-1)^g}{\epsilon^{2g-2}}
		\int_{\Mbar_g} \exp\left( \sum_{m>0 } s_m \kappa_m\right)
		=
		\frac{B_{2g}}{2g(2g-2)} \frac{1}{\epsilon^{2g-2}}
		\,,
	\end{equation}
	where $B_{2g}$ denotes the $(2g)$-th Bernoulli number.
\end{proposition} 

\begin{proof}
	We denote the reconstructed CohFT for $ r = 2 $ appearing in \cref{thm:dTheta} (whose explicit form appears in the proof of \cref{cor:KN}) for $ r = 2 $ by 
	\[
		K_{g,n}
		\coloneqq
		RT w_{g,n}
		=
		(-\epsilon^{2})^{2g-2+n} \exp\left( \sum_{m>0 } (-\epsilon^{2})^{-m} s_m \kappa_m \right) .
	\]
	Due to the existence of a vacuum, the class $ K_{g,n} $ satisfies the following pullback property of \cref{lem:pull-back} for all stable $ (g,n) $ including $ n = 0 $, i.e. under the forgetful map $ p \colon \Mbar_{g,1} \to  \Mbar_{g} $ we have
	\begin{equation*}
		p^* K_{g,0} = - \sum_{k\geq 0} \frac{(2k+1)!!}{\epsilon^{2k+2}} \, K_{g,1} \psi_1^k \,.
	\end{equation*}
	Plugging in the expression for $ \omega_{g,1} $ obtained as a part of \cref{thm:SC:Wspin:v1} into the definition for the $ F_{g} $ yields the following
	\begin{equation*}
		F_{g}
		=
		- \frac{1}{2g-2} \Res_{z = \epsilon} \left[
			\left( -z + \epsilon \log(z) \right) \int_{\Mbar_{g,1}} K_{g,1}
			\sum_{k \geq 0 } \psi_i^k \left(- \frac{(2k+1)!!}{(z-\epsilon)^{2k+2}} dz \right)
		\right] , 
	\end{equation*}
	where it is an easy calculation to see that $ d\xi^{k,1}(z) = - \frac{(2k+1)!!}{(z - \epsilon)^{2k+2}} dz$. Evaluating the residue gives
	\begin{equation*}
		F_{g}
		=
		\frac{1}{2g-2} \sum_{k \geq 1 } \left(
			\int_{\Mbar_{g,1}} K_{g,1} \psi_1^k
		\right) \frac{(2k-1)!!}{\epsilon^{2k}}
		= 
		- \frac{1}{2g-2} \int_{\Mbar_{g,1}} \psi_1 \, p^* K_{g,0} \,, 
	\end{equation*}
	where we plug in the pullback-property to obtain the second equality. Finally, applying the projection formula, noting that $ p_* \psi_1 = 2g-2 $, and rearranging the powers of $\epsilon$ gives the first equality in \labelcref{eqn:Fg:Theta}.

	On the other hand, up to a change of variables, the spectral curve $ \mathcal S_\epsilon $ for $ r = 2 $  coincides with the so-called Bessel curve of \cite[tables 1.1, 1.2]{IKT19} (with the identification of their $ \lambda_0 $ as $ \frac{\epsilon}{2} $) and they compute the free energies as $ F_{g} = \frac{B_{2g} }{2g(2g-2)} \frac{1}{\epsilon^{2g-2}}$. This gives the second equality in \labelcref{eqn:Fg:Theta}.
\end{proof}

%–––––––––––––––––––––––––––––––––––––––––––%
\subsection{The spectral curve for the Theta class }
%–––––––––––––––––––––––––––––––––––––––––––%

In this section, we will prove that the spectral curve obtained by taking the limit $ \epsilon \to 0 $ computes the descendant integrals of the (non-semisimple) CohFT $ \Theta^r $. The spectral curve $ \mc{S}_0 $ obtained by taking the limit of the family of spectral curves $ \mathcal{S}_{\epsilon} $ as $ \epsilon \to 0 $ is
\begin{equation}\label{eqn:S0} 
	x(z) = \frac{z^r}{r}\,,
	\qquad\quad
	y(z) = - \frac{1}{z}\,,
	\qquad\quad
	\omega_{0,2}(z_1,z_2) = \frac{dz_1 dz_2}{(z_1 - z_2)^2}\,.
\end{equation}
We will refer to this curve as the \textit{$ r $-Bessel spectral curve}. The $ r =2 $ version is known as the Bessel curve and was considered in \cite{DN18,CN19}.

In order to prove the result, we will use the following proposition that was proved in \cite{CCGG24}.

\begin{proposition}[{\cite[proposition~5.2]{CCGG24}}]\label{prop:limit:tr}
	Let $\mc{S}_{\epsilon}$ be a family of spectral curves indexed by $\epsilon\in\mathbb{C}$ such that in a neighbourhood of $\epsilon=0$, they satisfy the following assumptions.
	\begin{enumerate}
		\item $\mc{S}_{\epsilon}$ is defined by an algebraic equation linear in $x$:
		\begin{equation}
			P_{\epsilon}(x,y)=A_{\epsilon}(y) + x\,B_{\epsilon}(y) =0\,,
		\end{equation}
		where $A_{\epsilon}(y)$, $B_{\epsilon}(y)$ are polynomials in $y$ and $\epsilon$.
		
		\item For $\epsilon\neq 0$, $\mc{S}_{\epsilon}$ has $r-1$ simple ramification points, while $\mc{S}_0$ has a single ramification point of degree $r-1$ and is admissible in the sense of \cite{BBCCN24}. Moreover, the branch points are distinct.
		
		\item The multidifferentials $\omega_{g,n}(\epsilon;z_1,\dots,z_n)$ produced by topological recursion admit limits as $\epsilon\to 0$:
		\begin{equation}
			\omega_{g,n}(z_1,\dots,z_n)\coloneqq \underset{\epsilon \to 0}{\lim} \, \omega_{g,n}(\epsilon;z_1,\dots,z_n)\,.
		\end{equation} 
	\end{enumerate}
	Then the multidifferentials $\omega_{g,n}(z_1,\dots,z_n)$ satisfy the local Bouchard--Eynard topological recursion on the spectral curve $\mc{S}_0$. 
\end{proposition}

We will use the above proposition to prove the following theorem.

\begin{theorem}\label{thm:ThetaTR}
	The CohFT associated to the spectral curve $ \mathcal{S}_0 $ is the Theta class $ \Theta^{r}$. More precisely, the correlators computed by the Bouchard--Eynard topological recursion are 
	\begin{equation}\label{eqn:wgncohft}
		\omega_{g,n}(z_1,\dots,z_n)
		= 
		\sum_{a_1,\dots,a_n = 1}^{r-1}
		\int_{\overline{\mathcal{M}}_{g,n}}
			\Theta^{r}_{g,n}({v}_{a_1} \otimes\cdots\otimes {v}_{a_n})
			\prod_{j=1}^n \sum_{k_j \ge 0}
				\psi_j^{k_j} d\xi^{k_j,a_j}(z_j)\,,
	\end{equation}
	where 
	\begin{equation}
		d \xi^{k,a}(z) = \frac{(rk+a)!^{(r)}}{z^{rk+a+1}} dz\,.
	\end{equation}
	Here $m!^{(r)} = \prod_{k=0}^{\floor{(m-1)/r}} (m - kr)$ denotes the $r$-th factorial.
\end{theorem}

\begin{proof}
	Throughout this proof, we add the argument $ \epsilon $ to the correlators $ \omega_{g,n} $ and the basis of differentials $ d\xi^{k,a}(z) $ for clarity. The proof of the theorem requires two steps. 
	
	First, we need to check that the limit of the correlators constructed by TR on the curve $ \mathcal{S}_{\epsilon} $ exists, and coincides with the right-hand side of \cref{eqn:wgncohft}. This follows immediately upon taking the limit $ \epsilon \to 0 $ to the correlators computed in \cref{thm:SC:Wspin:v1}, i.e.~\cref{eqn:SC:Wspin:v2}. The limit exists by definition of the deformed Theta class, and we only need to observe that 
	\[
		\lim_{\epsilon \to 0} d \xi^{k,a}(\epsilon; z) = \frac{(rk+a)!^{(r)}}{z^{rk+a+1}} dz\,.
	\]
	The second step, which is more difficult, is to check that the limit of the correlators computed by TR on $ \mathcal{S}_\epsilon $ coincides with the correlators computed by TR on the limit curve $ \mathcal{S}_0 $. For this, we use \cref{prop:limit:tr}, which says that the above statement is true under certain assumptions:
	\begin{enumerate}
		\item The first assumption states that the equation defining $ \mathcal{S}_{\epsilon} $ is linear in $ x $ and polynomial in $y$ and~$\epsilon$. This is clearly the case for us, as we have $ P_{\epsilon}(x,y) = rxy^r - \epsilon r y^{r-1} + (-1)^{r+1} $.
		
		\item The second assumption states that the curve $ \mathcal{S}_{\epsilon} $ only has simple ramification points, and that $ \mc{S}_0 $ has a single ramification point which is admissible. This is true, as it corresponds to the case $ s = r-1 $ in the notation of \cite{BBCCN24} and thus satisfies the condition $ r = \pm 1 \pmod s $. Moreover, the branch points are given by $x_k = -(r - 1)\theta^k \epsilon^{\frac{r}{r-1}}$, which are distinct for $k = 1,\dots, r-1$.
		
		\item The third assumption states that the limit of the correlators exists, and we have just proved it in the first part of this proof.
	\end{enumerate}
	This completes the proof. 
\end{proof}

%–––––––––––––––––––––––––––––––––––––––––––%
\section{\texorpdfstring{$W$}{W}-constraints and integrability}
\label{sec:W:cnstrs:rKdV}
%–––––––––––––––––––––––––––––––––––––––––––%

The descendant potential of the Theta class can be expressed as the unique solution to a certain set of $\mc{W}$-constraints. This identification uses the equivalence of the Bouchard--Eynard topological recursion with the higher Airy structures obtained in \cite{BBCCN24}. We conjecture that this set of $ \mathcal{W} $-constraints equivalently characterises a certain $ r $-KdV tau function called the $ r $-Brézin--Gross--Witten ($ r $-BGW) tau function, first studied in \cite{MMS96}. This conjecture implies that the descendant potential of $ \Theta^r $ is an $ r $-KdV tau function. Finally we prove this conjecture for $ r = 2$ and $ r =3 $, and discuss the case of  $ r > 3 $ in detail, which we reduce to a single equation known as the string equation.

%–––––––––––––––––––––––––––––––––––––––––––%
\subsection{\texorpdfstring{$W$}{W}-constraints}
\label{sec:Wconst}
%–––––––––––––––––––––––––––––––––––––––––––%

We are interested in a very specific $ \mathcal{W} $-algebra -- the algebra $ \mathcal{W}^k(\mathfrak{gl}_r) $ at the so-called self-dual level $ k = -r + 1 $, and from this point onwards we work only at this level. An equivalent characterisation of the self-dual level is that the central change is $ c = r $. A well-known presentation of this $ \mathcal{W} $-algebra is constructed using the quantum Miura transformation \cite{FL88}: it is strongly freely generated by $ r $ fields $U^i(z)$
\begin{equation}\label{eqn:We}
	U^i(z) \coloneqq r^{i-1}  \norder {e_i \left(\chi^0(z), \ldots, \chi^{r-1}(z)\right) }\,,
	\qquad
	i = 1, \dots, r\,,
\end{equation}
where the $ \chi^j(z) $\footnote{
	The elements $( \chi^j )_{j = 0}^{r-1} $ form a basis of the Cartan subalgebra $ \mathfrak h $ of $\mathfrak{gl}_r$ with the bilinear form $ \langle \chi^i , \chi^j \rangle = \delta_{i,j} $.
} for $j = 0,\ldots,r-1 $ generate a Heisenberg (or free boson) algebra $ \mc{S}_0(\mathfrak gl_r) $ of rank $ r $. The $ e_i $ denote the elementary symmetric polynomials and the $ \norder {} $ denotes the normal ordering of the fields. A detailed study of the operator product expansions (OPEs) in this basis was carried out in \cite{Pro14}, where it is called the quadratic basis.

In \cite{BBCCN24}, these $ \mathcal{W} $-algebras were analysed thoroughly in the context of higher Airy structures, and a ``twisting" construction was used to construct examples of Airy structures. We summarize the result of the construction of \emph{loc.~cit.} and refer the reader to \cite[sections 3-4]{BBCCN24} for a complete discussion of the modules of interest. From the twisting construction, one obtains certain explicit representations of $ \mathcal{W}^k(\mathfrak{gl}_r) $ as differential operators. Consider the following differential operators
\begin{equation}\label{eqn:J}
	\widehat{J}_\ell =  
	\begin{cases}
		\hbar \frac{\partial}{\partial t_\ell}\,, & \ell  > 0\,, \\ - \ell t_{-\ell}\,,  & \ell \leq 0\,.
	\end{cases}
\end{equation}
We will refer to the $ t_i $ above as \textit{times}, in anticipation of the relation to the times of the $ r $-KdV hierarchy in the next section. Now, let us denote the fields $ U^i(z) $ in the representation referred to above as $ W^i(z) $. We assume that the fields $ W^i(z) $ have the mode expansion\footnote{
	The definition of the modes of a field here differs slightly from the definition of the modes in \cite{BBCCN24}, where they were defined as $ W^{i}(z) = \sum_{k \in \Z} W^i_k z^{-k-1} $. Here we adopt the physics convention of shifting by the conformal weight, which introduces some shifts in the conditions on the sums appearing in the explicit expression of the modes in the sequel.}  
\begin{equation}
	W^{i}(z) \coloneqq  \sum_{k \in \Z} W^i_k z^{-k-i}\,,
	\qquad
	i = 1, \dots, r\,,
\end{equation}
and in particular these $ W^i_k $ are differential operators in the times. Later, in \cref{sec:W1+inf}, we will give a natural interpretation of this representation from the point of view of the algebra $ \mathfrak{W}_{1+\infty} $.

\begin{proposition}[{\cite[corollary~4.7]{BBCCN24}}]
	The modes $ W^i_k $ have the explicit form 
	\begin{equation}\label{eqn:W:modes}
		W^i_k
		=
		\frac{1}{r} \sum_{j=0}^{\lfloor i/2 \rfloor} \hbar^{j} \frac{i!}{2^j j! (i-2j)!}
		\sum_{\substack{ p_{2j+1},\ldots, p_i \in \Z \\ \sum_l p_l = rk }} \Psi^{(j)}(p_{2j+1},p_{2j+2},\ldots, p_i) \,
		\norder {\prod_{l = 2j+1}^i \widehat{J}_{p_l}} \, .
	\end{equation}
	When $ j = i/2 $, the condition $ \sum_l p_l = r k $ on the sum is interpreted as the condition $ \delta_{k,0} $. Moreover, the function $ \Psi^{(j)} (a_{2j+1},\ldots, a_i) $ is defined for every $ i \geq 1 $, $ 0 \leq j \leq \lfloor i/2 \rfloor $  and $ a_{2j+1},\ldots, a_i \in \Z $ as the following sum over $r$-th roots of unity:
	\begin{equation}\label{eqn:higher:Bessel}
		\Psi^{(j)} (a_{2j+1},\ldots, a_i)
		\coloneqq
		\frac{1}{i!} \sum_{\substack{m_1, \ldots , m_{i}=0 \\ m_{l} \neq m_{l'}}}^{r-1}
		\left(
			\prod_{l'=1}^{j} \frac{\beta^{m_{2l' - 1}+m_{2l'}}}{(\beta^{m_{2l'}} - \beta^{m_{2l' - 1}})^2}
			\prod_{l=2j+1}^{i} \beta^{-m_{l} a_{l}}
		\right) \, ,
	\end{equation}
	where $ \beta $ is a primitive $ r $-th root of unity.
\end{proposition}

For some examples of functions $\Psi^{(j)}(a_{2j+1},\ldots, a_i)$ and some of their properties, we refer the reader to \cite[appendix A]{BBCCN24}. 

We do not recall the notion of Airy structures here for brevity (see \cite{KS18, ABCO24, BBCCN24}). We merely note that in order to get an Airy structure, and thus the $ \mathcal{W} $-constraints associated to the $ r $-Bessel spectral curve that we studied in the last section, we need to perform a \textit{dilaton shift}. The dilaton shift that we are interested in here\footnote{ In the context of the higher Airy structures of \cite[section 4.1]{BBCCN24}, we are dealing with the case $ s = r-1 $ in this paper.} is the following conjugation of the differential operators $W^i_k$:
\begin{equation}\label{eqn:H:cnstrnts}
	H^i_k \coloneqq \exp\left( - \frac{\widehat{J}_{r-1}}{(r-1)\hbar}\right) W^i_k \exp\left(  \frac{\widehat{J}_{r-1}}{(r-1)\hbar}\right).
\end{equation}
Applying the Baker--Campbell--Hausdorff formula, this amounts to a shift $ \widehat{J}_{-r+1} \to \widehat{J}_{-r+1} - 1 $, and thus we define 
\begin{equation}
	\widetilde{J}_k \coloneqq \widehat{J}_k - \delta_{k,-r+1}\,.
\end{equation}
We write the differential operators $ H^i_k $ explicitly for $ i = 1,2,3 $. First of all, we have  for any $ r \geq 2 $
\begin{equation}
\begin{aligned}
	H^1_k
	& =
	\widetilde{J}_{rk}\, ,  \\  
	H^2_k
	& =
	\frac{1}{2} \sum_{\substack{p_1,p_2 \in \Z \\ p_1+p_2 = rk }}
		\left( r \delta_{r|p_1,p_2} - 1 \right)
		\norder{\widetilde{J}_{p_1} \widetilde{J}_{p_2}}
	-
	\hbar \frac{(r^2-1)}{24} \delta_{k,0} \,.
\end{aligned}
\end{equation}
In the above equation, we denoted with $\delta_{r|a_1,\dots,a_n}$ the function taking value one if $r$ divides all $a_i$, and zero otherwise. When $ r \geq 3 $, the mode $H^3_k$ is given by
\begin{equation}
	H^3_k
	=
	\frac{1}{6} \sum_{\substack{p_1,p_2,p_3 \in \Z \\ p_1+p_2+p_3 = r k }}
		\left(
			r^2 \delta_{r|p_1,p_2p_3}
			- r ( \delta_{r|p_1,p_2+p_3} + \cdots )
			+ 2
		\right)
		\norder{ \widetilde{J}_{p_1} \widetilde{J}_{p_2} \widetilde{J}_{p_3} }
	-
	\hbar \frac{(r-2)(r^2-1)}{24} \widetilde{J}_{rk}\,. 
\end{equation}
The dots indicate other terms necessary to enforce symmetry under permutation. 

Using the formalism of higher Airy structures, it was proved in \cite{BBCCN24} that one can recast the Bouchard--\/Eynard topological recursion as the unique solution to a set of $ \mathcal{W} $-algebra constraints. In our situation, the statement is the following.

\begin{theorem}[{\cite[theorem~5.27]{BBCCN24}}]\label{thm:HAS}
	The set of $ \mathcal{W} $-constraints
	\begin{equation}
		H^i_k Z = 0\,, \qquad \text{ for all } \,\, k \geq -i+2\,, \,\, i = 1,\dots,r\,,
	\end{equation}
	forms an Airy structure, and thus there exists a unique solution  $ Z(\hbar;\bm{t}) $ of the form 
	\begin{equation}\label{eqn:ASZ}
		Z(\hbar;\bm{t})
		=
		\exp\left(
			\sum_{\substack{ g \geq 0, n \geq 1 \\ 2g-2+n > 0}}
				\frac{\hbar^{g-1}}{n!}
				\sum_{a_1,\ldots, a_n > 0}
					F_{g,n}[a_1,\ldots, a_n] \,\,
					t_{a_1} \cdots t_{a_n} \right),
	\end{equation}
	with the scalars $F_{g,n}[a_1,\ldots, a_n]$ being symmetric in the entries $a_i$. In addition, these scalars coincide with the expansion coefficients of the Bouchard--Eynard topological recursion correlators computed from the $ r $-Bessel spectral curve $ \mc{S}_0 $:
	\begin{equation}\label{eqn:Fgn}
		\omega_{g,n} ( z_1, \ldots, z_n)
		=
		\sum_{a_1, \ldots, a_n > 0} F_{g,n}[a_1, \ldots, a_n] \prod_{i=1}^n \frac{dz_i}{z_i^{a_i+1}} \,.
	\end{equation}
\end{theorem}

Notice that the scalars $ F_{g,n}[a_1, \ldots, a_n] $ vanish identically as soon as one of the $ a_i \equiv 0 \pmod r$ (as one can see immediately from the identification in \cref{thm:ThetaTR}). This is consistent with the constraint $ H^1_k Z = \frac{\de}{\de t_{kr}} Z =  0 $ for all $ k \geq 1 $, which implies that the $ 0 \pmod r $ times do not appear in $ Z $.

We remark that the above result is a special case of theorem~5.27 of \textit{loc.~cit.} where, in the notations there, we set $ s = r-1 $, $ \hat{\Phi} = 1 $ and $ F_{0,1}[-a] = -\delta_{a,r-1} $. Combining this result with \cref{thm:ThetaTR}, we get the main theorem of this section. First, let us define the descendant potential of the Theta class.

\begin{definition}
	The \textit{descendant potential} $Z^{\Theta^r}(\hbar;\bm{t})$ of the CohFT $ \Theta^r $ is defined as 
	\begin{equation}
		Z^{\Theta^r}
		=
		\exp\left(
			\sum_{\substack{ g \geq 0, n \geq 1 \\ 2g-2+n > 0 }} \frac{\hbar^{g-1}}{n!}
			\sum_{a_1,\ldots, a_n =1}^{r-1} \int_{\Mbar_{g,n}}
				\Theta^{r}_{g,n}( v_{a_1} \otimes \cdots \otimes v_{a_n} )
				\prod_{j=1}^n  \sum_{k_j \geq 0}
					\psi_j^{k_j} \,\, (rk_j+a_j)!^{(r)} t_{a_j + r k_j}
			\right).
	\end{equation}
\end{definition}

\begin{remark}\label{rem:descendant}
	If we extend the definition of the Theta class to take values not only between $ 1 \leq a_i \leq r-1 $, but instead extend it to all $ a_i \geq 1 $, we can express the descendant potential in a slightly cleaner form as 
	\begin{equation}
		Z^{\Theta^r}
		=
		\exp\left(
			\sum_{\substack{ g \geq 0, n \geq 1 \\ 2g-2+n > 0}} \frac{\hbar^{g-1}}{n!}
			\sum_{a_1,\ldots, a_n > 0} \int_{\Mbar_{g,n}}
				\Theta^{r}_{g,n}(v_{a_1} \otimes \cdots \otimes v_{a_n})
				\prod_{j=1}^n \braket{ a_j } t_{a_j}
		\right),
	\end{equation}
	where  $ \braket{a_j} $ is  the remainder of the Euclidean division of $a_j$ by $r$. This is easy to deduce using property~(\labelcref{prop:shift:ai}) in \cref{prop:Chiodo:properties}, and shows that the factors of $ (rk_j+a_j)!^{(r)} $ appear naturally. We would like to thank A.~Chiodo for pointing this out to us. 
\end{remark}

\begin{theorem}\label{thm:Wconst}
	The descendant potential $ Z^{\Theta^r} $ is the unique solution to the following set of $ \mathcal{W} $-algebra constraints 
	\begin{equation}
		H^i_k \, Z^{\Theta^r} = 0\,, \qquad \text{ for all } \,\, k \geq -i+2\,, \,\, i = 1, \dots, r\,,
	\end{equation}
	where we recall that the  differential operators $ H^i_k $ were defined in \cref{eqn:H:cnstrnts}.
\end{theorem}

\begin{proof}
	The identification of the Bouchard--Eynard topological recursion on the $ r $-Bessel spectral curve in \cref{thm:HAS} shows that the unique solution to the $ \mathcal{W} $-constraints is  the $ F_{g,n} $ of \cref{eqn:Fgn}. In \cref{thm:ThetaTR}, we proved that the topological recursion correlators on the $ r$-Bessel spectral curve are the descendant integrals of $ \Theta^r $. Thus, we have the identification
	\[
		F_{g,n}[a_1,\ldots, a_n]
		=
		\int_{\Mbar_{g,n}} \Theta^{r}_{g,n}(v_{a_1} \otimes \cdots \otimes v_{a_n})
		\prod_{j=1}^n \braket{ a_j } ,
	\]
	where all the $ a_i > 0 $, and we are using the observation in \cref{rem:descendant}.
\end{proof}

Our goal is to interpret the above set of $ \mathcal{W} $-constraints as characterising a tau function of the $r$-KdV hierarchy. To this end, we will consider a specific $r$-KdV tau function in \cref{subsec:rBGW}, and study its Kac--Schwarz operators, which in turn will give us $ \mathcal{W} $-constraints that act by a constant on the associated tau function. Before discussing integrability, we show that it suffices to find a very small subset of the  $ \mathcal{W} $-constraints appearing in \cref{thm:Wconst}: the $r$-th reduction condition and the string equation.

\begin{proposition}\label{prop:tau}
	Let $Z(\hbar;\bm{t})$ be a function of the form
	\begin{equation}\label{eqn:genus:expansion}
		Z(\hbar;\bm{t})
		=
		\exp\left(
			\sum_{\substack{ g \geq 0, n \geq 1 \\ 2g-2+n > 0}}
				\frac{\hbar^{g-1}}{n!}
				\sum_{a_1,\ldots, a_n > 0}
					\Phi_{g,n}[a_1,\ldots, a_n] \,\,
					t_{a_1} \cdots t_{a_n} \right)
	\end{equation}\
	satisfying the following conditions.
	\begin{enumerate}
		\item[i)] {\sc Symmetry}. The scalars $\Phi_{g,n}[a_1,\ldots, a_n]$ are symmetric in the entries $a_i$.

		\item[ii)] {\sc $r$-th reduction.} There exist constants $\nu_k$ such that
		\begin{equation}
			H^1_k \, Z = \nu_k \, Z\,, \qquad \text{ for all } k \geq 1\,.
		\end{equation}

		\item[iii)] {\sc String equation.} There exists a constant $\mu$ such that
		\begin{equation}
			H^{r}_{-r+2} \, Z = \mu \, Z\,.
		\end{equation}
	\end{enumerate}
	Then $Z$ coincides with the descendant potential of the Theta CohFT: $Z = Z^{\Theta^r}$.
\end{proposition}

\begin{proof}
	We will prove that if a function satisfying the conditions of the lemma exists, it must satisfy the full set of $ \mathcal{W} $-constraints appearing in \cref{thm:HAS}: $H^i_k \, Z = 0$ for all $k \geq -i+2\,, \,\, i = 1,\dots,r$. Once we establish this statement, the uniqueness part of \cref{thm:Wconst} implies the equivalence with the descendant potential $ Z^{\Theta^r} $.
	
	Now, we assume that a function $ Z $ satisfying the conditions of the proposition exists.  We can determine the constants $ \nu_k $ using the commutation relations of the $ \mathcal{W}^{-r+1}(\mathfrak{gl}_r) $ algebra. The following commutation relations can be derived using the OPEs found in \cite{Pro14} or by direct computation. For any $ \ell \geq 1 $, we have
	\[
		\bigl[ H^1_m, H^\ell_n \bigr] = \frac{1}{r} (r-\ell+1) H^{\ell-1}_{m+n}\,.
	\]
	Applying the above operator to $ Z $ and choosing $ \ell = r $, $ m = k \ge 1$ and $ n = -r+2$, we get 
	\[
		0 = \bigl[ H^1_k, H^r_{-r+2} \bigr] \, Z
		=
		\frac{1}{r} H^{r-1}_{-r+2+k} \, Z\,, \qquad \text{ for } -r+2+k \geq -r+3\,.
	\]
	Thus, we see that $ H^{r-1}_{k} $ acts by $ 0 $ on $ Z $ for $ k \geq -r+3 $. Now, by applying the same procedure successively (in decreasing order) to $ \ell = r-1,r-2,\ldots,2  $, we see that 
	\begin{equation}\label{eqn:almostvan}
		H^i_k Z = 0\,, 
		\qquad \text{ for all } \,\, k \geq -i+2\,, \,\, i = 1,\dots,r-1 \,.
	\end{equation}
	Notice that the calculation for $ \ell = 2 $ forces the constants $ \nu_k = 0 $ for $ k \geq 1 $. 
	
	Finally, the only operators whose action we need to study are the $ H^r_k $ for $  k \geq -r+ 2 $. For this purpose, we use the following commutation relations
	\[
	\begin{split}
		\bigl[ H^2_m, H^\ell_n \bigr]
		=
		(m+n-\ell m) H^\ell_{m+n} &
		+ \binom{m +1}{3} \frac{(r-\ell+2)(r-\ell+1)}{2 r^2 } H^{\ell-2}_{m+n}  \\
		&
		+ \frac{(r-\ell+1)}{r} \left( \sum_{i+j = m+n} (m-i) \norder{ H^1_i H^{\ell-1}_j} \right).
	\end{split}
	\]
	Applying the above operator to $ Z $, choosing $ \ell = r $, $ m = k \ge 0$ and $ n = -r+2$, together with the constraints \labelcref{eqn:almostvan}, we find 
	\[
		0 = \bigl[ H^2_k, H^r_{-r+2} \bigr] \, Z
		=
		\bigl( k(1-r)-r+2 \bigr) H^{r}_{-r+2+k} \, Z\,,
		\qquad \text{ for } -r+2+k \geq -r+2\,.
	\]
	This proves the statement.
\end{proof}

Thus, our goal is now to find a candidate KP tau function, study its associated symmetries using the Kac--Schwarz formalism, and obtain the $r$-th reduction constraints $ H^1_k $ for $ k \geq 1 $ (i.e.~$r$-KdV) as well as the string equation $ H^r_{-r+2} $. This would prove that the descendant potential $ Z^{\Theta^r} $ coincides with the candidate $r$-KdV tau function.

%–––––––––––––––––––––––––––––––––––––––––––%
\subsection{KP and its symmetries} 
%–––––––––––––––––––––––––––––––––––––––––––%

The Kadomtsev--Petviashvili (KP) hierarchy is an infinite set of evolutionary differential equations in infinitely many variables. In this section we briefly review the basic facts about the KP hierarchy from the Sato Grassmannian point-of-view \cite{JM83}, and introduce the concept of Kac--Schwarz operators to describe symmetries of the tau functions. We adopt the notations of \cite{Ale15}, and we refer the reader there for further details.

%–––––––––––––––––––––––––––––––––––––––––––%
\subsubsection{Fock space, free fermions and tau functions}
%–––––––––––––––––––––––––––––––––––––––––––%

Let $V = z.\C[z] \oplus \C\bbraket{z^{-1}}$ be the infinite-dimensional vector space of formal Laurent series in $z^{-1}$, which comes with the natural decomposition $V = V_+ \oplus V_-$. 

\begin{definition}
	Define the (big cell of the) \textit{ Sato Grassmannian} $\mathrm{Gr}$ as the set of all linear subspaces $H \subset V$ such that the projection $p_+ \colon H \to V_+$ is a linear isomorphism. 
\end{definition}

The Sato Grassmannian is Plücker embedded $\mathrm{Gr} \to \P \mathfrak{F}_0$ into the charge zero sector of the \textit{fermionic Fock space}, denoted  $\mathfrak{F}_0$. The fermionic Fock space  $\mathfrak{F}_0$ is defined as the span of all one-sided infinite wedge products $z^{k_1} \wedge z^{k_2} \wedge z^{k_3} \wedge \cdots$ such that  $k_i = i$ for $i$ sufficiently large. We call $\ket{0} = z^{0} \wedge z^{1} \wedge z^{2} \wedge \cdots$ the vacuum, and its dual vector $\bra{0}$ in $\mathfrak{F}^{\ast}$ the covacuum. 

Under the Pl\"ucker embedding, a point in the Sato Grassmannian corresponds to a semi-infinite wedge representative as
\begin{equation}
	\Phi_{1} \wedge \Phi_{2} \wedge \Phi_{3} \wedge \cdots,
	\qquad
	\Phi_{i} \in V = z.\C[z] \oplus \C\bbraket{z^{-1}}\,,
\end{equation}
where $\Phi_{i}(z) = z^{i-1} + O(z^{i-2})$.

Let us introduce the free fermionic operators $\psi_k, \, \psi_k^{\dag}$, $k \in \Z$, satisfying the usual anti-commutation relations and generating an infinite dimensional Clifford algebra:
\begin{equation}
	\bigl\{ \psi_k, \psi_l^{\dag} \bigr\} = \delta_{k,l}\,,
	\qquad\quad
	\bigl\{ \psi_k, \psi_l \bigr\} = \bigl\{ \psi_k^{\dag}, \psi_l^{\dag} \bigr\} = 0\,.
\end{equation}
We can collect them in generating series, called fermionic fields:
\begin{equation}
	\psi(z) = \sum_{k \in \Z} \psi_k \, z^{k}\,,
	\qquad\quad
	\psi^{\dag}(z) = \sum_{k \in \Z} \psi_k^{\dag} \, z^{-k} \,.
\end{equation}
The Lie algebra associated to the symmetry group of the Sato Grassmannian is a central extension of the bi-infinite general linear algebra, denoted $\widehat{\mathfrak{gl}}(\infty)$. This algebra can be realised by considering normally ordered bilinear combinations of fermions, sometimes referred to as bosons:
\begin{equation}
	\widehat{\mathfrak{gl}}(\infty) = \Set{
		\sum_{k,l \in \Z} a_{k,l} \norder{ \psi_{k} \psi_{l}^{\dag} } 
		|
		\text{$a_{k,l} \ne 0$ for only finitely many values of $k - l$}
	}.
\end{equation}
 Examples of elements in $\widehat{\mathfrak{gl}}(\infty)$ are the bosonic currents, which span a Heisenberg subalgebra:
\begin{equation}
	J_n = \hbar^{1/2} \sum_{k \in \Z} \norder{ \psi_k \psi_{k + n}^{\dag} } \, ,
	\qquad \qquad
	\bigl[ J_m, J_n \bigr] = \hbar m  \, \delta_{m+n,0} \, .
\end{equation} Here we add a formal parameter $ \hbar^{1/2} $ which can be viewed as an additional grading parameter\footnote{
	There are a couple of different conventions for the use of $ \hbar $ as a grading parameter. Here we stick to the use of $ \hbar $ as  in the formalism of Airy structures \cite{BBCCN24,KS18,ABCO24}.
}, and work over the field $ \C  (\!(\hbar^{1/2})\!) $.

A bosonic description of the KP hierarchy is in terms of tau functions. The bosonic counterpart of $\mathfrak{F}_0$ is the space $\C\bbraket{\bm{t}}$ of functions depending on an infinite set of variables $\bm{t} = (t_1, t_2, \dots )$ called the times.

\begin{definition}\label{def:bfcor}
	Define the \textit{ boson-fermion correspondence} as the linear isomorphism
	\begin{equation}
		\mathfrak{F}_0 \longrightarrow \C\bbraket{\bm{t}}\,,
		\qquad\quad
		\ket{\omega} \longmapsto \braket{0| e^{J_+(\bm{t})} | \omega },
	\end{equation}
	where $J_+(\bm{t}) = \sum_{n \ge 1} \hbar^{-1}  t_n J_n$.
\end{definition} 
A \textit{tau function} of the KP hierarchy is defined as the image of an element of the Sato Grassmannian under the boson-fermion correspondence. Tau functions are well-defined up to multiplication constants. 

%–––––––––––––––––––––––––––––––––––––––––––%
\subsubsection{Symmetries and the \texorpdfstring{$\mathfrak{W}_{1+\infty}$ algebra}{W algebra}}\label{sec:W1+inf}
%–––––––––––––––––––––––––––––––––––––––––––%

The boson-fermion correspondence allows us to translate the infinitesimal symmetries of the Sato Grassmannian described in terms of bosonic operators of $\widehat{\mathfrak{gl}}(\infty)$ into differential operators that act as infinitesimal symmetries of the KP hierarchy on $\C\bbraket{\bm{t}}$.

Indeed, from the commutation relations of the bosonic currents we have 
\begin{equation}
	\bra{0} e^{J_+(\bm{t})} J_n =
	\begin{cases}
		 \hbar \frac{\de}{\de t_n} \bra{0} e^{J_+(\bm{t})}\,, & \text{if } n > 0\,, \\
		- n t_{-n} \bra{0} e^{J_+(\bm{t})}\,,	& \text{if } n \le 0\,. \\
	\end{cases}
\end{equation}
As a consequence, for any operator $W$ which is a combination of bosonic currents, there exists an operator $\widehat{W}$ acting on $\C\bbraket{t}$ such that
\begin{equation}
	\widehat{W} \bra{0} e^{J_+(\bm{t})}
	=
	\bra{0} e^{J_+(\bm{t})} W\, .
\end{equation}
For the bosonic currents $ J_n $, we have the  explicit operators $ \widehat{J}_n $ as  defined earlier in \cref{eqn:J}.

An important subalgebra of $\widehat{\mathfrak{gl}}(\infty)$, denoted $\mathfrak{W}_{1+\infty}$, is the algebra generated by the elements
\begin{equation}
	W^{(m+1)}_n
	= 
	- \hbar^{(m+1)/2} \Res_{z} \left(
		z^{-1} \norder{ \psi(z) z^{m+n} \de_z^{m} \psi^\dag(z) }
	\right) dz\,,
	\qquad
	m, n \in \Z\,, \, m \ge 0 \,.
\end{equation}
More generally, we can associate to any element of  the algebra of diffeomorphisms of the circle
\begin{equation}
	\mathfrak{w}_{1+\infty}
	=
	\Set{
		z^{\alpha} \left(  {   \hbar^{1/2} } \frac{\de}{\de z} \right)^{\beta} | \alpha, \beta \in \Z, \, \beta \ge 0
	},
\end{equation}
a bosonic operator in $\mathfrak{W}_{1+\infty}$ by the following assignment: for $ a \in \mathfrak{w}_{1+\infty} $,
\begin{equation}\label{eqn:wtoW}
	a \longmapsto W_a =  \hbar^{1/2} \Res_{z=0} \left( z^{-1} \colon \psi(z) a \psi^{\dag}(z) \colon \right) dz\,.
\end{equation}
With this notation, $W^{(m+1)}_k = W_{-  \hbar^{m/2} z^{m+n} \de_{z}^m}$. In particular, one can check that $\mathfrak{W}_{1+\infty}$ is a central extension of $\mathfrak{w}_{1+\infty}$,  which carries a Virasoro field of central charge $ 1 $ defined by 
\begin{equation}
	L_n \coloneqq W_{-z^n(z\hbar^{1/2}\partial_z + \frac{n+1}{2})} \,.
\end{equation}

\begin{definition}
	Let $H = \mathrm{span}\set{ \Phi_1, \Phi_2, \dots }$ be a point in the Sato Grassmannian, and denote by $\tau$ the corresponding tau function. An operator $a \in \mathfrak{w}_{1+\infty}$ is called a \textit{Kac--Schwarz operator} for $\tau$ if the corresponding point of the Sato Grassmannian is stabilised by $a$: $a H \subset H$.
\end{definition}

The fundamental property of Kac--Schwarz operators is that they act as constants on the associated tau functions \cite{ASvM95,Ale15}.

\begin{proposition}\label{prop:KS:op}
	If $a \in \mathfrak{w}_{1+\infty}$ is a Kac--Schwarz operator for $\tau$, then
	\begin{equation}
		\widehat{W}_a \tau = C \tau\,,
	\end{equation}
	for some constant $C$.
\end{proposition}

\begin{example}[{$r$-KdV hierarchy}]\label{rem:rKdV}
	Consider $a = z^r$; the corresponding operator $\widehat{W}_a$ acting on functions of times is simply $\widehat{J}_r =  \hbar \frac{\de}{\de t_r}$. The above proposition states that, if a point $H$ in the Sato Grassmannian satisfies $z^r H \subset H$, then the associated tau function satisfies $ \hbar \frac{\de}{\de t_r} \tau = C \tau$, or equivalently $\hbar \frac{\de}{\de t_r} \log(\tau) = C$. Since $H$ is invariant under all powers of $a$, we obtain that the $\tau$ function is of the form
	\begin{equation}
		\tau = e^{\hbar^{-1} \sum_{k \ge 1} a_k t_{kr}} \, \tau' \,,
	\end{equation}
	for some constants $a_k$ and with $\tau'$ independent of $t_{kr}$. It is easy to check that $\tau'$ is still a KP tau function. A tau function satisfying the above property is called a tau function of the \emph{$r$-KdV hierarchy} (or $r$-th Gel'fand--Dickey hierarchy).
\end{example}

%–––––––––––––––––––––––––––––––––––––––––––%
\subsection{\texorpdfstring{$r$}{r}-KdV for the Theta CohFT}
%–––––––––––––––––––––––––––––––––––––––––––%

In this section, we first consider a specific $ r $-KdV tau function, called the $ r $-Brézin--Gross--Witten tau function, and study the Kac--Schwarz operators associated to it. We then study the associated $ \mathcal{W} $-constraints and compare them with the $ \mathcal{W} $-constraints characterising the descendant potential of the Theta class. For $ r = 2 $ and $3 $, we obtain a complete match and thereby conclude that the descendant potential $ Z^{\Theta^r} $ is an $ r $-KdV tau function.

%–––––––––––––––––––––––––––––––––––––––––––%
\subsubsection{The $ r $-BGW tau function}
\label{subsec:rBGW}
%–––––––––––––––––––––––––––––––––––––––––––%

The Brézin--Gross--Witten (BGW) matrix model was introduced in the '80s in the context of lattice gauge theory \cite{BG80, GW80}. Originally defined as an integral over the space of unitary matrices, it can also be described as an integral\footnote{We remark that the matrix models considered here are formal. That is, the matrix integrals are computed by Taylor expanding the exponential of cubic and higher terms, then subsequently exchanging summation and integration.} over the space of Hermitian matrices. In this paper, we are interested in the following generalisation of the BGW matrix model first studied by Mironov--Morozov--Semenoff \cite{MMS96}, depending on a integer parameter $r \ge 2$:
\begin{equation}
	Z^{r\textup{-BGW}}(\Lambda)
	=
	\frac{1}{C_N} \int_{\mc{H}_N}
		e^{ -\hbar^{-1/2} \,
		\Tr\left( \frac{M^{1-r}}{r(1-r) } + \Lambda M + \hbar^{1/2} N \log(M)\right)}
		[dM] \, .
\end{equation}
We call the above model the \textit{$ r $-BGW matrix model}. Here $C_N$ is an irrelevant normalisation factor, $\Lambda$ is a Hermitian matrix called the external field, and $[dM]$ is the standard Lebesgue measure on the space $\mc{H}_N$ of $N \times N$ Hermitian matrices given by
\begin{equation}
	[ dM ]
	=
	2^{- \frac{N}{2}} \left( \frac{N}{\pi} \right)^{\frac{N^2}{2}}
	\prod_{i=1}^N dM_{i,i}
	\prod_{i < j} d\Re(M_{i,j}) \, d\Im(M_{i,j})\, .
\end{equation}
Notice that the $r$-BGW model is very similar to the matrix integral solution to two-dimensional $W_r$-gravity \cite{AvM92} corresponding to the $r$-Witten--Kontsevich tau function for the $r$-KdV hierarchy. Indeed, the $r$-BGW matrix model can be obtained by replacing $ r $ in the matrix model considered in \cite{AvM92} by $ -r $, together with the addition of the $\log$ factor. 

The $r$-BGW model is an instance of matrix models with external field, which are known to define KP tau functions for $N \to +\infty$ in the times given by the Miwa parametrisation
\begin{equation}
	t_k \coloneqq \frac{1}{k} \Tr \Lambda^{-k}\, .
\end{equation}
See for instance \cite{Ale15} and references therein. The tau function can be described as a point in the Sato Grassmannian $ H = \mathrm{span}\set{ \Phi_1, \Phi_2, \dots } $, where the formal power series $\Phi_i$ are the asymptotic expansions of some integrals defined from the matrix model. In the case of the $ r $-BGW matrix model, we obtain
\begin{equation}
	\Phi_i(z)
	\coloneqq
	e^{-S(z)} \, \Psi_i \left(\tfrac{z^r}{r} \right)\,,
	\qquad \text{ for } i \ge 1 \,,
\end{equation}
where $S(z) = - \hbar^{-1/2} \frac{z^{r-1}}{r-1} - \frac{1}{2} \log(z^{r-1})$, and the $\Psi_i$ are defined via (the asymptotic expansion of)the integral representation
\begin{equation}
	\Psi_i(x)
	\coloneqq
	\hbar^{i/2-1/4} \frac{1}{\sqrt{2\pi}}
	\int_{\Gamma} e^{ \hbar^{-1/2} \left(\frac{w^{1-r}}{r(1-r)} - xw\right) } \frac{dw}{w^i} \,.
\end{equation}
Here $\Gamma$ is the Hankel contour.

In \cref{lem:r:Bessel}, we prove that the $\Phi_i$ are formal power series that have the right asymptotic behaviour as $z \to \infty$: $\Phi_i(z) \sim z^{i-1} + O(z^{i-2})$. Thus, for the purposes of this article, we will take the tau function determined by $H$ as the definition of $Z^{r\textup{-BGW}}$.

\begin{definition}
	Define the \emph{$r$-BGW tau function} $Z^{r\textup{-BGW}}(\hbar;\bm{t})$ as the KP tau function associated to the point $ H = \mathrm{span}\set{ \Phi_1, \Phi_2, \dots } $ in the Sato Grassmannian.
\end{definition}

Before studying the Kac--Schwarz operators associated to this KP solution, let us characterise the functions $\Psi_i$  as  solutions  to certain differential equations.

\begin{definition}
	We define the \textit{$ r $-Bessel quantum curve} as the $ r $-th order differential equation given by
	\begin{equation}\label{eqn:qc}
		\left( (-1)^r r \left(  \hbar^{1/2}  \frac{d}{dx} \right)^{r-1} x   \hbar^{1/2} \frac{d}{dx} - 1 \right) \Psi(x) = 0\,.
	\end{equation}
\end{definition}

\begin{lemma}\label{lem:r:Bessel}
	For any $r \ge 2$, the $r$-Bessel quantum curve is solved by the function
	\begin{equation}
		\Psi(x)
		=
		\hbar^{-1/4} \frac{1}{\sqrt{2\pi}} \int_{\Gamma} e^{ \hbar^{-1/2} \left(\frac{w^{1-r}}{r(1-r)} - xw\right) } \frac{dw}{w} \, .
	\end{equation}
	Moreover, the function $\Psi(x)$ and its (anti-)derivatives $\Psi_i = (-\hbar^{-1/2})^{i-1} \Psi^{(1-i)}$, $i \in \Z$, admit the following asymptotic expansions as $x \to + \infty$:
	\begin{equation}
		\Psi_i(x)
		=
		\hbar^{-1/4} 
		\frac{1}{\sqrt{2\pi}} \int_{\Gamma} e^{   \hbar^{-1/2} \left(\frac{w^{1-r}}{r(1-r)} - xw\right) } \frac{dw}{w^i}
		\sim e^{S(z)} \Bigl( z^{i-1} + O\bigl( z^{i-r} \bigr) \Bigr)\, ,
	\end{equation}
	where $x(z) = \frac{z^r}{r}$ and $S(z) = - \hbar^{-1/2} \frac{z^{r-1}}{r-1} - \frac{1}{2} \log(z^{r-1})$. The solution admitting the above asymptotic expansion is unique.
\end{lemma}

\begin{proof}
	Observe that
	\[
	 \left( (-1)^r r \left(
		\hbar^{1/2} \frac{d}{dx} \right)^{r-1} x \hbar^{1/2} \frac{d}{dx} - 1
		\right) \Psi(x) =
		-  \hbar^{1/4} r \sqrt{\frac{1}{2\pi}} \int_{\Gamma}
		\frac{d}{dw} \left(
		w^{r-1} 
		e^{ \hbar^{-1/2} \left(\frac{w^{1-r}}{r(1-r)} -  xw \right) }
		\right) dw
		= 0.
	\]
	Hence, $\Psi$ solves the ODE. The asymptotic behaviour of $\Psi_i$ can be computed using the steepest descent method (see for instance analogous computations in \cite{CCGG24}). Uniqueness of the solution follows from the fact that the ODE uniquely determines the coefficients of the asymptotic expansion.
\end{proof}

Notice that for $r = 2$, the above equation reduces to the modified Bessel equation of order zero after the identification $x = \frac{z^2}{2}$: 
\begin{equation}
	\left( \hbar z^2 \frac{d^2}{dz^2} + \hbar z \frac{d}{dz} - z^2 \right) \Psi = 0\,.
\end{equation}
In particular, the functions $\Psi_i(x)|_{x=z^2/2}$ coincide with the modified Bessel functions of the second kind, up to a normalisation factor:
\begin{equation}
	\Psi_{i}\left(\tfrac{z^2}{2}\right) = \hbar^{-1/4} \sqrt{\frac{2}{\pi}} \, z^{i-1} \, K_{1-i}\left(\hbar^{-1/2}z\right).
\end{equation}
In line with the above remark, the ODE from \cref{lem:r:Bessel} for general $r$ can be thought of as a higher Bessel equation, and the functions $\Psi_i$ as higher Bessel functions.

%–––––––––––––––––––––––––––––––––––––––––––%
\subsubsection{Kac--Schwarz operators for the $ r $-BGW tau function }
%–––––––––––––––––––––––––––––––––––––––––––%

We would like to understand the Kac--Schwarz operators that uniquely characterise the tau function. We start using the analysis done in \cite{MMS96}, but the operators found there do not uniquely specify the tau function. For $ r =2 $ and $ r = 3 $, we find Kac--Schwarz operators that uniquely specify the tau function and produce the $ \mathcal{W} $-constraints we are looking for, but we are unable to do so for  $ r \geq 4 $. For $ r = 2 $, these operators were first derived in \cite{Ale18}. 

Recalling the identification $ x = x(z) = \frac{z^r}{r} $, consider the following operators in $ \mathfrak{w}_{1+\infty} $,
\begin{equation}
\begin{aligned}
	A \coloneqq x\, ,
	\qquad 
	B \coloneqq x \hbar^{1/2} \frac{d}{dx}\,,
	\qquad
	C \coloneqq \hbar \frac{d}{dx} x \frac{d}{dx}\,,
\end{aligned}
\end{equation}
and the following conjugated forms of them 
\begin{equation}
\begin{aligned}
	a &\coloneqq
	e^{-S} A e^{S}
	=
	\frac{z^r}{r}\, , \\
	b &\coloneqq
	e^{-S} B e^{S}
	=
	\frac{1}{r} \left( z \hbar^{1/2} \frac{d}{dz} - z^{r-1} - \hbar^{1/2} \frac{r-1}{2} \right)\,, \\
	c &\coloneqq
	e^{-S} C e^{S}
	=
	\frac{1}{r} \left( \frac{1}{z^{r-2}} \hbar \frac{d^2}{dz^2} + \hbar \frac{2-r}{z^{r-1}} \frac{d}{dz} - 2 \hbar^{1/2} \frac{d}{dz} + z^{r-2} + \hbar \frac{(r-1)^2}{4}\frac{1}{z^r} \right)\,.
\end{aligned}
\end{equation} 

\begin{theorem}\label{thm:KS:operators}
	The following holds.
	\begin{itemize}
		\item {\sc $r$-th reduction.} The operator $a$ is Kac--Schwarz for $Z^{r\textup{-BGW}}$. In particular, $Z^{r\textup{-BGW}}$ is a tau function of the $r$-KdV hierarchy.

		\item {\sc Constraints.} The operators $b$ and $c$ are Kac--Schwarz for $Z^{r\textup{-BGW}}$.

		\item {\sc Uniqueness.} When $ r = 2 $ or $3$ the operators $a$, $b$ and $c$ determine the tau function uniquely.
	\end{itemize}
\end{theorem}

\begin{proof}
	Let us prove first that $H$ is stabilised by the operators $a$, $b$ and $c$. Up to conjugation by $e^{-S}$ and the  identification  $x = x(z)$, we can work with the operators $A = x$, $B = x  \hbar^{1/2}  \frac{d}{dx}$ and $C =  \hbar \frac{d}{dx}  x  \frac{d}{dx}$, and the functions $\Psi_i$. This said, we find
	\[
		A \Psi_i
		=
			 \frac{\Psi_{i+r}}{r} - i \hbar \Psi_{i+1}
		\in H\,,
	\]
	using the relation
	\[
		\frac{d}{dw} \left( e^{   \hbar^{-1/2} \left(\frac{w^{1-r}}{r(1-r)} -    xw\right) } \frac{1}{w^i} \right)
		=
		e^{   \hbar^{-1/2} \left(\frac{w^{1-r}}{r(1-r)} -    xw\right) } \left( \frac{{   \hbar^{-1/2}} }{r w^{i+r}} - \frac{i}{w^{i+1}} -  \hbar^{-1/2}  \frac{x}{w^i} \right)
	\]
	and then integrating by parts. This proves that $a$ is Kac--Schwarz for $\tau$. For $b$, we obtain 
	\begin{equation}\label{eqn:Baction}
		B \Psi_i
		= - A  \Psi_{i-1}
		= -  \frac{\Psi_{i+r-1}}{r} + (i-1) \hbar \Psi_{i}
		\in H\,,
	\end{equation}
	from \hbox{$\hbar^{1/2} \frac{d}{dx} \Psi_i = - \Psi_{i-1}$}.
	The operator $ C $ acts as 
	\begin{equation}\label{eqn:Caction}
		C \Psi_i =  \frac{\Psi_{i+r-2}}{r} - (i-1) \hbar \Psi_{i-1}\,,
	\end{equation}
	and thus $ c $  is also Kac--Schwarz. The $r$-KdV statement follows from \cref{rem:rKdV}.
	
	Lastly, let us prove uniqueness for $ r = 2 $ and $3 $. For $ r=2 $, we have $( 2 C - 1 ) \Psi_1 = 0$, and thus the operator $ C $ determines $ \Psi_1 $ uniquely. The other $ \Psi_i $ for $ i \geq 2 $ are determined by applying $ B $ to $ \Psi_i $ thanks to \cref{eqn:Baction}. Thus all the $ \Phi_i $ for $ i \geq 1 $ are uniquely determined. As for $ r = 3 $, we instead use $( -3 B C - A ) \Psi_1 = 0$, which determines $ \Psi_1 $ uniquely. Now applying $ C $ repeatedly to $  \Psi_1 $ determines all the other $ \Psi_i $ for $ i 
	\geq 2 $ due to \cref{eqn:Caction}, and hence all the $ \Phi_i $ for $ i \geq 1 $ are uniquely determined. Note that in both cases the equation determining $\Psi_1$ uniquely is the quantum curve.
\end{proof}

%–––––––––––––––––––––––––––––––––––––––––––%
\subsubsection{From \texorpdfstring{$ \mathfrak{W}_{1+\infty} $}{W-infinity} to \texorpdfstring{$ \mathcal{W}^{-r+1}(\mathfrak{gl}_r) $}{W-algebra}}
%–––––––––––––––––––––––––––––––––––––––––––%

The Kac--Schwarz operators obtained in \cref{thm:KS:operators} translate immediately to  $ \mathcal{W} $-constraints on the $ r $-BGW tau function, and we would like to compare them with the $ \mathcal{W} $-constraints characterising the descendant potential $ Z^{\Theta^r} $ of \cref{thm:Wconst}. The former $ \mathcal{W} $-constraints are obtained from a representation of the $ \mathfrak{W}_{1+\infty} $ algebra, while the latter are obtained from a representation of the $ \mathcal{W}^{-r+1}(\mathfrak{gl}_r) $ algebra. Thus, we need to understand the relation between the two.

The algebra $ \mathfrak{W}_{1+\infty} $, which is the central extension of the Lie algebra of diffeomorphisms of the circle, admits the structure of a vertex algebra with Virasoro field of central charge $ c=1 $ \cite{FKRW95}.  Moreover, \textit{loc.~cit.} proves the following isomorphisms as vertex algebras 
\begin{equation}
	\mathfrak{W}_{1+\infty}  \cong \mathcal{W}^0(\mathfrak{gl}_1) \cong \mathcal{S}_0(\mathfrak{gl}_1)\,.
\end{equation}
In general, however, we can consider different central extensions of $ \mathfrak{w}_{1+\infty} $ with correspondingly different central charges $ c \in \C$, which we denote by $ \mathfrak{W}_{1+\infty}^c $. When $ c $ is an integer $ r \geq 1 $, \cite{FKRW95} proves  the following isomorphism of vertex algebras 
\begin{equation}\label{eqn:Winfred}
	\mathfrak{W}^r_{1+\infty} \cong \mathcal{W}^{-r+1}(\mathfrak{gl}_r)\,,
\end{equation}
where we recall that $\mathcal{W}^{-r+1}(\mathfrak{gl}_r) $ also has central charge $ c = r $. 

Our main focus is not on the central extension $ \mathfrak{W}_{1+\infty} $ we considered so far, but rather on the central extension $ \mathfrak{W}^r_{1+\infty} $ for $ r \geq 2 $. We are interested in a specific representation of this central extension (see \cite{FKN92}). Consider the Lie subalgebra of $ \mathfrak{w}_{1+\infty} $ generated by the following elements
\begin{equation}
	m^r_{1+\infty}
	=
	\Set{
		e^{-S(z)}
		\left( x(z)^{\alpha} \left( \hbar^{1/2} \frac{\de}{\de x(z)} \right)^{\beta} \right)
		e^{S(z)}
		| \alpha, \beta \in \Z, \, \beta \ge 0
	},
\end{equation}
where we recall the notation $x(z) = \frac{z^r}{r}$ and $S(z) = - \hbar^{-1/2} \frac{z^{r-1}}{r-1} - \frac{1}{2} \log(z^{r-1})$. Then consider the Lie subalgebra of $ \mathfrak{W}_{1+\infty} $ generated by the associated elements using the map \labelcref{eqn:wtoW}. More precisely, define 
\begin{equation}
	M^r_{1+\infty} \coloneqq \langle W_a \,|\, a \in m^r_{1+\infty} \rangle\,.	
\end{equation}
Then $ M^r_{1+\infty} $ is a representation of $ \mathfrak{W}^r_{1+\infty} $, and using the isomorphism \labelcref{eqn:Winfred} we obtain an induced representation of $ \mathcal{W}^{-r+1}(\mathfrak{gl}_r) $. It is proved in \cite{FKN92} that this induced representation of $ \mathcal{W}^{-r+1}(\mathfrak{gl}_r) $ coincides with the one that we considered in \cref{sec:Wconst} using the twisted representation of the Heisenberg algebra $ \mathcal{S}_0(\mathfrak{gl}_r) $ in the context of higher Airy structures \cite{BBCCN24}. 

Let us describe the isomorphism \labelcref{eqn:Winfred} in more detail:
\begin{equation}
\begin{aligned}
	\mathfrak{W}^r_{1+\infty}
	& \xrightarrow{\;\, \cong \,\;}
	\mathcal{W}^{-r+1} (\mathfrak{gl}_r)\,, \\
	- z^{l+k} \bigl( \hbar^{1/2} \de_z \bigr)^l
	& \longmapsto
	\frac{1}{l+1} \Res_{z=0} \left( z^{k+l} \sum_{i=1}^r \norder{ \bigl( \chi^i(z) + \hbar^{1/2} \de_z \bigr)^l \chi^i(z)} \right).
\end{aligned}
\end{equation}
The fields appearing on the right-hand side of the above equation
\begin{equation}
	\widetilde{U}^{l+1}(z) \coloneqq \frac{1}{l+1} \sum_{i=1}^r  \norder{ \left(\chi^i(z) +  \hbar^{1/2}  \de_z\right)^l \chi^i(z)}\,, \qquad 0 \leq l \leq r-1\,,
\end{equation}
form a set of strong generators for the $ \mathcal{W}^{-r+1}(\mathfrak{gl}_r) $ vertex algebra. Let us compare them to the generating fields $ U^{l+1}(z) $ which are defined as elementary symmetric polynomials \labelcref{eqn:We}. For the first few generating fields we get 
\begin{equation}
\begin{aligned}
	U^1(z)
	& =
	\widetilde{U}^{1} (z) \, , \\
	r U^2(z)
	& =
	-\widetilde{U}^2(z) + \frac{1}{2} \hbar^{1/2} \partial \widetilde{U}^1(z) + \frac{1}{2} \norder{\widetilde{U}^1 \widetilde{U}^1} (z) \, , \\
	r^2 U^3(z)
	& = \widetilde{U}^{3} (z) - \hbar^{1/2} \partial \widetilde{U}^2 +\frac{1}{6} \hbar \partial^2 \widetilde{U}^1(z) - \norder{\widetilde{U}^1 \widetilde{U}^2} (z) + \frac{1}{6} \norder{\widetilde{U}^1 \widetilde{U}^1 \widetilde{U}^1} (z) + \frac{1}{2} \norder{\widetilde{U}^1  \hbar^{1/2} \partial \widetilde{U}^1}(z) \, .
\end{aligned}
\end{equation}

Using this, we can express the operators $ W^i_k $ of \cref{eqn:W:modes} in the representation $ M^r_{1+\infty} $ as $ \mathcal{W} $-algebra operators corresponding to Kac--Schwarz operators. 
\begin{equation}\label{eqn:W:as:KS}
\begin{aligned}
	- \frac{1}{r^k} H^1_k
	& =
	\widehat{W}_{ a^{k}} \, , \\
	\frac{1}{r^{k+1}} H^2_k
	& =
	\widehat{W}_{a^k ( b + \frac{k+1}{2} )}
	+ \frac{1}{2} \sum_{k_1+k_2 = k} \norder{ \widehat{W}_{ a^{k_1}} \widehat{W}_{a^{k_2}} } \, , \\
	- \frac{1}{r^{k+2}} H^3_k
	& =
	\widehat{W}_{a^{k+1} c + (k+1) a^k b + \frac{1}{6} (k+2)(k+1) a^k }
	+ \sum_{k_1+k_2 = k} \norder{\widehat{W}_{a^{k_1}} \widehat{W}_{a^{k_2}\left(b + \frac{k_2+1}{2}\right)} } \\
	& \hphantom{ = \widehat{W}_{a^{k+1} c + (k+1) a^k b + \frac{1}{6} (k+2)(k+1) a^k } }
	+ \frac{1}{6} \sum_{k_1+k_2+k_3=k} \norder{ \widehat{W}_{ a^{k_1} } \widehat{W}_{ a^{k_2} } \widehat{W}_{ a^{k_3} }} \, .
\end{aligned}
\end{equation}
As a sanity check, it is straightforward to see that the modes $ W_{a^k(b + \frac{k+1}{2} )}  $ generate a Virasoro subalgebra of central charge $ c = r $. We have also verified \cref{eqn:W:as:KS} by direct computations of the required $ \widehat{W} $ operators and comparison with the $ H^i_k $ \labelcref{eqn:H:cnstrnts}. 

\begin{remark}
	We note that the conjugation by the summand $ -\hbar^{-1/2} \frac{z^{r-1}}{r-1} $ of $ S(z) $ of the Kac--Schwarz operators corresponds precisely to the dilaton shift in the definition of the $ H^i_k $ in \cref{eqn:H:cnstrnts}.
\end{remark}

%–––––––––––––––––––––––––––––––––––––––––––%
\subsubsection{Equivalence between $ r $-BGW tau function and the descendant potential $ Z^{\Theta^r} $}
%–––––––––––––––––––––––––––––––––––––––––––%

We are ready to formulate our main conjecture, that can be viewed as a ``negative spin" version of the Witten $ r $-spin conjecture. The statement for $ r =2 $ was conjectured by Norbury in \cite{Nor23} and verified up to $ O(\hbar^7) $.

\begin{conjecture}\label{conj:rKdV}
	The descendant potential $ Z^{\Theta^r} $ of the $ \Theta^r $-class coincides with the $ r $-BGW tau function, and thus is a tau function for the $ r $-KdV integrable hierarchy.
\end{conjecture}

To recap, let us discuss the necessary steps to prove this conjecture. We know that $ Z^{\Theta^r} $ satisfies a set of $ \mc{W} $-constraints that characterise it uniquely. What we do not know yet is the existence of a $ r $-KdV tau function that also satisfies the exact same constraints. Our conjecture is that the $ r $-BGW tau function does satisfy these constraints. For $ r =3 $, \cite{Mah20} studied the Ward identities for the $ 3 $-BGW model, suggesting their equivalence to the string equation. 

\begin{proposition}\label{prop:stringequation}
	\Cref{conj:rKdV} is equivalent to showing that the string equation holds, i.e.,
	\begin{equation}\label{eqn:string}
		H^r_{-r+2} \, Z^{r\textup{-BGW}} = \mu \, Z^{r\textup{-BGW}} \, ,
	\end{equation}
	for some constant $ \mu \in \C$.
\end{proposition}

\begin{remark}
	We can view the above equation as specifying the initial condition for the time evolution of the integrable hierarchy, and in this sense can be considered as a string equation in analogy with the $ r $-Witten--Kontsevich tau function. This is a consequence of the form of the operator $ H^r_{-r+2} $. In the grading of Airy structures, we have $H^r_{-r+2} = \hbar \frac{\partial}{\partial t_1} + O(2)$, and thus $ H^r_{-r+2} $ controls the dependence on the time $ t_1 $ of the solution.
\end{remark}

\begin{proof}
	First, we note that $ Z^{r\textup{-BGW}} $ admits a genus expansion of the form \labelcref{eqn:genus:expansion}. Indeed, from the full asymptotic expansion of the basis vectors $ \Phi_i(z) $ derived in \cite[section~4.1]{AD25}, we deduce that taking $ \hbar \to 0 $ gives us the point $ z^0 \wedge z^1 \wedge z^2 \wedge \cdots $ on the Sato Grassmannian. This point is the origin of the Sato Grassmannian and the tau function at this point is identically $ 1 $. Taking into account the additional rescaling of the times in our definition of the boson-fermion correspondence by $ \hbar^{-1/2} $ in \cref{def:bfcor}, this implies that $ Z^{r\textup{-BGW}} $ admits a genus expansion:
	\[
		Z(\hbar;\bm{t})
		=
		\exp\left(
		\sum_{\substack{ g \geq 0, n \geq 1 \\ 2g-2+n > 0}}
		\frac{\hbar^{g-1}}{n!}
		\sum_{a_1,\ldots, a_n > 0}
		\Phi_{g,n}[a_1,\ldots, a_n] \,\,
		t_{a_1} \cdots t_{a_n} \right).
	\]
	Then, thanks to \cref{prop:tau}, we know that showing
	\begin{equation}
		H^1_k \, Z^{r\textup{-BGW}} = \nu_k \, Z^{r\textup{-BGW}} \, , 
		\qquad \text{ for all $k \geq 1$\,, \, \, and } \qquad
		H^{r}_{-r+2} \, Z^{r\textup{-BGW}} = \mu \, Z^{r\textup{-BGW}} \, ,
	\end{equation}
	for some constants $\nu_k, \mu \in \C$, immediately implies that $ Z^{r\textup{-BGW}} $ satisfies all the required $ \mathcal{W} $-constraints. 
	
	The first condition that $ H^1_k $ acts by a constant follows directly, as $ H^1_k = -r^k \widehat W_{a^k} $ from \cref{eqn:W:as:KS} and $ a^k $, for $ k \geq 1 $, is a Kac--Schwarz operator.
\end{proof}

For $ r = 2 $, the full Virasoro constraints including the string equation were  proved by \cite{GN92, Ale18}. Alexandrov's proof in \cite{Ale18} uses the formalism of Kac--Schwarz operators and thus the following proof is identical to his. While we are unable to prove this string equation for any $ r > 3 $\footnote{
	While \cite{MMS96} sketches an argument to prove the string equation for any $ r $, there is a gap in the proof there that we were unable to resolve. We thank A.~Mironov and A.~Morozov for discussions about their work.
}, we prove it in this paper for $ r = 3 $ to obtain one of our main results.

\begin{theorem}\label{thm:r23}
	When $ r = 2 $ or $ 3 $, the descendant potential of the $ \Theta^r $-class coincides with the $ r $-BGW tau function of the $ r $-KdV integrable hierarchy.
\end{theorem}

\begin{proof}
	\Cref{prop:stringequation} shows that we only need to prove the string equation.

	For $ r = 2 $, the string equation is given by $H^2_{0} \, Z^{2\textup{-BGW}} = \mu \, Z^{2\textup{-BGW}}$. In \cref{thm:KS:operators}, we showed that $ a $ and $ b $ are Kac--Schwarz operators. Thus the associated $ \mathcal{W} $-algebra operators act by a constant on the tau function. By \cref{eqn:W:as:KS}, we see that $H^2_{0} $ also acts by a constant.

	For $ r = 3 $, the string equation is given by $H^3_{-1} \, Z^{3\textup{-BGW}} = \mu \, Z^{3\textup{-BGW}}$. Again from \cref{thm:KS:operators} we know that $a$, $b$ and $c$ are Kac--Schwarz operators. Thus, the form of $ H^3_{-1} $ in \cref{eqn:W:as:KS} immediately gives the string equation.
\end{proof}

For $ r=2 $, the above theorem proves Norbury's conjecture \cite{Nor23}.

We leave the proof of the string equation in general, equivalently the $ r $-KdV \cref{conj:rKdV}, to future work. Instead of finding a Kac--Schwarz operator that corresponds to the operator $ H^r_{-r+2} $, an alternative approach to proving the string equation is to use the Ward identities for the $ r $-BGW matrix model. 

%–––––––––––––––––––––––––––––––––––––––––––%
\appendix
%–––––––––––––––––––––––––––––––––––––––––––%
\section{Integrals of Airy functions and Scorer functions}
%–––––––––––––––––––––––––––––––––––––––––––%

In this appendix, we consider the coefficients $ P_m(r,-1) $ and $H_k(r,a)$ that we encountered in the calculation of the translation in \cref{sec:TR}. The polynomials $ P_m(r,-1) $ appear in the asymptotic expansion of an integral of the hyper-Airy function \labelcref{eqn:hyperAirya}, while the coefficients for the case $ r = 3 $ appear in the asymptotic expansion of the Scorer functions. The following lemma generalises to higher $r$ the relations
\begin{equation}
\begin{split}
	\Gi(t) & = \Bi(t) \int_{t}^{+\infty} \Ai(s) ds + \Ai(t) \int_{0}^{t} \Bi(s) ds \, , \\
	\Hi(t) & = \Bi(t) \int_{-\infty}^{t} \Ai(s) ds - \Ai(t) \int_{-\infty}^{t} \Bi(s) ds \, , \\
\end{split}
\end{equation}
between the Airy functions $\Ai$ and $\Bi$, and the Scorer functions $\Gi$ and $\Hi$. We do not study the above relation directly, but merely the one appearing by taking the asymptotic expansion on both sides.

\begin{lemma}\label{lem:P-1}
	The following relation holds:
	\begin{equation}\label{eqn:integral:Airy:Scorer}
		P_{m}(r,-1)
		=
		\sum_{\substack{a,b \ge 0 \\ a + b = r-2}} \; \sum_{\substack{i,\ell \ge 0 \\ i + (r-1)\ell = m}}
			(1-r)^{b + (r-1)\ell} \frac{(r\ell + b)!}{\ell! r^{\ell}}
			P_{i-b}(r,a)\,,
	\end{equation}
	with the convention that $P_j(r,a) = 0$ for $j < 0$.
\end{lemma}

\begin{proof}
	Denote the right hand side of \cref{eqn:integral:Airy:Scorer} by $Q_m(r,-1)$. In order to prove $P_m(r,-1) = Q_m(r,-1)$ for all $m$, let us proceed by induction. The base case $P_0(r,-1) = Q_0(r,-1) = 1$ is easy to verify. Assume now $P_{m-1}(r,-1) = Q_{m-1}(r,-1)$. Recall that the coefficients $P_j(r,a)$ are defined recursively by
	\begin{equation*}
	\begin{cases}
		P_j(r,a) - P_j(r,a-1) = r \left( j - \frac{1}{2} - \frac{a}{r} \right) P_{j-1}(r,a-1)\,,
		\qquad
		\text{for } a = 0,\dots,r-2\,, \\
		P_j(r,0) = P_j(r,r-1)\,,
	\end{cases}
	\end{equation*}
	with initial condition $P_0 = 1$. We can now expand the definition of $Q_m(r,-1)$ by using the recursion relations:
	\[
	\begin{split}
		Q_m(r,-1)
		& =
		\sum_{\substack{a,b \ge 0 \\ a + b = r-2}}
		\overbrace{
		\sum_{\substack{i,\ell \ge 0 \\ i + (r-1)\ell = m}}
			(1-r)^{b + (r-1)\ell} \frac{(r\ell + b)!}{\ell! r^{\ell}}
			P_{i-b}(r,a+1)
		}^{\eqqcolon S_{a,b}}
			\\
		& \qquad
		-
		\sum_{\substack{a,b \ge 0 \\ a + b = r-2}}
		\underbrace{
		\sum_{\substack{i,\ell \ge 0 \\ i + (r-1)\ell = m}}
			(1-r)^{b + (r-1)\ell} \frac{(r\ell + b)!}{\ell! r^{\ell}}
			r \left( i - b - \tfrac{1}{2} - \tfrac{a+1}{r} \right) P_{i-b-1}(r,a)
		}_{\eqqcolon T_{a,b}} .
	\end{split}
	\]
	We can now combine the terms $S_{a-1,b+1}$ with $T_{a,b}$, with the convention that the indices are considered modulo $(r-2)$. We start by considering the extreme case $(a,b) = (0,r-2)$. We have
	\[
		S_{r-2,0}
		=
		\sum_{\substack{i,\ell \ge 0 \\ i + (r-1)\ell = m}}
			(1-r)^{(r-1)\ell} \frac{(r\ell)!}{\ell! r^{\ell}}
			P_{i}(r,0)\,,
	\]
	and on the other hand
	\[
	\begin{split}
		T_{0,r-2}
		& =
		\sum_{\substack{i,\ell \ge 0 \\ i + (r-1)\ell = m}}
			(1-r)^{r-2 + (r-1)\ell} \frac{(r\ell + r-2)!}{\ell! r^{\ell}}
			r \left( i - r + \tfrac{3}{2} - \tfrac{1}{r} \right) P_{i-r+1}(r,0) \\
		& =
		r(m - \tfrac{1}{2})
		\sum_{\substack{i,\ell \ge 0 \\ i + (r-1)\ell = m}}
			(1-r)^{r-2 + (r-1)\ell} \frac{(r\ell + r-2)!}{\ell! r^{\ell}}
			P_{i-r+1}(r,0) \\
		& \qquad
		+ \sum_{\substack{i,\ell \ge 0 \\ i + (r-1)\ell = m}}
			(1-r)^{(r-1)(\ell+1)} \frac{(r\ell + r-2)!}{\ell! r^{\ell}}
			\bigl( r(\ell + 1) - 1 \bigr) P_{i-r+1}(r,0)\,.
	\end{split}
	\]
	By performing the shift $(i,\ell) \mapsto (i + (r-1),\ell-1)$ in the second sum, we obtain $S_{r-2,0}$ except for the term corresponding to $(i,\ell) = (m,0)$, that is $P_m(r,0)$:
	\[
	\begin{split}
		T_{0,r-2}
		& =
		r(m - \tfrac{1}{2})
		\sum_{\substack{i,\ell \ge 0 \\ i + (r-1)\ell = m}}
			(1-r)^{r-2 + (r-1)\ell} \frac{(r\ell + r-2)!}{\ell! r^{\ell}}
			P_{i-r+1}(r,0)
		+
		S_{r-2,0} - P_m(r,0)\,.
	\end{split}
	\]
	All together, we find
	\[
		S_{r-2,0} - T_{0,r-2}
		=
		P_m(r,0)
		- r(m - \tfrac{1}{2})
		\sum_{\substack{i,\ell \ge 0 \\ i + (r-1)\ell = m}}
			(1-r)^{r-2 + (r-1)\ell} \frac{(r\ell + r-2)!}{\ell! r^{\ell}}
			P_{i-r+1}(r,0)\, .
	\]
	The computation for $(a,b) \neq (0,r-2)$ is simpler, as one can simplify the expression to
	\[
		S_{a-1,b+1} - T_{a,b}
		=
		- r(m - \tfrac{1}{2})
		\sum_{\substack{i,\ell \ge 0 \\ i + (r-1)\ell = m}}
			(1-r)^{b + (r-1)\ell} \frac{(r\ell + b)!}{\ell! r^{\ell}}
			P_{i-b+1}(r,a)\, .
	\]
	All together, we find
	\[
		Q_m(r,-1)
		=
		P_m(r,0)
		- r(m - \tfrac{1}{2})
		\sum_{\substack{a,b \ge 0 \\ a+b=r-2}} \sum_{\substack{i,\ell \ge 0 \\ i + (r-1)\ell = m}}
			(1-r)^{b + (r-1)\ell} \frac{(r\ell + b)!}{\ell! r^{\ell}}
			P_{i-b+1}(r,a)\,.
	\]
	To conclude, by shifting $i \mapsto i-1$, we recognise $Q_{m-1}(r,-1)$ which equals $P_{m-1}(r,-1)$ by the induction hypothesis:
	\[
		Q_m(r,-1)
		=
		P_m(r,0)
		- r(m - \tfrac{1}{2}) P_{m-1}(r,-1)
		=
		P_m(r,-1)\,,
	\]
	where in the last equality we used the recursion relation defining $P_m(r,-1)$. 
\end{proof}

\printbibliography

\end{document}